\documentclass[12pt]{article}

\global\arraycolsep=1pt
\usepackage{amsmath,amssymb,amsthm, graphicx,latexsym,epic, color}

\usepackage{ytableau}
\ytableausetup{aligntableaux=top,mathmode,boxsize=0.8em}
\def\young(#1){\ytableaushort{#1}}
\def\yng(#1){\tiny {\ydiagram{#1}}}
\newcommand{\beq}{\begin{equation}}
\newcommand{\eeq}{\end{equation}}
\newcommand{\beqa}{\begin{eqnarray}}
\newcommand{\eeqa}{\end{eqnarray}}
\newcommand{\CR}{\nonumber \\}

\newcommand{\barsigma}{\overline{\sigma}}
\numberwithin{equation}{section}

 \newtheorem{dfn}{Definition}[section]
 
 \newtheorem{prp}[dfn]{Proposition}
 \newtheorem{lem}[dfn]{Lemma}
 \newtheorem{cor}[dfn]{Corollary}
 
 \newtheorem{rmk}[dfn]{Remark}
  \newtheorem{con}[dfn]{Conjecture}

\newtheorem*{ack}{Acknowledgements}

\begin{document}
%%%%%%%%%%%%%%%%%%%%%%%%%%%%%%%%%%%%%%%%%%%%%%%%%%%%%%%%%%%%%%%%%%%%%%%%%%%%%%%%%%%%%%%%%%%%%%%%%%%%%%%%%%%%%%%%%%%%%%%%%%

%%%%%%%%%%%%%%%%%%%%%%%%%%%%%%%%%%%%%%%%%%%%%%%%%%%%%%%%%%%%%%%%%%%%%

%%%%%%%%%%%%%%%%%%%%%%%%%%%%%%%%%%%%%%%%%%%%%%%%%%%%%%%%%%%%%%%%%%%%%
\begin{titlepage} 
%
%%
%\begin{flushright}
%version one
%\end{flushright}
%%
%
\begin{center}
\Large{\bf Shifted quantum toroidal algebra of type $\mathfrak{gl}_{1|1}$
and the Pieri rule of the super Macdonald polynomials}
\end{center} 

\bigskip
\bigskip

\begin{center}
\large 
Hiroaki Kanno$^{a,}$\footnote{kanno@math.nagoya-u.ac.jp}, 
Ryo Ohkawa$^{b,}$\footnote{ohkawa@kurims.kyoto-u.ac.jp}, and
Jun'ichi Shiraishi$^{c,}$\footnote{shiraish@ms.u-tokyo.ac.jp} \\

\bigskip

$^a${\small {\it Graduate School of Mathematics, Nagoya University,
Nagoya 464-8602, Japan}}\\
$^b${\small {\it Research Institute for Mathematical Sciences, Kyoto University, Kyoto
606-8502, Japan}} \\
$^c${\small {\it Graduate School of Mathematical Sciences, University of Tokyo, Komaba, Tokyo 153-8914, Japan}}
\end{center}
\bigskip

\begin{abstract}
The super Macdonald polynomials indexed by the super partitions form 
a basis of the level zero super Fock module (combinatorial representation) of the quantum toroidal algebra 
$\mathcal{U}_{q,t}(\widehat{\widehat{\mathfrak{gl}}}_{1|1})$. 
The action of the super charges of $\mathcal{U}_{q,t}(\widehat{\widehat{\mathfrak{gl}}}_{1|1})$
implies the Pieri rule of the super Macdonald polynomials. 
We can express the Pieri rule in terms of differential operators in the power sums $p_k$ and the fermionic power sums $\pi_k$,
which leads to the operators on the Fock space of a free boson and a free fermion. 
From the Pieri rule we compute the supersymmetric Hamiltonians given by the anti-commutator 
of the super charges and recover the results previously obtained in the literature. 
It is remarkable that we have to deal with a shifted quantum toroidal algebra.
\end{abstract}
\end{titlepage}

%%%%%%%%%%%%%%%%%%%%%%%%%%%%%%%%%%%%%%%%%%%%%%%%%%%%%%%

\setcounter{tocdepth}{1}
\tableofcontents

\setcounter{footnote}{0}

%%%%%%%%%%%%%%%%%%%%%%%%%%%%%%%%%%%%%%%%%%%%%%%%%%%%%%%%%

\section{Introduction and Summary}

The super Macdonald polynomials $\mathcal{M}_{\Lambda}(x, \theta;q,t)$ generalize the Macdonald polynomials
to the super space with additional Grassmann coordinates $\theta_i$ \cite{Blondeau-Fournier:2011sft}.\footnote{
There is another kind of the super Macdonald polynomials associated with the root system of 
Lie super algebra $A(m\vert n)$ \cite{Sergeev-Veselov1,Sergeev-Veselov2}
 (see also \cite{Cheewaphutthisakun:2025zoc,Cheewaphutthisakun:2025ovm}).}
They are invariant under {\it simultaneous} permutations of the commuting and the anti-commuting coordinates;
\begin{equation}
\mathcal{M}_{\Lambda}(x_i, \theta_i ;q,t) =  \mathcal{M}_{\Lambda}(x_{\sigma(i)}, \theta_{\sigma(i)} ;q,t), \quad \sigma \in S_N.
\end{equation}
One may regard $(x_i, \theta_i)$ as the coordinates of $N$-particles on the super space. 
The super Macdonald polynomials $\mathcal{M}_{\Lambda}(x, \theta;q,t)$ are indexed by the set of super partitions.
A super partition is a non-increasing sequence of non-negative elements in $\mathbb{Z}/2$;
\begin{equation}
\Lambda_1 \geq \Lambda_2 \geq \ldots \geq \Lambda_i \geq  \ldots \geq \Lambda_{\ell(\Lambda)} >0 , \qquad \Lambda_{\ell(\Lambda)+1} =0,
\end{equation}
where $\ell(\Lambda)$ is the number of non-zero components in $\Lambda$. %$\vert \Lambda \vert = \sum_{i=1}^{\ell(\Lambda)} \Lambda_i$ is called level of $
We define integral elements in $\mathbb{Z}/2$ are even and non-integral elements are odd, which gives a $\mathbb{Z}_2$ grading.
We require that if $\Lambda_k$ is odd, the inequality is strict $\Lambda_{k-1} > \Lambda_k > \Lambda_{k+1}$.

The Macdonald polynomials $P_\lambda(x;q,t)$ labelled by the partitions $\lambda$ are closely related to
the representation theory of the $\mathfrak{gl}_1$ quantum toroidal algebra \cite{FT,FFJMM,Awata:2011ce}. 
In a similar manner we can investigate the properties of the super Macdonald polynomials 
in terms of the quantum toroidal algebra $\mathcal{U}_{q,t}(\widehat{\widehat{\mathfrak{gl}}}_{1|1})$
of type $\mathfrak{gl}_{1|1}$ \cite{Galakhov:2024cry,Galakhov:2024zqn}, 
which has the super (odd) currents $E_i(z), F_i(z)~i=1,2$. However, there is also
a crucial difference from the case of the Macdonald polynomials. As we will explain below the toroidal algebra is
shifted, which gives us several technical challenges in the representation theory.
In \cite{Kanno:2025ifd} we investigated the relation between the super Macdonald polynomials and 
the Nekrasov-type factors arising from the moduli space of instantons (framed perverse coherent sheaves) on the blow-up of
$\mathbb{C}^2$ or its compactified version $\mathbb{P}^2$.\footnote{See \cite{Filoche:2026xix}, for a recent development 
on the relation of the wall-crossing formula of instanton counting on the blow-up and the combinatorics of the super partitions.}
It was motivated by the fact that the super Macdonald polynomials
form a basis of the level zero Fock representation of $\mathcal{U}_{q,t}(\widehat{\widehat{\mathfrak{gl}}}_{1|1})$,
which in turn arises as the BPS algebra of $D$-branes on the resolved conifold. It is expected that the BPS algebra
acts on the equivariant $K$-theory group of the moduli space of BPS states. 
The torus fixed points on the moduli space are labelled by the super partitions.
Assuming the equivariant localization, we can introduce the Nekrasov-type factor for a pair of fixed points. 
Then one of the main results in  \cite{Kanno:2025ifd} is that the change of the Nekrasov-type factor 
under the addition or removal of a half-box in the super Young diagram is related the the Pieri rule 
of the super Macdonald polynomials which is obtained by the representation theory of the quantum toroidal algebra.

The Macdonald polynomials $P_\lambda(x;q,t)$ are simultaneous eigenstates of the
Ruijsenaars-Macdonald Hamiltonians \cite{Ruijs,Mac1988,MacD}
\begin{equation}\label{RM-Hamiltonian}
\hat H_r(x):=t^{r(r-1)/2}\sum_{|I|=r}\prod_{i\in I;j\notin I}\frac{tx_i-x_j}{x_i-x_j}\prod_{i\in I}\hat T_{q,x_i}, \qquad 1 \leq r \leq N,
\end{equation}
where 
\begin{equation}
\hat T_{q,x_i}f(x_1,\ldots,x_i,\ldots,x_N):=f(x_1,\ldots,qx_i,\ldots,x_N)
\end{equation}
is the $q$-difference operator on $x=(x_1, \ldots, x_N)$. The eigenvalues are
\begin{equation}
\hat D^x(u) P_{\lambda}(x;q,t) =  P_{\lambda}(x;q,t) \prod_{i=1}^N (1- u t^{N-i} q^{\lambda_i}),
\end{equation} 
where we have introduced the generating function of the mutually commuting Hamiltonians;
\begin{equation}
\hat D^x(u):=\sum_{r=0}^N (-u)^r\hat H_r(x).
\end{equation}
Since the Macdonald polynomials are invariant under the involution of the parameters $\iota : (q,t) \to (q^{-1}, t^{-1})$,
namely $P_\lambda(x;q,t) = P_\lambda(x;q^{-1},t^{-1})$, they are also eigenfunctions of the $(q,t)$ inverted Hamiltonians $\iota(\hat H_r)$. 
Consequently $\iota(\hat H_r)$ are mutually commuting with the Hamiltonians \eqref{RM-Hamiltonian}.
In the level zero Fock representation of the quantum toroidal algebra 
$\mathcal{U}_{q,t}(\widehat{\widehat{\mathfrak{gl}}}_{1|1})$, the original Hamiltonians $\hat H_r$ and 
the $(q,t)$ inverted Hamiltonians $\iota(\hat H_r)$ come from the positive and the negative Cartan currents $K^\pm(z)$, respectively. 

The supersymmetric Hamiltonians for the super Macdonald polynomials have been already 
worked out in \cite{Blondeau-Fournier:2012exj,Alarie-Vezina:2019ohz} and \cite{Galakhov:2025phf}. 
In this paper we show that the same Hamiltonians are obtained from the Pieri rule for the super Macdonald polynomials, 
which follows from the level zero Fock representation of $\mathcal{U}_{q,t}(\widehat{\widehat{\mathfrak{gl}}}_{1|1})$ 
with the generating currents (see section \ref{section:shifted-algebra});
\begin{align}
E_i(z) &= \sum_{k \in \mathbb{Z}} E_{i,k} z^{-k}, \qquad F_i(z) = \sum_{k \in \mathbb{Z}} F_{i,k} z^{-k}, \\
K_i^{\pm}(z) &= \sum_{r \geq 0} K_{i, \pm r}^{\pm} z^{\mp r}
= K_{i,0}^{\pm} \exp 
\left( \pm \sum_{r=1}^\infty H_{i, \pm r}z^{\mp r} \right).
\end{align}
A significant aspect of our derivation is that the algebra we employ is a shifted quantum toroidal algebra. 
Namely in the commutation relation
\begin{align}
\label{EF-com}
\left[ E_i(z), F_j(w) \right]_{+} &= \delta_{ij} \delta \left( \frac{w}{z} \right)  
\left( z^{r_i} K_i^{+}(z) - K_i^{-}(z) \right), \\
\left[ E_i(z), F_{j,k}\right]_{+} &= \delta_{ij} z^k 
\left( z^{r_i} K_i^{+}(z) - K_i^{-}(z) \right), \qquad (r_1=-1,~r_2=1),
\end{align}
we have shift parameters $r_i$ which affect the mode expansion of the Cartan currents $K_i^{+}(z)$.\footnote{There is
no shift for the Cartan currents $K_i^{-}(z)$.}
In the level zero representation $K_i^{\pm}(z)$ are mutually commuting and the super Macdonald polynomials 
$\mathcal{M}_{\Lambda}(x, \theta;q,t)$ are simultaneous eigenstates of the Hamiltonians $H_{i, \pm 1}$
(see Proposition \ref{eigenvalues} for explicit formulas of the eigenvalues).
The shifted anti-commutation relation implies that the first Hamiltonians $H_{i, \pm 1}$ are expressed 
as the anti-commutator of the appropriate modes of the currents $E(z)$ and $F(z)$;
\begin{align}
\label{H2-1-anti-com}
\left[ E_{2,0}, F_{2,-1} \right]_{+} &= K_{2,0}^{+} + K_{2,0}^{-}H_{2,-1}
= 1 -(t/q)^{\frac{1}{2}}u H_{2,-1}, \\
\label{H1-1-anti-com}
\left[ E_{1,-1}, F_{1,0} \right]_{+} &= K^{-}_{1,0} H_{1,-1} 
= -u^{-1} H_{1,-1}, \\
\label{H21-anti-com}
\left[ E_{2, +1}, F_{2,-1} \right]_{+} &= K_{2,0}^{+}H_{2,+1} - K_{2,0}^{-}
= (t/q)^{\frac{1}{2}}u (1+ u^{-1}(q/t)^{\frac{1}{2}}H_{2,+1}), \\
\label{H11-anti-com}
\left[ E_{1,0}, F_{1,2} \right]_{+} &= K_{1,0}^{+} H_{1,+1} = H_{1,+1}.
\end{align}

The generating function of the numbers of super partitions with a fixed $|\Lambda|$;
\begin{equation}
\sum_{\Lambda} q^{2|\Lambda|} =  \prod_{k=1}^\infty \frac{1+ q^{2k-1}}{1-q^{2k}},
\qquad |\Lambda| := \displaystyle{\sum_{i=1}^{\ell(\Lambda)}} \Lambda_i \in \mathbb{Z}/2,
\end{equation}
agrees with the character of the tensor product of the Fock space $\mathcal{F}_B$ 
of a free boson $a_n~(n \in \mathbb{Z})$ with the commutation relation
$\left[ a_n, a_m \right] = \delta_{n+m,0}$
and the Fock space $\mathcal{F}_F$ of a free (NS) fermions $\psi_r~(r \in \mathbb{Z} + \frac{1}{2})$ 
with the anti-commutation relation $\left\{ \psi_{r}, \psi_{s} \right\} = \delta_{r+s, 0}$.
Hence, defining the vacuum by $a_{n}\vert \varnothing \rangle= \psi_r \vert \varnothing \rangle=0~(n,r>0)$,
we have a one to one correspondence between super partitions and the states in the Fock space $\mathcal{F}_B \otimes \mathcal{F}_F$.
Explicitly, by separating a super partition $\Lambda$ to the odd part $\lambda^{\mathsf{a}}$
and the even part $\lambda^{\mathsf{s}}$;
\begin{equation}
\Lambda = (\lambda^{\mathsf{a}}; \lambda^{\mathsf{s}}) 
= (\lambda_1, \ldots, \lambda_m; \lambda_{m+1}, \ldots, \lambda_{m+\ell}),
\end{equation}
with $\lambda_1 > \lambda_2 \cdots > \lambda_m \geq 0,~
\lambda_{m+1} \geq \lambda_{m+2} \geq \cdots \geq \lambda_{m+\ell} > 0$,
we define
\begin{equation}
\mathcal{F}_B \otimes \mathcal{F}_F \ni \vert \Lambda \rangle := a_{-\lambda_{m+1}} \cdots a_{-\lambda_{m+\ell}}  \cdot
\psi_{-\lambda_1} \cdots  \psi_{-\lambda_{m}} \vert \varnothing \rangle. 
\end{equation}

As generators of the algebra of the supersymmetric polynomials we employ 
the bosonic and the fermionic power sum polynomials defined by
\begin{equation}
p_k := \sum_{i=1}^N x_i^k, \qquad \pi_k := \sum_{i=1}^N \theta_i x_i^{k-1}.
\end{equation}
In this paper we take the projective limit of the graded ring of the supersymmetric polynomials 
with $2N$ variables. Then we can assume that the generators $p_k$ and $\pi_k$ are all independent
and define a linear isomorphism between the space of supersymmetric polynomials and the Fock space 
$\mathcal{F}_B \otimes \mathcal{F}_F$ given by 
\begin{equation}
\vert \Lambda \rangle \longleftrightarrow 
p_{\Lambda}:= \pi_{\lambda_1+\frac{1}{2}} \cdots \pi_{\lambda_m +\frac{1}{2}} p_{\lambda_{m+1}} \cdots p_{\lambda_{m+\ell}}.
\end{equation}
Note that $p_{\Lambda}$, which are labelled by the set of super partitions, give a basis of the space of supersymmetric polynomials. 
The super Macdonald polynomials $\mathcal{M}_{\Lambda}(x, \theta;q,t)$ give another basis
which simultaneously diagonalizes the Cartan modes of the quantum toroidal algebra 
$\mathcal{U}_{q,t}(\widehat{\widehat{\mathfrak{gl}}}_{1|1})$. 
The Pieri rule of the super Macdonald polynomials follows from the action of the creation and the annihilation currents 
$E_i(z)$ and $F_i(z)$ of $\mathcal{U}_{q,t}(\widehat{\widehat{\mathfrak{gl}}}_{1|1})$. 
By expanding $\mathcal{M}_{\Lambda}(x, \theta;q,t)$ in terms of $p_{\Lambda}$,\footnote{See Appendix \ref{App:Lower-Mac}
for explicit examples.} we can express the Pieri rule by linear differential operators in $p_k$ and $\pi_k$.
It turns out that among the modes of the super currents $E_i(z)$ and $F_i(z)$ 
there is a particular mode that takes simple form;
\begin{con}\label{Pieri-con}
\begin{equation}
E_{1,0} = \pi_1, \qquad
E_{2,0} = \sum_{k=1}^\infty c_k[p] \frac{\partial}{\partial \pi_k},
\end{equation}
\begin{equation}
F_{1,+1} = \frac{\partial}{\partial \pi_1}, \qquad
F_{2,-1} = -\sum_{k=1}^\infty \pi_k \cdot \widetilde{c}_k \left[ \partial/ \partial p \right],
\end{equation}
where the generating functions for the coefficients are
\begin{equation}\label{c-gen}
\sum_{k=0}^\infty c_k[p] z^k = \exp \left(\sum_{r=1}^\infty \frac{1-t^r}{r} p_r z^r \right)
= \prod_{i=1}^\infty \frac{1-tu x_i}{1-ux_i},
\end{equation}
and
\begin{equation}\label{tc-gen}
\sum_{k=0}^\infty \widetilde{c}_k \left[ \partial/ \partial p \right] z^{-k}
= \exp \left( \sum_{n=1}^\infty (q^{-n}-1) \frac{\partial}{\partial p_n} z^{-n} \right).
\end{equation}
\end{con}
For the super currents $F_i(z)$ the particular mode is different from the zero mode,
which is related to the fact that the toroidal algebra is shifted. 
Conjecture \ref{Pieri-con} is the starting point of the computations in this paper. 
For example, from Conjecture \ref{Pieri-con} we can derive
\begin{equation}
 E_{1, -1} = u^{-1}\sum_{k=1}^\infty \pi_k \cdot \widetilde{c}_{k-1} \left[\partial/\partial p \right], \qquad
F_{1, 0} = u^{-1} \sum_{k=1}^\infty c_{k-1}[p] \frac{\partial}{\partial \pi_k}.
\end{equation}
Combining these formulas and \eqref{H2-1-anti-com}--\eqref{H1-1-anti-com}, we obtain the negative mode Hamiltonians 
\begin{equation}
\label{H2-bilinear}
1 -(t/q)^{\frac{1}{2}}u H_{2,-1}
= - \sum_{k=1}^\infty
\left( c_k[p] \widetilde{c}_k [\partial/\partial p]  +
 \sum_{k=1}^\infty \sum_{\ell=1}^\infty \widetilde{C}_{k, \ell}[p, \partial/ \partial p]
\pi_k \frac{\partial}{\partial \pi_\ell} \right),
\end{equation}
and 
\begin{equation}
\label{H1-bilinear}
 - u H_{1,-1} =  1 + \sum_{k=1}^\infty c_k[p] \widetilde{c}_k[\partial/\partial p] + 
\sum_{\ell=1}^\infty \sum_{k=1}^\infty \widetilde{C}_{k,\ell}[p, \partial/ \partial p]
 \pi_{k+1} \frac{\partial}{\partial \pi_{\ell+1}}.
\end{equation}
The coefficients of the fermion bi-linear terms are
\begin{align}
\widetilde{C}_{k, \ell}[p, \partial/ \partial p] &:= \Big[\widetilde{c}_k 
\left[ \partial/ \partial p \right] , c_\ell [p] \Big] \CR
 &=  (1-t)(q^{-1}-1) \sum_{n=\ell-k}^{\ell-1} [[\ell -n ]]_{(t,q^{-1})} 
 \cdot c_n[p] \cdot \widetilde{c}_{k-\ell+n} \left[ \partial/ \partial p \right] \CR
 &= (1-t)(q^{-1}-1) \sum_{m=0}^{k-1} [[ k-m ]]_{(t,q^{-1})} 
 \cdot c_{m+\ell-k}[p] \cdot \widetilde{c}_{m} \left[ \partial/ \partial p \right],
\end{align}
where $[[ n ]]_{(t,q^{-1})}$ is defined below. Note that $c_n[p]=0$ for $n \leq 0$. 
The bosonic parts of $H_{2,-1}$ and $H_{1,-1}$ are the same and agree with the $(q,t)$ inverted Hamiltonian $\iota(\hat{H}_1)$, 
where we substitute $(q,t) \to (q^{-1}, t^{-1})$ for the first Hamiltonian of \eqref{RM-Hamiltonian}. 
This gives a consistency check of Conjecture \ref{Pieri-con}. 

Though the eigenvalues of $H_{i, -1}$ and $H_{i,+1}$ are related by the involution of the parameters 
$\iota: (q,t) \to (q^{-1}, t^{-1})$ (see Proposition \ref{eigenvalues}), this does not mean
 $H_{i, -1}$ and $H_{1,+1}$ are related by the same involution. This is due to the fact that
 the super Macdonald polynomials are not invariant under the involution $\iota$
\begin{equation}
 \mathcal{M}_\Lambda(x, \theta;q^{-1}, t^{-1}) \neq  \mathcal{M}_\Lambda(x, \theta;q, t).
\end{equation}
This is in sharp contrast to the case of the Macdonald polynomials $P_\lambda(x; q,t)$ which are invariant 
under the involution $\iota$ and 
the original Hamiltonians $\hat{H}_r$ and the $(q,t)$-inverted $\iota(\hat{H}_s)$ are mutually commuting.
On the other hand, though the bosonic parts of $H_{i, -1}$ and $H_{i,+1}$ are in fact
related by the involution $\iota$, the fermionic parts are not. 
If we denote the $(q,t)$ inverted form of $H_{i,-1}$ by $H_{i,-1}^\vee$,
we do not have the simple relation $H_{i, +1} = H_{i, -1}^\vee$ any more
and $H_{i,-1}$ and $H_{i,-1}^\vee$ are not mutually commuting,
while  $H_{i,-1}$ and $H_{i,+1}$ should mutually commute.

Consequently, the computation of the positive mode Hamiltonians $H_{1,1}$ and $H_{2,1}$ gets more involved,
which may be related to the fact that the shift parameter $r_i$ affects $K_i^{+}(z)$, but not $K_i^{-}(z)$. 
To evaluate $H_{1,1}$ and $H_{2,1}$ by the anti-commutation relations \eqref{H21-anti-com} and  \eqref{H11-anti-com}, 
we need a formula for $E_{2,1}$ and $F_{1,2}$.
Motivated by the formula for $H_{1,1}$ and $H_{2,1}$ in \cite{Galakhov:2025phf}, we propose
the following integral representation of $E_{2,1}$ and $F_{1,2}$;
\begin{con}\label{Pieri-con2}
\begin{align}
u^{-1} (q/t)^{\frac{1}{2}} E_{2,1} &=
- \sum_{k=1}^\infty \oint \frac{dw}{w} w^{-2k} V_B^{(-)}(w) V_B^{(+)}(w)  
\langle \varnothing \vert \widetilde{V}_F^{(-)}(w) \widetilde{V}_F^{(+)}(w) \vert \varnothing\rangle_{F} q^{k} \frac{\partial}{\partial \pi_k},
\\
u^{-1} F_{1,2} &=
\sum_{k=1}^\infty\oint \frac{dw}{w}  w^{-2k+2} V_B^{(-)}(w) V_B^{(+)}(w) 
\langle \varnothing \vert \widetilde{V}_F^{(-)}(w) \widetilde{V}_F^{(+)}(w) \vert \varnothing\rangle_{F} q^{k-1} \frac{\partial}{\partial \pi_k}.
\end{align}
\end{con}
For the definitions of the vertex operators $V_B^{(\pm)}(w)$ and $\widetilde{V}_F^{(\pm)}(w)$, 
see section \ref{section;integral-formula}. The meaning of the vacuum expectation value is
also explained there. Though only $\frac{\partial}{\partial \pi_k}$ appears explicitly in the above formula,
the vacuum expectation value is expanded in terms of the fermion bi-linears $\pi_\ell \frac{\partial}{\partial \pi_m}$.
Consequently, if we expand $E_{2,1}$ and $F_{1,2}$ in fermionic variables
$\pi_k$, the expansion has any higher order terms. The Hamiltonians $H_{1,1}$ and $H_{2,1}$ also have the same property. 
We would like to emphasize that this is qualitatively different from $H_{1,1}$ and $H_{2,1}$ whose fermionic parts are bi-linear 
in fermionic variables (see \eqref{H2-bilinear} and \eqref{H1-bilinear}). 

Though the formulas for $E_{2,1}$ and $F_{1,2}$ in Conjecture \ref{Pieri-con2} look rather involved, 
they actually satisfy the required properties (see \eqref{EF-com} with $i\neq j$);
\begin{prp}\label{Prop-1.3}
Conjectures \ref{Pieri-con} and \ref{Pieri-con2} imply the anti-commutation relations
\begin{equation}
\left[ E_{2,1}, F_{1,0} \right]_{+} = \left[ E_{2,0} , F_{1,2} \right]_{+} = 0.
\end{equation}
\end{prp}
%%%
Moreover, the positive mode Hamiltonians $H_{1,1}$ and $H_{2,1}$ are generated 
by the following anti-commutators (see \eqref{H21-anti-com} and \eqref{H11-anti-com}); 
\begin{prp}\label{Prop-1.4}
\begin{align}
\left[ E_{2, 1}, F_{2,-1} \right]_{+} &= (t/q)^{\frac{1}{2}}u (1+ u^{-1}(q/t)^{\frac{1}{2}}H_{2,+1}), \\
\left[ E_{1,0}, F_{1,2} \right]_{+} &= H_{1,+1},
\end{align}
where $H_{2,+1}$ and $H_{1,+1}$ are Hamiltonians obtained in  \cite{Galakhov:2025phf}.
\end{prp}
%%%
%
The quantum toroidal algebra $\mathcal{U}_{q,t}(\widehat{\widehat{\mathfrak{gl}}}_{1|1})$ is invariant under the involution $\iota$. 
But the level zero representation that gives the super Macdonald polynomials $\mathcal{M}_\Lambda(x, \theta; q,t)$ is
asymmetric for $\iota$. We believe this is closely related to the fact that the algebra is shifted. 
In \cite{Alarie-Vezina:2019ohz} the operator $\mathsf{T}_q$ that satisfies 
\begin{equation}
\mathsf{T}_q \mathcal{M}_{\Lambda}(x, \theta;q,t) = q^{|\Lambda^{\mathsf{a}}|} \mathcal{M}_{\Lambda}(x, \theta;q^{-1},t^{-1}),
\end{equation}
was introduced, where $|\Lambda^{\mathsf{a}}|$ denotes the total number of boxes in the odd rows. 
 The operator $\mathsf{T}_q$ was defined as a linear operator on each fermion sector 
in the space of supersymmetric polynomials. It is desirable to work out a realization of $\mathsf{T}_q$
as an involution of the quantum toroidal algebra. We leave it to a future work.

The paper is organized as follows;
in section 2 we review the shifted quantum toroidal algebra of type $\mathfrak{gl}_{1\vert 1}$. 
The super Macdonald polynomials are identified with a basis which diagonalizes the mutually commuting Cartan modes 
in the super Fock representation. We work out the generating functions of the 
eigenvalues. By looking at the action of the super charges of the shifted quantum toroidal algebra, 
we present the Pieri formula of the super Macdonald polynomials in section 3. The validity of our conjecture is checked
for lower levels of  the super Macdonald polynomials. In section 4 we obtain the first pair of negative mode Hamiltonians 
by computing the anti-commutator of the super charges. We also make comparisons with the results in the literatures 
\cite{Alarie-Vezina:2019ohz} and \cite{Galakhov:2025phf}. As we explained above, the positive modes of the Hamiltonians are more involved. 
Some of the leading terms are computed in section 5, based on the Pieri rule for the higher modes 
of the super currents $E_i(z)$ and $F_i(z)$. 
Finally motivated with the integral formula of the Hamiltonians in  \cite{Galakhov:2025phf},
we propose the corresponding integral formula for the super charges $E_{2,1}$ and $F_{1,2}$ and discuss some of the consequences. 
In particular we prove Propositions \ref{Prop-1.3} and \ref{Prop-1.4}.

There are several appendices to this paper. 
In appendix A we provide a list of the super Macdonald polynomials up to level 4. 
We have explicitly checked the validity of our conjecture 
based on these data. In appendix B, we work out the Pieri rule of the higher modes. In particular the Pieri rule for the positive modes
is used to compute the lower term of the expansion of positive Hamiltonians in  fermionic derivatives. 
In Appendix C we review the integral formula for the positive Hamiltonians proposed by \cite{Galakhov:2025phf}
and show that they commute with the appropriate super charges. 
This is a consistency check of the results in \cite{Galakhov:2025phf}
and plays a crucial role in the related computations in section 6. 
Finally in Appendix D, we review the basic properties of the involution 
operator $\mathsf{T}_q$ introduced in \cite{Alarie-Vezina:2019ohz}. 

%%%%%%%%%%%%%%%%%%%%%%%%%%%%%%%%%%%%%%%%%%%%%%%%%%%%%%%%%%%%%%%%%%%%%%%%%%%%%%%%%%%%%%%%%

\subsection{Some notation}

It is convenient to introduce $[[n]]_{(t_1,t_2)}$ in terms of the generating function.
\begin{equation}\label{qt-integer1}
\sum_{n=1}^\infty [[n]]_{(t_1,t_2)} z^{n-1} =\frac{1}{(1-t_1z)(1-t_2z)} .
\end{equation}
Since
\begin{equation*}
\frac{1}{(1-t_1z)(1-t_2z)} = \frac{1}{(t_1-t_2)z} \left( \frac{1}{1-t_1z} - \frac{1}{1-t_2z}  \right) 
= \frac{1}{t_1 - t_2} \sum_{n=1}^\infty (t_1^n -t_2^n) z^{n-1},
\end{equation*}
we have
\begin{equation}\label{qt-integer2}
[[n]]_{(t_1,t_2)} = \frac{t_1^n - t_2^n}{t_1 - t_2}.
\end{equation}
Note that
\begin{equation}\label{qt-integer3}
[[n]]_{(t_1,t_2)} = \sum_{k=0}^{n-1} t_1^k t_2^{n-k-1} =t_2^{n-1} \sum_{k=0}^{n-1} (t_1/t_2)^k. 
\end{equation}

%%%%%%%%%%%%%%%%%%%%%%%%%%%%%%%%%%%%%%%%%%%%%%%%%%%%%%%%%%%%

\section{Shifted quantum toroidal algebra of type $\mathfrak{gl}_{1\vert 1}$}
\label{section:shifted-algebra}

The shifted quantum toroidal algebra $\mathcal{U}_{q,t}(\widehat{\widehat{\mathfrak{gl}}}_{1|1})$
of type $\mathfrak{gl}_{1\vert 1}$ is generated by $E_{i,k}, F_{i,k}, K_{i, \pm r}^{\pm}$
and a central element $C$, where $i \in \{1,2 \}, k \in \mathbb{Z}$ and $r \in \mathbb{Z}_{\geq 0}$. 
The algebra is $\mathbb{Z}_2$-graded. The generators $E_{i,k}, F_{i,k}$ are odd and other generators are even.
To write down the defining relations of the algebra it is convenient to introduce the generating currents;
\begin{align}\label{EF-modes}
E_i(z) &= \sum_{k \in \mathbb{Z}} E_{i,k} z^{-k}, \qquad F_i(z) = \sum_{k \in \mathbb{Z}} F_{i,k} z^{-k}, \\
K_i^{\pm}(z) &= \sum_{r \geq 0} K_{i, \pm r}^{\pm} z^{\mp r}
= K_{i,0}^{\pm} \exp 
\left( \pm \sum_{r=1}^\infty H_{i, \pm r}z^{\mp r} \right).
\end{align}
One of the canonical ways to write down the commutation relations of the quantum toroidal algebra 
is based on the deformed Cartan matrix.
However, for the Lie super algebra $\mathfrak{gl}_{1\vert 1}$ the Cartan matrix vanishes and this does not work.
We employ the idea of the quiver quantum toroidal algebra to deduce the commutation relations 
\cite{Galakhov:2021xum, Galakhov:2021vbo, Noshita:2021dgj},
which are given by\footnote{Usually we impose $K_{i,0}^{+} K_{i,0}^{-} = K_{i,0}^{-} K_{i,0}^{+} =1$ 
for the zero modes of the unshifted algebra. But in the shifted case we do not assume it.
There are also Serre relations, which we omit in this paper.}
\begin{align}\label{K-zero-mode}
K_i^{\pm}(z) K_j^{\pm}(w) &= K_j^{\pm}(w) K_i^{\pm}(z),
\\
\label{KK-com}
K_i^{-}(z) K_j^{+}(w) &= \frac{\varphi^{j \Rightarrow i}(z,Cw)}{\varphi^{j \Rightarrow i}(Cz,w)}
K_j^{+}(w) K_i^{-}(z),
\\
\label{KE-com}
K_i^{\pm}(C^{\frac{1\mp 1}{2}}z) E_j(w) &= \varphi^{j \Rightarrow i}(z,w)  
E_j(w) K_i^{\pm}(C^{\frac{1\mp 1}{2}}z), 
\\
\label{KF-com}
K_i^{\pm}(C^{\frac{1\pm 1}{2}}z) F_j(w) &= \varphi^{j \Rightarrow i}(z,w)^{-1}  
F_j(w) K_i^{\pm}(C^{\frac{1\pm 1}{2}}z),
\\
\label{shifted}
\left[ E_i(z), F_j(w) \right]_{+} &= \delta_{ij} \left[ \delta \left( \frac{Cw}{z} \right) z^{r_i} K_i^{+}(z) - 
\delta \left( \frac{Cz}{w} \right)K_i^{-}(w) \right], 
\\
E_i(z) E_j(w) &= (-1) \cdot \varphi^{j \Rightarrow i}(z,w)  E_j(w) E_i(z),  
\\
F_i(z) F_j(w) &= (-1) \cdot \varphi^{j \Rightarrow i}(z,w)^{-1}  F_j(w) F_i(z), 
\end{align}
where $r_i \in \mathbb{Z}$, which only appears in \eqref{shifted}, is the shift parameter.
The structure function $\varphi^{i \Rightarrow j}(z,w)$ is defined as follows;
\beqa\label{bond1}
 \varphi^{1 \Rightarrow 1}(z,w)  &=&  \varphi^{2\Rightarrow 2}(z,w)  = 1, 
 \\
 \varphi^{1 \Rightarrow 2}(z,w) &=&   \varphi^{2 \Rightarrow 1}(z,w)^{-1} \CR
 &=& \frac{\phi(q_2;z,w)  \phi(q_2^{-1};z,w)}{\phi(q_1;z,w)  \phi(q_1^{-1};z,w)}
 = \frac{(z-q_2 w)(z - q_2^{-1}w)}{(z-q_1 w)(z - q_1^{-1}w)},
\eeqa
where we introduce the notation\footnote{In our previous paper \cite{Kanno:2025ifd},
we employed the symmetric form $\phi(p;z,w) := (p^{\frac{1}{2}}z - p^{-\frac{1}{2}}w)$.
But it turns out that the Pieri formula will become simpler with \eqref{phi-def}.}
\begin{equation}\label{phi-def}
\phi(p;z,w) := (z - p^{-1}w).
\end{equation}

The commutation relations of the currents $K_i^{\pm}(z)$ are reproduced by 
the commutation relation of their modes
\begin{equation}
\left[ H_{i,r}, H_{j,s} \right] = \epsilon_{ij} \delta_{r+s, 0} \frac{C^r - C^{-r}}{r}
 (q^{\frac{r}{2}} - q^{-\frac{r}{2}})(t^{\frac{r}{2}} - t^{-\frac{r}{2}}),
\end{equation}
where
\begin{equation}
q_1^2 = q t^{-1}, \qquad q_2^2 = qt.
\end{equation}
In fact we obtain
\begin{align}
K_j^{+}(w) K_i^{-}(z) &= \exp \left( \epsilon_{ij} \sum_{r=1}^\infty \frac{1}{r}(C^r-C^{-r})
 (q^{\frac{r}{2}} - q^{-\frac{r}{2}})(t^{\frac{r}{2}} - t^{-\frac{r}{2}})
\left(\frac{z}{w}\right)^r \right)  K_i^{-}(z) K_j^{+}(w) \CR
&= \left[\frac{(1- q_1 (Cz/w) )(1- q_1^{-1}(Cz/w))}{(1- q_2 (Cz/w) )(1- q_2^{-1} (Cz/w) )}\right]^{\epsilon_{ij}} \CR
& \qquad \qquad \times \left[(Cz,w) \to (z,Cw) \right]^{-\epsilon_{ij}} K_i^{-}(z) K_j^{+}(w),
\end{align}
which agrees with \eqref{KK-com}. 

We can also check that the commutation relations with the modes of $E_i(w)$ and $F_i(w)$;
%\begin{align}
%\left[ H_{i,r}, E_j(w) \right] &= +\epsilon_{ij} \frac{C^{\frac{1}{2}(r-|r|)}}{r}
%(q^{\frac{r}{2}} - q^{-\frac{r}{2}})(t^{\frac{r}{2}} - t^{-\frac{r}{2}}) w^{r}E_j(w), \label{HE-com}
%\\
%\left[ H_{i,r}, F_j(w) \right] &= - \epsilon_{ij}  \frac{C^{\frac{1}{2}(r+|r|)}}{r}
%(q^{\frac{r}{2}} - q^{-\frac{r}{2}})(t^{\frac{r}{2}} - t^{-\frac{r}{2}}) w^{r}F_j(w), \label{HF-com}
%\end{align}
%Substituting the mode expansion \eqref{EF-modes} to \eqref{HE-com} and \eqref{HF-com}, we have 
\begin{align}
\left[ H_{i,r}, E_{j,s} \right] &= +\epsilon_{ij} \frac{C^{\frac{1}{2}(r-|r|)}}{r}
 (q^{\frac{r}{2}} - q^{-\frac{r}{2}})(t^{\frac{r}{2}} - t^{-\frac{r}{2}}) E_{j,r+s}, 
\\
\left[ H_{i,r}, F_{j,s} \right] &= - \epsilon_{ij}  \frac{C^{\frac{1}{2}(r+|r|)}}{r} 
(q^{\frac{r}{2}} - q^{-\frac{r}{2}})(t^{\frac{r}{2}} - t^{-\frac{r}{2}})F_{j,r+s},
\end{align}
recover \eqref{KE-com} and \eqref{KF-com}.

The algebra is a Hopf superalgebra with the following coproduct;\footnote{For the definition of the counit and the antipode see \cite{Noshita:2021dgj}.}
\begin{align}
\Delta E_i(z) &= E_i(z) \otimes 1 + K_i^{-}(C_1 z ) \otimes E_i(C_1 z), 
\\
\Delta F_i(z) &= F_i(C_2z) \otimes K_i^{+}(C_2 z) + 1 \otimes F_i(z), 
\\
\Delta K_i^{+} &= K_i^{+}(z) \otimes K_i^{+}(C_1^{-1}z),
\\
\Delta K_i^{-} &= K_i^{-}(C_2^{-1} z) \otimes K_i^{-}(z),
\\
\Delta C &= C \otimes C,
\end{align}
where $C_1 = C \otimes 1$ and $C_2 = 1 \otimes C$. 

%%%%%%%%%%%%%%%%%%%%%%%%%%%%%%%%%%%%%%%%%%%%%%%%%%%%%%%%%%%%%%%%%%%%%%%%%%%%%%%%%%%%%%%%%%%%%%%%%%%%%%%%%%%

\subsection{Level zero representations}

When $C=1$, we say the representation has level zero. In this case all the Cartan modes $K_{i,\pm r}^\pm$ are
commuting and there is a basis consisting of simultaneous eigenvectors of $K_{i,\pm r}^\pm$.
A basic example of the level zero representation is the vector representation $V(u)$ with the spectral parameter $u$. 
Let us introduce the vector space spanned by $[u]_{j, \barsigma}$, where $ j \in \mathbb{Z}, \barsigma \in \mathbb{Z}_2 = \{ 0, 1\}$.
It is convenient to identify $E_2(z) \equiv E_0(z)$ and similarly for other currents
and to use the notation $\overline{s}$ when $s$ is regarded as an element in $\mathbb{Z}_2$. 
With this convention the representation $V(u)$ is defined as follows;
\begin{align}
E_{\overline{s}}(z) \cdot [u]_{k, \barsigma} &= \mathcal{E}_{\overline{s}}~
\delta \left( \frac{z}{u q_1^{k+1} q_2^{k+1 - \barsigma}} \right) \overline{\delta}_{\overline{s}+\barsigma, 1} 
[u]_{k+\overline{s},1-\barsigma},
\\
F_{\overline{s}}(z) \cdot [u]_{k, \overline{\sigma}} &= \mathcal{F}_{\overline{s}}~
\delta \left( \frac{z}{u q_1^{k+1-\overline{s}} q_2^{k}} \right) \overline{\delta}_{\overline{s}+\barsigma, 0}
[u]_{k-\overline{s},1-\barsigma},
\\
K_{\overline{s}}^{\pm}(z) \cdot [u]_{k, \barsigma} 
&= \left[ \Psi^{({\overline{s}})}_{[u]_{k, \barsigma}}(z) \right]_{\pm} [u]_{k, \barsigma},
\end{align}
where
\begin{align}
\label{Cartan-eigen1a} 
\Psi^{(\overline{1})}_{[u]_{k, \barsigma}}(z) 
&= \frac{\phi(q_1^{-k-2+\barsigma} q_2^{-k+\barsigma}; z,u)}
{\phi(q_1^{-k-1+\barsigma} q_2^{-k-1+\barsigma}; z,u)},
\\
\label{Cartan-eigen2a} 
\Psi^{(\overline{2})}_{[u]_{k, \barsigma}}(z)
&= \frac{\phi(q_1^{-k} q_2^{-k-1};z,u)}{\phi(q_1^{-k-1} q_2^{-k};z,u)}.
\end{align}
For a rational function $g(z)$, $[g(z)]_\pm$ denotes the expansion around $z=\infty, 0$.

The normalization factors $\mathcal{E}_{\overline{s}}$ and $\mathcal{F}_{\overline{s}}$ are fixed
by the commutation relation \eqref{shifted}.
By using the formula
\begin{equation}
\left[ \frac{z - \beta u}{z - \alpha u} \right]_{+} - \left[ \frac{z - \beta u}{z - \alpha u} \right]_{-}
= \left( 1- \frac{\beta}{\alpha} \right) \delta\left( \frac{z}{\alpha u} \right),
\end{equation}
we obtain
\begin{equation}
\mathcal{E}_1 \mathcal{F}_1 = 1- q_1 q_2^{-1}= 1 - t^{-1},
\qquad
\mathcal{E}_2 \mathcal{F}_2 = 1- q_1^{-1} q_2= 1 -t.
\end{equation}
For later convenience we choose
\begin{equation}
\mathcal{E}_1 = 1, \quad \mathcal{F}_1 = 1-t^{-1},
\qquad \mathcal{E}_2 = 1-t, \quad \mathcal{F}_2 = 1.
\end{equation}
The vector representation is unshifted. Namely we have the commutation relation \eqref{shifted} with $r_i=0$.

Let us consider $N$-times tensor product $V(u_1) \otimes V(u_2) \otimes \cdots \otimes V(u_N)$ of the vector 
representations with spectral parameters $u_r = (q_1 q_2^{-1})^{r-1} u$. 
We assign $(\lambda_i, \overline{\sigma}_i),~\lambda_i \in \mathbb{Z}_{\geq 0},
\overline{\sigma}_i \in \mathbb{Z}_2$ for the $i$-th component.
The states in the tensor product are denoted 
\begin{equation}
\vert \lambda, \overline{\sigma} \rangle := \prod_{i=1}^N [(q_1q_2^{-1})^{i-1} u ]_{\lambda_i -1, \overline{\sigma}_i}
\in V(u_1) \otimes V(u_2) \otimes \cdots \otimes V(u_N). 
\end{equation}
For the tensor product of the level zero representations we use the (Drinfeld) coproduct of the algebra;
\begin{align}
\Delta^{N-1} (E_{\overline{s}}(z)) &= \sum_{k=1}^N \overbrace{K^{-}_{\overline{s}}(z) 
\otimes \cdots \otimes K^{-}_{\overline{s}}(z)}^{k-1} \otimes E_{\overline{s}}(z) \otimes 
\overbrace{1 \otimes \cdots \otimes 1}^{N-k}, \label{copro-E} \\
\Delta^{N-1} (F_{\overline{s}}(z)) &= \sum_{k=1}^N \overbrace{1 \otimes \cdots \otimes 1}^{k-1}
 \otimes F_{\overline{s}}(z) \otimes \overbrace{K^{+}_{\overline{s}}(z) 
 \otimes \cdots \otimes K^{+}_{\overline{s}}(z)}^{N-k}, \label{copro-F} \\
 \Delta^{N-1} (K^{\pm}_{\overline{s}} (z)) &= \overbrace{K^{\pm}_{\overline{s}} (z) \otimes 
 \cdots \otimes K^{\pm}_{\overline{s}}(z)}^{N}. \label{copro-K}
\end{align}
With the help of the coproduct, we can construct the super Fock representation 
by taking an infinite tensor product of the vector representations.
There is a basis of the super Fock representation which is labelled by the set of super Young diagrams.
By using the fact that the adjacent spectral parameters are shifted as $u_r = (q_1 q_2^{-1})^{r-1} u$,
we can check that the subspace spanned by the states $\vert \lambda, \overline{\sigma} \rangle$ corresponding 
to a super partition is invariant under the action of $E_{\overline{s}}(z)$ and $F_{\overline{s}}(z)$.
In particular, the state $\vert \varnothing, \overline{0} \rangle$ corresponding to the empty partition is
the highest weight state annihilated by $F_{\overline{s}}(z)$.

From the coproduct \eqref{copro-E} we obtain the action of $E_{\overline{s}}(z)$ as follows;
\begin{align}\label{E-action}
E_{\overline{s}}(z) \cdot \vert \lambda, \barsigma \rangle 
&= \mathcal{E}_{\overline{s}} \sum_{k=1}^{\ell(\lambda)+1}
 (-1)^{F(k)} \cdot \overline{\delta}_{s + \barsigma_k,1} 
\prod_{i=1}^{k-1} \left[  \Psi^{(s)}_{[u (q_1 q_2^{-1})^{i-1}]_{\lambda_i -1, \overline{\sigma_i}}}(z) \right]_{-}
\CR
& \qquad \times 
 \delta\left(\frac{z}{u (q_1q_2^{-1})^{k-1} q_1^{\lambda_k} q_2^{\lambda_k - \overline{\sigma_k}}} \right) 
\vert \lambda + 1_{k} \cdot \delta_{\overline{s},1}, \barsigma + \overline{1_k} \rangle,
\end{align}
where $F(k):= \displaystyle{\sum_{i=1}^{k-1}} \barsigma_i $ is the number of fermionic rows above the $k$-th row and 
the sign factor $ (-1)^{F(k)}$ comes from the fermionic nature of $E_{\overline{s}}(z)$.

The $N \to \infty$ limit of the $E_{\overline{s}}(z)$ action is well-defined,
since it is bounded by $\ell(\lambda)+1$. On the other hand we have to regularize the action of the Cartan 
currents $K^\pm_{\overline{s}}(z)$, since the coproduct \eqref{copro-K} leads to the infinite product. 
The infinite product is regularized by specifying the order of taking the product, such that the cancellation of 
infinitely many factors takes place. We define
\begin{align}
\prod_{i=\ell(\lambda) +1}^\infty \Psi^{(\overline{1})}_{[u(q_1q_2^{-1})^{i-1}]_{-1,0}}
&= \prod_{i=\ell(\lambda) +1}^\infty \frac{\phi(q_1^{-1} q_2; z, u(q_1q_2^{-1})^{i-1})}
{\phi(1; z, u(q_1q_2^{-1})^{i-1})}= \frac{1}{z - u(q_1q_2^{-1})^{\ell(\lambda)}},
\end{align}
and
\begin{align}
\prod_{i=\ell(\lambda) +1}^\infty \Psi^{(\overline{2})}_{[u(q_1q_2^{-1})^{i-1}]_{-1,0}}
&=  \prod_{i=\ell(\lambda) +1}^\infty \frac{\phi(q_1; z, u(q_1q_2^{-1})^{i-1})}
{\phi(q_2; z, u(q_1q_2^{-1})^{i-1})}= z- uq_1^{-1}(q_1q_2^{-1})^{\ell(\lambda)}.
\end{align}
Substituting the definition \eqref{Cartan-eigen1a} and \eqref{Cartan-eigen2a}, 
we obtain the following generating functions of the eigenvalues of the Cartan modes;
\begin{align}\label{K1-eigen}
K_1(z) &= \frac{1}{z - t^{-\ell(\lambda)}u}\prod_{i=1}^{\ell(\lambda)}
\frac{z-q^{\lambda_i - \barsigma_i}t^{-i}u}{z-q^{\lambda_i - \barsigma_i}t^{1-i}u},
\\
\label{K2-eigen}
K_2(z) &=  (z- (t/q)^{\frac{1}{2}}t^{-\ell(\lambda)}u)
\prod_{i=1}^{\ell(\lambda)} \frac{z-q^{\lambda_i -\frac{1}{2}}t^{\frac{3}{2}-i}u}{z-q^{\lambda_i -\frac{1}{2}}t^{\frac{1}{2}-i}u},
\end{align}
where we have defined $q_1q_2^{-1}=t^{-1}$ and $q_1q_2= q$.

The leading coefficients of the Cartan currents $K_{i}^\pm(z)$ are\footnote{In \cite{Kanno:2025ifd}
the symmetric definition of $\phi(p;z,w)$ led to the $\ell(\lambda)$ dependent leading coefficients. 
An advantage of the present definition \eqref{phi-def} is that the leading coefficients
become $\ell(\lambda)$ independent.}
\begin{equation}
 z^{-1}(K_1^{+})_{0} = z^{-1}, \qquad (K_1^{-})_0 = - u^{-1}, 
\end{equation}
and 
\begin{equation}
z (K_2^{+})_{0} = z, \qquad (K_2^{-})_0 = -(t/q)^{\frac{1}{2}}u.
\end{equation}
We see that the shift parameters of the super Fock representation are $r_1=-1$ and $r_2=1$.

Finally from the coproduct \eqref{copro-F} the action of $F_{\overline{s}}(z)$ is
\begin{align}\label{F-action}
F_{\overline{s}}(z) \vert \lambda, \barsigma \rangle &= \mathcal{F}_{\overline{s}}
 \sum_{k=1}^{\ell(\lambda)} (-1)^{F(k)} \cdot \overline{\delta}_{\overline{s},\sigma_k}
\prod_{i=k+1}^{\ell(\lambda)} \left[ \Psi^{(s)}_{[u(q_1q_2^{-1})^{i-1}]_{\lambda_i-1, \barsigma_i}}(z) \right]_{+} \CR
& \times  \prod_{i=\ell(\lambda)+1}^{\infty} \left[ \Psi^{(s)}_{[u(q_1q_2^{-1})^{i-1}]_{-1, 0}}(z) \right]_{+}  
\delta\left( \frac{z}{uq_1^{\lambda_k +k -1 -\overline{s}}q_2^{\lambda_k -k}} \right) 
\vert \lambda - 1_k \cdot \overline{\delta}_{s,1}, \barsigma - \overline{1_k} \rangle.
\end{align}
The infinite product is regularized in the same manner as $K_{\overline{s}}(z)$.

%%%%%%%%%%%%%%%%%%%%%%%%%%%%%%%%%%%%%%%%%%%%%%%%%%%%%%%%%%%%%%%%%%%%%%

\subsection{Eigenvalues of a \lq\lq quartet\rq\rq\ of Hamiltonians}

In the level zero representation $K_i^{\pm}(z)$ are mutually commuting and the super Macdonald polynomials 
$\mathcal{M}_{\Lambda}(x, \theta;q,t)$ are simultaneous eigenstates of the Hamiltonians $H_{i, \pm 1}$.
In this subsection we evaluate the eigenvalues of the Hamiltonians $H_{i, \pm 1}$.
\begin{prp}
\label{eigenvalues}
\begin{align}
\label{eigen-1}
1 - u(t/q)^{\frac{1}{2}}H_{2,-1} &=  (t-1)(q^{-1}-1) \sum_{(i,j) \in \Lambda^{\circledast}}q^{1-j} t^{i-1},
\\
1 + u H_{1,-1} &= (t-1)(q^{-1}-1) \sum_{(i,j) \in \Lambda^{*}} q^{1-j}t^{i-1},
\\
1+ u^{-1}(q/t)^{\frac{1}{2}}H_{2,+1} &=  (t^{-1}-1)(q-1) \sum_{(i,j) \in \Lambda^{\circledast}}q^{j-1} t^{1-i},
\\
\label{eigen-4}
1- u^{-1} H_{1,+1} &= (t^{-1}-1)(q-1) \sum_{(i,j) \in \Lambda^{*}} q^{j-1} t^{1-i},
\end{align}
where $u$ is the spectral parameter of the representation.
We define $\sigma_i \in \{0,1 \} $ according to the $\mathbb{Z}_2$ grading of $\Lambda_i$.
Then $\Lambda^{\circledast}$ and $\Lambda^{*}$ are the ordinary partitions whose components are defined 
by $\Lambda + \frac{1}{2}\sigma$ and  $\Lambda - \frac{1}{2}\sigma$, respectively. 
\end{prp}

From \eqref{K1-eigen} and \eqref{K2-eigen} we can compute the eigenvalues of the first Hamiltonians
\begin{equation}
 H_{1,-1} = -\frac{d}{dz} \log K_1^{-}(0), \qquad
H_{2,-1} = - \frac{d}{dz}\log K_2^{-}(0).
\end{equation}
We have
\begin{equation}
\frac{d}{dz}\log K_2^{-} (z)
= \frac{1}{z- (t/q)^{\frac{1}{2}} t^{-\ell(\lambda)}u}+ \sum_{i=1}^{\ell(\lambda)}
\left( \frac{1}{z-q^{\lambda_i - \frac{1}{2}}t^{\frac{3}{2}-i}u}-\frac{1}{z-q^{\lambda_i - \frac{1}{2}}t^{\frac{1}{2}-i}u} \right),
\end{equation}
and
\begin{align}\label{First-eigen-1}
1 - u(t/q)^{\frac{1}{2}}H_{2,-1} &=  1 -t^{\ell(\lambda)} + (t-1) \sum_{i=1}^{\ell(\lambda)}q^{-\lambda_i} t^{i-1} \CR
&= (t-1) \sum_{i=1}^{\ell(\lambda)}(q^{-\lambda_i}-1) t^{i-1} =  (t-1) \sum_{i=1}^{\infty}(q^{-\lambda_i}-1) t^{i-1} \CR
&=  (t-1)(q^{-1}-1) \sum_{(i,j) \in \Lambda^{\circledast}}q^{1-j} t^{i-1} .
\end{align}
Similarly from
\begin{equation}
\frac{d}{dz} \log K_1^{-}(z) = - \frac{1}{z- t^{-\ell(\lambda)}u} + \sum_{i=1}^{\ell(\lambda)} 
\left( \frac{1}{z- q^{\lambda_i - \sigma_i}t^{-i}u} - \frac{1}{z- q^{\lambda_i - \sigma_i}t^{1-i}u} \right),
\end{equation}
we obtain
\begin{align}\label{First-eigen-2}
1 + u H_{1,-1} &= (t-1) \sum_{i=1}^{\ell(\lambda)} (q^{-\lambda_i+ \sigma_i}-1)t^{i-1} \CR
&= (t-1)(q^{-1}-1) \sum_{(i,j) \in \Lambda^{*}} q^{1-j}t^{i-1}.
\end{align}

We can also compute the eigenvalues of the Hamiltonians $H_{1,+1}$ and $H_{2,+1}$ on the positive side. 
Since the algebra is shifted, we have
\begin{align}
\log (zK_1^{+}(z)) &= \log K_{1,0}^{+} + \sum_{r=1}^\infty H_{1,r}z^{-r},
\\
\log (z^{-1}K_2^{+}(z)) &= \log K_{2,0}^{+} + \sum_{r=1}^\infty H_{2,r}z^{-r}.
\end{align}
Recall that $K_{1,0}^{+}=K_{2,0}^{+}=1$. Setting $w=z^{-1}$ we have
\begin{align}
\frac{d}{dw} \left( \sum_{r=1}^\infty H_{1,r}w^r \right)
&= \frac{t^{-\ell(\lambda)}u}{1-t^{-\ell(\lambda)}uw}
- \sum_{i=1}^{\ell(\lambda)} \left( \frac{q^{\lambda_i - \barsigma_i}t^{-i}u}{1-q^{\lambda_i - \barsigma_i}t^{-i}uw} 
- \frac{q^{\lambda_i - \barsigma_i}t^{1-i}u}{1-q^{\lambda_i - \barsigma_i}t^{1-i}uw} \right),
\\
\frac{d}{dw} \left( \sum_{r=1}^\infty H_{2,r}w^r \right)
&= - \frac{(t/q)^{\frac{1}{2}}t^{-\ell(\lambda)}u}{1-(t/q)^{\frac{1}{2}}t^{-\ell(\lambda)}uw}
- \sum_{i=1}^{\ell(\lambda)} \left( \frac{q^{\lambda_i - \frac{1}{2}}t^{\frac{3}{2}-i}u}{1-q^{\lambda_i - \frac{1}{2}}t^{\frac{3}{2}-i}uw}
- \frac{q^{\lambda_i - \frac{1}{2}}t^{\frac{1}{2}-i}u}{1-q^{\lambda_i - \frac{1}{2}}t^{\frac{1}{2}-i}uw}
\right).
\end{align}
Hence, we obtain
\begin{align}
1- u^{-1} H_{1,+1} &= 1-  t^{-\ell(\lambda)} + (t^{-1}-1) \sum_{i=1}^{\ell(\lambda)} q^{\lambda_i - \barsigma_i}t^{1-i} \CR
&= (t^{-1}-1) \sum_{i=1}^{\ell(\lambda)} (q^{\lambda_i - \barsigma_i} -1) t^{1-i},
\\
1+ u^{-1}(q/t)^{\frac{1}{2}}  H_{2,+1} &= 1 - t^{-\ell(\lambda)} + (t^{-1}-1) \sum_{i=1}^{\ell(\lambda)}  q^{\lambda_i}t^{1-i} \CR
&= (t^{-1}-1) \sum_{i=1}^{\ell(\lambda)}  (q^{\lambda_i}-1) t^{1-i}.
\end{align}
Thus the comparison of the eigenvalues implies the correspondence
\begin{align}
1+ uH_{1,-1} &\longleftrightarrow 1 - u^{-1} H_{1,+1},
\\
1- u(t/q)^{\frac{1}{2}} H_{2,-1} &\longleftrightarrow 1 + u^{-1}(q/t)^{\frac{1}{2}} H_{2,+1},
\end{align}
under the involution $(q,t) \to (q^{-1}, t^{-1})$. 
Recall that the (odd) super Macdonald polynomials are not invariant under the involution $(q,t) \to (q^{-1}, t^{-1})$.
But the above result means that they are actually eigenfunctions of four mutually commuting Hamiltonians $H_{1, \pm 1}$ and $H_{2, \pm 1}$.
It also means though eigenvalues are related by the involution $(q,t) \to (q^{-1}, t^{-1})$,
the Hamiltonians are not simply related by the inversion.

It is natural to identify the finite variable version of the Hamiltonians with $\mathcal{D}_{1,N}, \mathcal{D}_{2,N}$ and 
$\overline{\mathcal{D}}_{1,N}, \overline{\mathcal{D}}_{2,N}$ in \cite{Alarie-Vezina:2019ohz}. 
We may obtain the eigenvalues for $N$ variables case by rewriting the summation $\displaystyle{\sum_{(i,j) \in \lambda}}$
to $\displaystyle{\sum_{i=1}^N}\displaystyle{\sum_{j=1}^{\lambda_i}}$. Then we find the following relations;
\begin{align}
(t-1)\mathcal{D}_{2,N} - t^N &=- u(t/q)^{\frac{1}{2}}H_{2,-1}, \\
\label{finite-H}
(t-1)\mathcal{D}_{1,N} -t^N &= u H_{1,-1}, \\
(t^{-1}-1) \overline{\mathcal{D}}_{2,N} - t^{-N} &= u^{-1}(q/t)^{\frac{1}{2}}H_{2,+1}, \\
(t^{-1}-1)  \overline{\mathcal{D}}_{1,N} - t^{-N} &=- u^{-1} H_{1,+1},
\end{align}
where the additional term $-t^{\pm N}$ is due to the truncation of the length of partitions by $N$. 

Similarly the comparison of the eigenvalues implies the relations to the Hamiltonians in \cite{Galakhov:2025phf};
\begin{align}
\label{R-dictionary+}
\widehat{\mathcal{H}}^{\mathsf{t}, +} &= - u^{-1}(q/t)^{\frac{1}{2}} H_{2,+1},
\qquad 
\widehat{\mathcal{H}}^{\overline{\mathsf{t}}, +} = u^{-1} H_{1,+1}, \\
\label{R-dictionary-}
\widehat{\mathcal{H}}^{\mathsf{t}, -} &= u (t/q)^{\frac{1}{2}} H_{2,-1},
\qquad 
\widehat{\mathcal{H}}^{\overline{\mathsf{t}}, -} = - u H_{1,-1},
\end{align}
where in  \cite{Galakhov:2025phf} $\mathsf{t}$ and $\overline{\mathsf{t}}$ are denoted by the upper left triangle and its compliment in the square $\square$. 
%\begin{picture}(5,4)
%\setlength{\unitlength}{0.8mm}
%\thicklines
%\put(0,4){\line(1,0){5}}
%\put(0,-1){\line(0,1){5}}
%\put(0,-1){\line(1,1){5}}
%\end{picture}
%%%%%%%%%%%%%%%%%%%%%%%%%%%%%%%%%%%%%%%%%%%%%%%%%%%%%%%%%%%%%%%%%%%%%%%%%%%%%

\section{Pieri rule for the super Macdonald polynomials}
\label{section:Pieri}

Let us define the Pieri coefficients $\psi_{k}^{(\overline{s})}(q,t)$ by the action of the zero modes
of $E_{\overline{s}}(z)$;
\begin{align}
E_{\overline{s},0} \vert \lambda, \barsigma \rangle &= \sum_{k=1}^{\ell(\lambda)+1}\psi_{k}^{(\overline{s})}(q,t)
\cdot \overline{\delta}_{s + \barsigma_k,1} 
\vert \lambda + 1_{k} \cdot \delta_{\overline{s},1}, \barsigma + \overline{1_k} \rangle.
\end{align}

From \eqref{E-action}, we obtain
\begin{align}
\psi_{k}^{(1)}(q,t) &=
 \mathcal{E}_{1} \cdot (-1)^{F(k)} \prod_{i=1}^{k-1} \Psi^{(1)}_{[u t^{1-i}]_{\lambda_i -1, \overline{\sigma_i}}}(u t^{1-k} q^{\lambda_k})  \CR
%&= (-1)^{F(k)} \prod_{i=1}^{k-1} \frac{t^{1-k}q^{\lambda_k} - q^{\lambda_i - \barsigma_i}t^{-i}}{t^{1-k}q^{\lambda_k} - q^{\lambda_i - \barsigma_i}t^{1-i}} \CR
&= (-1)^{F(k)} \prod_{i=1}^{k-1} \frac{1- q^{\lambda_i -\lambda_k - \barsigma_i}t^{k-i-1}}{1- q^{\lambda_i -\lambda_k - \barsigma_i}t^{k-i}},
\end{align}
and
\begin{align}
\psi_{k}^{(2)}(q,t) &= \mathcal{E}_{2} \cdot
 (-1)^{F(k)} \prod_{i=1}^{k-1} \Psi^{(2)}_{[u t^{1-i}]_{\lambda_i -1, \overline{\sigma_i}}}
(u t^{\frac{1}{2}-k} q^{\lambda_k-\frac{1}{2}}) \CR
&= (1-t) (-1)^{F(k)} \prod_{i=1}^{k-1} \frac{1- q^{\lambda_i - \lambda_k}t^{1+k-i}}
 {1- q^{\lambda_i -\lambda_k}t^{k-i}}. 
\end{align}
Since the action of the current $E_{\overline{s}}(z)$ involves the delta function
\begin{equation}
\delta \left(\frac{z}{u(q_1q_2^{-1})^{k-1} q_1^{\lambda_k}q_2^{\lambda_k - (1-\overline{s})}}\right)
= \delta \left(\frac{z}{ut^{1-k-\frac{1}{2}(1-\overline{s})}q^{\lambda_k-\frac{1}{2}(1-\overline{s})}}\right),
\end{equation} 
the action of the modes $E_{\overline{s},n}$ is obtained by multiplying the factor
\begin{equation}\label{En-Pieri}
(ut^{1-k-\frac{1}{2}(1-\overline{s})}q^{\lambda_k-\frac{1}{2}(1-\overline{s})})^n.
\end{equation}
In the geometric construction of the level zero representation by the correspondence,
this is nothing but the multiplication of the first Chern class of an appropriate line bundle. 

Similarly we define the Pieri coefficients $\widetilde{\psi}_{k}^{(\overline{s})}(q,t)$ by the action of the zero modes
of $F_{\overline{s}}(z)$;
\begin{align}
F_{\overline{s},0} \vert \lambda, \barsigma \rangle &= \sum_{k=1}^{\ell(\lambda)}\widetilde{\psi}_{k}^{(\overline{s})}(q,t)
\cdot \overline{\delta}_{\overline{s} + \barsigma_k,0} 
\vert \lambda - 1_{k} \cdot \delta_{\overline{s},1}, \barsigma - \overline{1_k} \rangle.
\end{align}
From \eqref{F-action}, we obtain
\begin{align}\label{F1-coeff}
\widetilde{\psi}_k^{(1)}(q,t) 
&=  \mathcal{F}_{\overline{1}} \cdot(-1)^{F(k)} u^{-1} \frac{1}{q^{\lambda_k-1} t^{1-k}-t^{-\ell(\lambda)}}
\prod_{i=k+1}^{\ell(\lambda)} \Psi^{(1)}_{[u t^{1-i}]_{\lambda_i-1, \barsigma_i}}(u q^{\lambda_k-1} t^{1-k}) \CR
&=  (-1)^{F(k)} u^{-1} \frac{t^{k-1}(1-t)}{1-q^{\lambda_k-1}t^{\ell(\lambda)-k+1}}
\prod_{i=k+1}^{\ell(\lambda)} \frac{1 - q^{\lambda_k -\lambda_i -1 + \barsigma_i}t^{i-k+1}}
{1 - q^{\lambda_k -\lambda_i -1 + \barsigma_i}t^{i-k}},
\end{align}
and 
\begin{align}
\widetilde{\psi}_k^{(2)}(q,t)  
&=  \mathcal{F}_{\overline{2}} \cdot(-1)^{F(k)}u(t/q)^{\frac{1}{2}} (q^{\lambda_k} t^{-k}-  t^{-\ell(\lambda)}) 
\prod_{i=k+1}^{\ell(\lambda)} \Psi^{(2)}_{[ut^{1-i}]_{\lambda_i-1, \barsigma_i}}(u q^{\lambda_k -\frac{1}{2}} t^{\frac{1}{2}-k}) \CR
&=(-1)^{F(k)+1}u(t/q)^{\frac{1}{2}} t^{-k} (1-  q^{\lambda_k} t^{-k+\ell(\lambda)}) 
\prod_{i=k+1}^{\ell(\lambda)} \frac{1 - q^{\lambda_k - \lambda_i}t^{i-k-1}}{1 - q^{\lambda_k -\lambda_i}t^{i-k}}.
\end{align}

The factor $t^{\pm k}$ may be eliminated by considering not the zero modes but the shifted modes $F_{1, +1}$ and $F_{2,-1}$.
Recall that the action of the current $F_{\overline{s}}(z)$ involves the delta function 
\begin{equation}
\delta \left( \frac{z}{u q_1^{\lambda_k+k-1-\overline{s}} q_2^{\lambda_k - k}} \right)
= \delta \left( \frac{z}{u q^{\lambda_k -\frac{1}{2}(1+\overline{s})} t^{- k + \frac{1}{2}(1+\overline{s})}} \right).
\end{equation}
Hence the Pieri coefficients for $F_{1, +1}$ and $F_{2,-1}$ are
\begin{align}
\widetilde{\psi}_k^{(1,+1)}(q,t) 
&= (-1)^{F(k)} \frac{q^{\lambda_k-1}(1-t)}{1-q^{\lambda_k-1}t^{\ell(\lambda)-k+1}}
\prod_{i=k+1}^{\ell(\lambda)} \frac{1 - q^{\lambda_k -\lambda_i -1 + \barsigma_i}t^{i-k+1}}
{1 - q^{\lambda_k -\lambda_i -1 + \barsigma_i}t^{i-k}},
\\
\widetilde{\psi}_k^{(2,-1)}(q,t)  
&= (-1)^{F(k)+1}{q^{-\lambda_k}} (1-  q^{\lambda_k} t^{-k+\ell(\lambda)}) 
\prod_{i=k+1}^{\ell(\lambda)} \frac{1 - q^{\lambda_k - \lambda_i}t^{i-k-1}}{1 - q^{\lambda_k -\lambda_i}t^{i-k}}.
\end{align} 

By looking at the Pieri rule of the super Macdonald polynomials at lower levels explicitly, 
we have arrived at the following formula;
\begin{equation}\label{E-rep}
E_{1,0} = \pi_1, \qquad
E_{2,0} = \sum_{k=1}^\infty c_k[p] \frac{\partial}{\partial \pi_k},
\end{equation}
where
\begin{equation}\label{c_k-def}
\sum_{k=0}^\infty c_k[p] u^k = \exp \left(\sum_{r=1}^\infty \frac{1-t^r}{r} p_ru^r \right)
= \prod_{i=1}^\infty \frac{1-tu x_i}{1-ux_i}.
\end{equation}
Note that $E_{2,0}$ is linear in the derivatives in the fermionic power sum $\pi_k$.

Interestingly it is not the zero modes, but the shifted modes of $F_{\overline{s}}(z)$
that have a nice expression in terms of the multiplication and the derivative of the fermionic power sums. 
\begin{equation}\label{F-rep}
F_{1,+1} = \frac{\partial}{\partial \pi_1}, \qquad
F_{2,-1} = -\sum_{k=1}^\infty \pi_k \cdot \widetilde{c}_k \left[ \partial/ \partial p \right],
\end{equation}
where
\begin{equation}\label{tc_k-def}
\sum_{k=0}^\infty \widetilde{c}_k \left[ \partial/ \partial p \right] z^{-k}
= \exp \left( \sum_{n=1}^\infty (q^{-n}-1) \frac{\partial}{\partial p_n} z^{-n} \right).
\end{equation}
In fact among the modes of the currents $F_{\overline{s}}(z)$, $F_{1,+1}$ and $F_{2,-1}$ are chosen
by the property that they are independent of the spectral parameter $u$. 

%%%%%%%%%%%%%%%%%%%%%%%%%%%%%%%%%%%%%%%%%%%%%%%%%%%%%%%%%%%%%%%%%%%%%%%%%%%%

%%%%%%%%%%%%%%%%%%%%%%%%%%%%%%%%%%%%%%%%%%%%%%%%%%%%%%%%%%%%%%%%%%%%%%%%%%%%%%%%%%

\section{Hamiltonians $H_{1, -1}$ and $H_{2,-1}$ from the Pieri formula}
\label{section:Hamiltonians}

We first note that the mode expansion of the shifted anti-commutation relation \eqref{shifted} implies
\begin{align}
\left[ E_{1,0}, F_{1,+1} \right]_{+} &= z(z^{-1}K_1^{+}(z) - K_1^{-}(z))\vert_{z^0} = K_{1,0}^{+} =1,
\\
\left[ E_{2,0}, F_{2,-1} \right]_{+} &= z^{-1}(z K_2^{+}(z) - K_2^{-}(z))\vert_{z^0} = K_{2,0}^{+} + K_{2,0}^{-}H_{2,-1} \CR
&= 1 -(t/q)^{\frac{1}{2}}u H_{2,-1}.
\end{align}

\begin{prp}
Let us define\footnote{This is motivated by \eqref{First-eigen-1}.}
\begin{equation}
\mathcal{D}_2 := \frac{1}{1-t} \left((t/q)^{\frac{1}{2}}u H_{2,-1} -1 \right) 
=\mathcal{D}_2^{B} + \mathcal{D}_2^{F},
\end{equation}
where we have decomposed $\mathcal{D}_2$ to the bosonic term $\mathcal{D}_2^{B}$ and the fermionic one $\mathcal{D}_2^{F}$. 
Then $\mathcal{D}_2^{B}$ agrees with the Ruijsenaars-Macdonald Hamiltonian\footnote{
Compared with the standard definition, the parameters $(q,t)$ are inverted to $(q^{-1}, t^{-1})$.}
and $\mathcal{D}_2^{F}$ is bi-linear in fermions;
\begin{equation}
\mathcal{D}_2^{F} = - \sum_{k=1}^\infty \sum_{\ell=1}^\infty \widetilde{C}_{k, \ell}[p, \partial/ \partial p] 
%\Big[ \overline{c}_k[p],  \widetilde{c}_\ell \left[ \partial/ \partial p \right] \Big]  
\pi_k \frac{\partial}{\partial \pi_\ell},
\end{equation}
where
\begin{equation}\label{Fbi-linear}
 \widetilde{C}_{k, \ell}[p, \partial/ \partial p] 
 =  (1-t)(q^{-1}-1) \sum_{n=\ell -k}^{\ell-1} [[\ell -n ]]_{(q^{-1},t)} \cdot c_n[p] \cdot \widetilde{c}_{k-\ell+n} \left[ \partial/ \partial p \right].
\end{equation}
\end{prp}

\begin{proof}
Substituting \eqref{E-rep} and \eqref{F-rep} we obtain
\begin{equation}\label{H2}
1 -(t/q)^{\frac{1}{2}}u H_{2,-1}
= - \sum_{k=1}^\infty \sum_{\ell=1}^\infty
\left( c_k[p] \widetilde{c}_\ell \left[ \partial/ \partial p_i \right] 
\left[  \frac{\partial}{\partial \pi_k}, \pi_\ell\right] 
+ \Big[ \widetilde{c}_\ell \left[ \partial/ \partial p \right],  c_k[p] \Big]  
\pi_\ell\frac{\partial}{\partial \pi_k} \right).
\end{equation}
By the commutation relation of the fermionic power sum,\footnote{The relation to
the mode expansion of the NS fermion field $\psi(z) = \sum_{r \in \mathbb{Z} + \frac{1}{2}} \psi_r z^{-r}$
should be $\pi_\ell = \psi_{-\ell + \frac{1}{2}}$ and $\frac{\partial}{\partial \pi_k} = q^{-k} \psi_{k- \frac{1}{2}}$.}
\begin{equation}
\left[  \frac{\partial}{\partial \pi_k}, \pi_\ell\right]_{+} = \delta_{k, \ell}
\end{equation}
The bosonic part $\mathcal{D}_2^{B}$ is
\begin{equation}
\mathcal{D}_2^{B} = \sum_{k=1}^\infty \overline{c}_k[p] \widetilde{c}_k\left[ \partial/ \partial p \right]
= \frac{1}{1-t} \oint \frac{dz}{z} \exp \left(\sum_{n=1}^\infty \frac{(1-t^n)p_n}{n} z^n \right)
 \exp \left( \sum_{n=1}^\infty (q^{-n} -1) \frac{\partial}{\partial p_n} z^{-n} \right),
\end{equation}
which agrees with eq.(47) of \cite{Awata:1995zk} (see also \cite{Awata:1994xd}),
if we make the inversion $(q,t) \to (q^{-1}, t^{-1})$.\footnote{In \cite{Alarie-Vezina:2019ohz}
they employed the $(q,t)$-inverted version of the Macdonald-Ruijsenaars operator.}

On the other hand the fermionic part $\mathcal{D}_2^{F}$ is bi-linear in fermions and 
%\begin{equation}
%\mathcal{D}_2^{F} = - \sum_{k=1}^\infty \sum_{\ell=1}^\infty 
%\widetilde{C}_{k, \ell}[p, \partial/ \partial p]  \pi_k \frac{\partial}{\partial \pi_\ell},
%\end{equation}
%where the coefficients of the fermion bi-linear terms are
the generating function of the coefficients
$\widetilde{C}_{k, \ell}[p, \partial/ \partial p] := \Big[\widetilde{c}_k \left[ \partial/ \partial p \right] , c_\ell [p] \Big]$
 is computed as follows;
\begin{align}
\sum_{k, \ell=1}^\infty \widetilde{C}_{k, \ell}[p, \partial/ \partial p] z^{-k} w^{\ell} 
&= \left[ \exp \left( \sum_{n=1}^\infty (q^{-n}-1) \frac{\partial}{\partial p_n} z^{-n} \right), 
\exp \left(\sum_{m=1}^\infty \frac{1-t^m}{m} p_m w^m \right)\right] \CR
&= \frac{(1-t)(q^{-1}-1)(w/z)}{(1-t(w/z))(1-q^{-1}(w/z)}  \CR
& \qquad \times \exp \left(\sum_{m=1}^\infty \frac{1-t^m}{m} p_m w^m \right)
\exp \left( \sum_{n=1}^\infty (q^{-n}-1) \frac{\partial}{\partial p_n} z^{-n} \right) \CR
&= (1-t)(q^{-1}-1) \sum_{j=1}^\infty \sum_{m,n=0}^\infty
[[j]]_{(q^{-1},t)} \cdot c_n[p] \cdot \widetilde{c}_m \left[ \partial/ \partial p \right] z^{-(j+m)} w^{j+n}.
\end{align}
Hence, taking the coefficient of $j+m=k$ and $j+n=\ell$, we obtain \eqref{Fbi-linear}.
\end{proof}
Later in subsection \ref{subsec:VS-Galakhov} we will see that 
the fermion bi-linear terms with the coefficients \eqref{Fbi-linear} agrees with 
the proposal in \cite{Galakhov:2025phf}.

%%%%%%%%%%%%%%%%%%%%%%%%%%%%%%%%%%%%%%%%%%%%%%%%%%%%%%%%%%%%%%%%%%%%%%%%%

\subsection{The second Hamiltonian $H_{1, -1}$}

From the commutation relations
\begin{align}
\left[ H_{2,-1}, E_{1,0} \right] &= \left[ H_{2,-1}^{F}, \pi_1 \right] 
= (q^{\frac{1}{2}} - q^{-\frac{1}{2}})(t^{\frac{1}{2}}-t^{-\frac{1}{2}}) E_{1,-1}, \\
\left[ H_{2,-1}, F_{1,1} \right] &= \left[ H_{2,-1}^{F}, \frac{\partial}{\partial \pi_1} \right] 
= - (q^{\frac{1}{2}} - q^{-\frac{1}{2}})(t^{\frac{1}{2}}-t^{-\frac{1}{2}}) F_{1,0},
\end{align}
and \eqref{H2} we have
\begin{align}
u (1-q^{-1})(t-1) E_{1, -1} &= (-1)(1-t) \sum_{\ell=1}^\infty \left[ p_1, \widetilde{c}_\ell \left[\partial/\partial p \right] \right] \pi_\ell, \\
u (1-q^{-1})(t-1) F_{1, 0} &= \sum_{k=1}^\infty \left[ c_k[p],  (q^{-1}-1) \frac{\partial}{\partial p_1}  \right] \frac{\partial}{\partial \pi_k}.
\end{align}

By looking at the commutation relation with the generating function, we obtain
\begin{align}
\Big[ p_1, \widetilde{c}_\ell \left[\partial/\partial p \right] \Big] 
&=  (1-q^{-1}) \widetilde{c}_{\ell-1} \left[\partial/\partial p \right], \\
\left[  \frac{\partial}{\partial p_1}, c_{k} [p] \right]
&= (1-t) c_{k-1}[p],
\end{align}
which implies
\begin{equation}
 E_{1, -1} = u^{-1}\sum_{\ell=1}^\infty \pi_\ell \cdot \widetilde{c}_{\ell-1} \left[\partial/\partial p \right], \qquad
F_{1, 0} = u^{-1} \sum_{k=1}^\infty c_{k-1}[p] \frac{\partial}{\partial \pi_k}.
\end{equation}
From the commutation relation
\begin{equation}
\left[ E_{1,-1}, F_{1,0} \right] = (z^{-1}K_1^{+}(z) - K_1^{-}(z)) \vert_{z^{+1}} 
= K^{-}_{1,0} H_{1,-1} = -u^{-1} H_{1,-1},
\end{equation}
we have
\begin{align}
 - H_{1,-1} &= u \left[ E_{1,-1}, F_{1,0} \right] = u^{-1}
 \left[\sum_{\ell=1}^\infty \widetilde{c}_{\ell-1} \left[\partial/\partial p \right] \cdot \pi_\ell, 
 \sum_{k=1}^\infty c_{k-1}[p] \frac{\partial}{\partial \pi_k} \right] \CR
&=  u^{-1}\left[ 1 + \sum_{k=1}^\infty c_k[p] \widetilde{c}_k[\partial/\partial p] + 
\sum_{\ell=2}^\infty \sum_{k=2}^\infty \widetilde{C}_{\ell,k} \pi_\ell \frac{\partial}{\partial \pi_k} \right], \label{H1}
\end{align}
where we have used $c_0[p]= \widetilde{c}_0 [\partial/\partial p] =1$.

Now we define\footnote{This is motivated by \eqref{First-eigen-2}.}
\begin{equation}
\mathcal{D}_1 := \frac{1}{t-1} \left( 1 + uH_{1,-1} \right) =\mathcal{D}_1^{B} + \mathcal{D}_1^{F}. 
\end{equation}
Then the bosonic part $\mathcal{D}_1^{B}$ is exactly the same as $\mathcal{D}_2^{B}$. 
Note that the fermionic part $\mathcal{D}_1^{F}$  does not involve $\pi_1$ and its derivative. 
In fact the second Hamiltonian should not have the derivative of $\pi_1$,
since its eigenvalue on $\mathcal{M}_{(\frac{1}{2})} = \pi$ vanishes. 
Up to level 2, the total Hamiltonian on the fermionic Macdonald polynomials is
\begin{equation}
\mathcal{D}_1 \sim (q^{-1}-1) \left[ p_1 \frac{\partial}{\partial p_1}+\pi_2\frac{\partial}{\partial \pi_2} \right].
\end{equation}
Since all the fermionic Macdonald polynomials are linear in $p_1$ and $\pi_2$, their eigenvalues are $q^{-1}-1$.
This is consistent, because $\Lambda^{*}=(1)$ for all the super partitions 
$\Lambda = (\frac{3}{2}), (1, \frac{1}{2}), (\frac{3}{2}, \frac{1}{2})$.

\subsection{Commutativity with the super charges}
We have obtained 
\begin{equation}
(t/q)^{\frac{1}{2}}u H_{2,-1}
= 1 + \sum_{k=1}^\infty
c_k[p] \widetilde{c}_k [\partial/\partial p]  +
 \sum_{k=1}^\infty \sum_{\ell=1}^\infty \widetilde{C}_{k, \ell}[p, \partial/ \partial p]
\pi_k \frac{\partial}{\partial \pi_\ell},
\end{equation}
and 
\begin{equation}
 - u H_{1,-1} =  1 + \sum_{k=1}^\infty c_k[p] \widetilde{c}_k[\partial/\partial p] + 
\sum_{\ell=1}^\infty \sum_{k=1}^\infty \widetilde{C}_{k,\ell}[p, \partial/ \partial p]
 \pi_{k+1} \frac{\partial}{\partial \pi_{\ell+1}},
\end{equation}
where
\begin{equation}
\widetilde{C}_{k, \ell}[p, \partial/ \partial p] := \Big[\widetilde{c}_k 
\left[ \partial/ \partial p \right] , c_\ell [p] \Big].
\end{equation}

These results are consistent with the commutation relation of the quantum 
toroidal algebra. Namely we have\footnote{Since $H_{1,-1}$ does not involve $\pi_1$,
$\left[ H_{1,-1}, E_{1,0} \right] = \left[ H_{1,-1}, F_{1,+1} \right] = 0$ is trivial.}
\begin{prp}\label{Prop;commutativity}
\begin{align}
\left[ H_{2,-1}, E_{2,0} \right] &= \left[ H_{2,-1}, F_{2,-1} \right] = 0, \\
\left[ H_{1,-1}, E_{1,-1} \right] &= \left[ H_{1,-1}, F_{1,0} \right] = 0.
\end{align}
\end{prp}

\begin{proof}
We show the relation for $H_{2,-1}$. The relation for $H_{1,-1}$ can be proved similarly. 
Recall that 
\begin{equation}
E_{2,0} = \sum_{m=1}^\infty c_m[p] \frac{\partial}{\partial \pi_m}, \qquad
F_{2,-1} = -\sum_{m=1}^\infty \pi_m \cdot \widetilde{c}_m \left[ \partial/ \partial p \right].
\end{equation}
We have
\begin{align}
& (t/q)^{\frac{1}{2}}u \left[ H_{2,-1}, E_{2,0} \right] \CR
&= \sum_{k=1}^\infty \sum_{m=1}^\infty c_k[p] \Big[ \widetilde{c}_k[\partial/\partial p]
, c_m[p] \Big] \frac{\partial}{\partial \pi_m} \CR
& \qquad +  \sum_{\ell=1}^\infty \sum_{k=1}^\infty \sum_{m=1}^\infty 
\left( c_m[p] \widetilde{C}_{k,\ell}[p, \partial/ \partial p] 
\left[ \pi_k \frac{\partial}{\partial \pi_\ell}, \frac{\partial}{\partial \pi_m} \right] 
+ \left[ \widetilde{C}_{k,\ell}[p, \partial/ \partial p] , c_m[p] \right]  
\pi_k \frac{\partial}{\partial \pi_\ell}\frac{\partial}{\partial \pi_m} \right) \CR
&= \sum_{\ell=1}^\infty \sum_{k=1}^\infty \sum_{m=1}^\infty 
\Big[ \left[ \widetilde{c}_k[\partial/\partial p], c_\ell[p] \right] , c_m[p] \Big]  
\pi_k \frac{\partial}{\partial \pi_\ell}\frac{\partial}{\partial \pi_m}.
\end{align}
Since $\Big[c_\ell[p], c_m[p] \Big]=0$, the Jacobi identity implies 
$\Big[ \left[ \widetilde{c}_k[\partial/\partial p], c_\ell[p] \right] , c_m[p] \Big]$
is symmetric in $\ell$ and $m$. Hence we have the desired result. 
We also have 
\begin{align}
& u \left[ H_{2,-1}, F_{2,-1} \right] \CR
&= \sum_{k=1}^\infty \sum_{m=1}^\infty  \Big[ c_k[p], \widetilde{c}_m [\partial/\partial p] \Big] 
\widetilde{c}_k[\partial/\partial p]\pi_m \CR
& \qquad +  \sum_{\ell=1}^\infty \sum_{k=1}^\infty \sum_{m=1}^\infty 
\left( \widetilde{C}_{k,\ell}[p, \partial/ \partial p] \widetilde{c}_m [\partial/\partial p]
\left[ \pi_k \frac{\partial}{\partial \pi_\ell}, \pi_m \right] 
+ \left[ \widetilde{C}_{k,\ell}[p, \partial/ \partial p] , c_m[p] \right]  
 \pi_m  \pi_k \frac{\partial}{\partial \pi_\ell}\right) \CR
 &= - \sum_{\ell=1}^\infty \sum_{k=1}^\infty \sum_{m=1}^\infty 
 \Big[ \left[ c_\ell[p], \widetilde{c}_{k} [\partial/\partial p]\right], 
 \widetilde{c}_m [\partial/\partial p]  \Big]
 \pi_m  \pi_k \frac{\partial}{\partial \pi_\ell}.
\end{align}
By the same argument above we see that the right hand side vanishes. 
\end{proof}

The proof of Prop \ref{Prop;commutativity} tells that assuming that the bosonic part of $H_{2,-1}$
and $H_{1,-1}$ is the same as the $(q,t)$-inverted form of Ruijsenaars-Macdonald Hamiltonian,
we can fix the fermionic part by the commutativity with the super charges. 

%%%%%%%%%%%%%%%%%%%%%%%%%%%%%%%%%%%%%%%%%%%%%%%%%%%%%%%%%%%%%%%%%%%%%%%%%%%%%

\subsection{Comparison of the super charges with \cite{Alarie-Vezina:2019ohz}}

In \cite{Alarie-Vezina:2019ohz} they defined the super charges
\begin{equation}
\mathcal{Q}_1 := \sum_{i=1}^N \theta_i \tau_i^{-1}, \qquad
\mathcal{Q}_2 := \sum_{i=1}^N A_i(t^{-1}) \frac{\partial}{\partial \theta_i},
\end{equation}
where $\tau_i$ is the $q$-shift operator of $x_i \to qx_i$ and 
\begin{equation}
A_i(t) := \prod_{j \neq i} \frac{tx_i - x_j}{x_i - x_j}.
\end{equation}
One of the super Hamiltonian is given by the anti-commutator
\begin{equation}
\mathcal{D}_{1,N} = t^{N-1} \left[ \mathcal{Q}_1, \mathcal{Q}_2 \right]_{+}.
\end{equation}

\begin{prp}\label{finiteN-supercharge}
\begin{equation}
 u E_{1,-1} = \mathcal{Q}_1, \qquad
u F_{1,0} = t^{N-1} (1-t)\mathcal{Q}_2 + t^N \frac{\partial}{\partial \pi_1}.
\end{equation}
\end{prp}

To see the agreement with the original form of the Ruijsenaars-Macdonald Hamiltonian \eqref{RM-Hamiltonian} 
and to prove the above proposition, we need the following formula.

\begin{lem}[\cite{Awata:1995zk}\footnote{Eqs.(48) and (49).}]
\label{lem:powersum-rep}
\begin{align}\label{q-shift}
\tau_i &=~:\!\exp \left( \sum_{n=1}^\infty (q^n-1)x_i^n \frac{\partial}{\partial p_n} \right)\!: \CR
&= \oint \frac{dz}{z} \sum_{n=0}^\infty x_i^n z^n \exp \left( \sum_{k=1}^\infty (q^k-1)\frac{\partial}{\partial p_k}z^{-k} \right).
\end{align}
\begin{equation}\label{t-identity}
\sum_{i=1}^N \prod_{j \neq i} \frac{tx_i - x_j}{x_i - x_j} \sum_{n=0}^\infty x_i^n z^n = \frac{t^N}{t-1}
\exp \left( \sum_{k=1}^\infty \frac{1-t^{-k}}{k} p_k z^k \right) - \frac{1}{t-1}.
\end{equation}
\end{lem}
\begin{proof}
From $\tau_i p_n = (p_n + (q^n-1)x_i) \tau_i$, we see the representation \eqref{q-shift}. 
To prove the second formula we use Macdonald III-(2.9) and (2.10);
\begin{equation}
\sum_{r=0}^\infty g_r(x,t) u^r = \prod_{i=1}^N \frac{1-x_i tu}{1-x_i u}
= \exp \left( \sum_{n=1}^\infty \frac{1 - t^n}{n} p_n(x) u^n \right), 
\end{equation}
where
\begin{equation}
g_0(x,t)  =1, \qquad g_r(x,t) = (1-t) \sum_{i=1}^N x_i^r \prod_{j \neq i} \frac{x_i -t x_j}{x_i-x_j}, \quad r \geq 1.  
\end{equation}
We also have\footnote{All the residues on the left hand side vanish and hence it only depends on $t$.}
\begin{equation}
\sum_{i=1}^N \prod_{j \neq i} \frac{x_i-tx_j}{x_i-x_j} = \frac{1-t^N}{1-t}. 
\end{equation}
Combining these, we obtain 
\begin{equation}
(1-t) \sum_{i=1}^N \prod_{j \neq i} \frac{x_i-tx_j}{x_i-x_j} \sum_{n=0}^\infty  x_i^n u^n 
= \exp \left( \sum_{n=1}^\infty \frac{1 - t^n}{n} p_n(x) u^n \right) - t^N,
\end{equation}
which implies \eqref{t-identity} after $t \to t^{-1}$.
\end{proof}
With the definitions \eqref{c_k-def} and \eqref{tc_k-def}  we obtain
\begin{cor}
\begin{align}
\label{q-shift-powersum}
\tau_i^{-1} &= \sum_{n=0}^\infty \widetilde{c}_n[\partial/ \partial p] x_i^n,
\\
c_k[p] &= (1-t)  \sum_{i=1}^N \prod_{j \neq i} \frac{x_i-tx_j}{x_i-x_j} x_i^k, \qquad k \geq 1.
\end{align}
\end{cor}

Now we can show Prop \ref{finiteN-supercharge}.
\begin{align}
 u E_{1,-1} &= \sum_{n=1}^\infty \widetilde{c}_{n-1}[p] \pi_n \CR
 &= \sum_{i=1}^N \theta_i \left( \sum_{n=1}^\infty \widetilde{c}_{n-1}[p] x_i^{n-1} \right) \CR
 &= \sum_{i=1}^N \theta_i \tau_i^{-1} = \mathcal{Q}_1.
\end{align}
For the second equation, we compute
\begin{align}
\mathcal{Q}_2 &= t^{1-N} \sum_{i=1}^N \prod_{j \neq i} \frac{x_i -t x_j}{x_i - x_j} 
\left( \sum_{n=1}^\infty x_i^{n-1} \frac{\partial}{\partial \pi_n} \right) \CR
&= t^{1-N} (1-t)^{-1} \sum_{k=1}^\infty c_k[p] \frac{\partial}{\partial \pi_{k+1}}
+ t^{1-N} \frac{1-t^N}{1-t} \frac{\partial}{\partial \pi_{1}}.
\end{align}
Hence
\begin{align}
(1-t) t^{N-1} \mathcal{Q}_2 &= \sum_{k=1}^\infty c_k[p] \frac{\partial}{\partial \pi_{k+1}}
+ (1-t^N)  \frac{\partial}{\partial \pi_{1}} \CR
&= u F_{1,0} - t^N   \frac{\partial}{\partial \pi_{1}}. 
\end{align}

Prop \ref{finiteN-supercharge} implies
\begin{align}
u H_{1,-1} &= - \Big[ uE_{1,-1}, uF_{1,0} \Big] \CR
&= - \Big[ \mathcal{Q}_1, t^{N-1}(1-t)\mathcal{Q}_2 + t^N \frac{\partial}{\partial \pi_1} \Big] \CR
&= (t-1) \mathcal{D}_{1,N} - t^N.
\end{align}
Thus, we recover \eqref{finite-H}. 

In \cite{Alarie-Vezina:2019ohz} the Hamiltonian $\mathcal{D}_{1,N}$ is expressed as an anti-commutator of the super charges, but 
the second Hamiltonian $\mathcal{D}_{2,N}$ is not. In this paper we have shown that both Hamiltonians 
are in fact anti-commutator of the super charges
based on the structure of the quantum toroidal algebra. 

%%%%%%%%%%%%%%%%%%%%%%%%%%%%%%%%%%%%%%%%%%%%%%%%%%%%%%%%%%%%%%%%%%%%%%%%%%%%%%%%%%%%%%

\subsection{Comparison with \cite{Galakhov:2025phf} }
\label{subsec:VS-Galakhov}

In \cite{Galakhov:2025phf} the four Hamiltonians for the super Macdonald polynomials 
were proposed. Among them two Hamiltonians are bi-linear in fermions. Let us look at them (See eq.(100) in  \cite{Galakhov:2025phf});
\begin{align}
\widehat{\mathcal{H}}_{2}^{-} &= \oint \frac{dw}{w} \langle \varnothing \vert V^{(+)}(w)  V^{(-)}(w) \vert \varnothing \rangle_B, 
\\
\widehat{\mathcal{H}}_{1}^{-} &= \oint \frac{dw}{w} \langle \varnothing \vert \widetilde{V}^{(+)}(w)  \widetilde{V}^{(-)}(w) \vert \varnothing \rangle_B, 
\end{align}
where\footnote{We change the convention $(q^2, t^2) \to (q,t)$ from \cite{Galakhov:2025phf}
and the notation of the fermionic power sum from $\theta_k$ to $\pi_k$.}
\begin{align}
V^{(+)}(w) &= \exp \left[ \sum_{k=1}^\infty \left( \frac{1-t^k}{k}p_k w^{2k} s^{-k} + (1-t)w^{2k-1}\pi_k \nu \right) \right],
\\
V^{(-)}(w) &= \exp \left[ \sum_{k=1}^\infty \left( (q^{-k} -1) \frac{\partial}{\partial p_k} w^{-2k}
+ (q^{-1}-1) w^{-2k+1} \frac{\partial}{\partial \pi_k} \nu^\dagger s^{k-1} \right) \right].
\end{align}
$s$ is an auxiliary bosonic operator and $\nu$ and $\nu^\dagger$ are anti-commuting fermionic operators. 
We have $\widetilde{V}^{(+)}(w)= V^{(+)}(w)$ and $\widetilde{V}^{(-)}(w)$ is defined by 
replacing the factor $s^{k-1}$ in the last term of $V^{(-)}(w)$ with $s^{k-2}$. 
The rule of computing the vacuum expectation values is
\begin{equation}\label{VEV}
\langle \varnothing \vert s^a \vert \varnothing \rangle =1, \qquad
\langle \varnothing \vert \nu \nu^\dagger s^b \vert \varnothing \rangle =
\begin{cases}
\frac{1-(qt)^{b+1}}{q^b(1-qt)}, \quad b \geq 0, \\
0, \quad \hbox{otherwise}.
\end{cases}
\end{equation}
Note that  the operator $s$ appears only in the coefficients of $p_k$ and $\frac{\partial}{\partial \pi_k}$. 
Hence, the selection rule in \eqref{VEV} means we cannot have bosonic rows of higher degree
after removing a fermionic row. Expanding the exponentials in the anti-commuting variables $\nu$ and $\nu^\dagger$,
we have
\begin{align}
V^{(+)}(w) &= \left( 1  +  (1-t) \sum_{k=1}^\infty w^{2k-1}\pi_k \nu \right)
\cdot \exp \sum_{k=1}^\infty \left( \frac{1-t^k}{k}p_k w^{2k} s^{-k} \right),
\\
V^{(-)}(w) &= \left( 1+ (q^{-1}-1) \sum_{k=1}^\infty  w^{-2k+1} \frac{\partial}{\partial \pi_k} \nu^\dagger s^{k-1}\right)
\cdot \exp \sum_{k=1}^\infty \left( (q^{-k} -1) \frac{\partial}{\partial p_k} w^{-2k}\right).
\end{align}

Then the fermion bi-linear terms of $\hat{\mathcal{H}}_{2}^{-}$ are
\begin{equation}\label{F-bi-linear}
\hat{\mathcal{H}}_{2}^{-,F} = (1-t)(q^{-1}-1) \sum_{k,\ell =1}^\infty 
C_{k, \ell}[p, \partial/ \partial p] \left( \pi_k \frac{\partial}{\partial \pi_\ell} \right),
\end{equation}
with
\begin{align}\label{bi-linear-coeff}
C_{k,\ell}[p, \partial/ \partial p]
&= \sum_{n,m=0}^\infty \oint \frac{dw}{w} \langle \varnothing \vert w^{2(k-\ell+n-m)}s^{\ell-1-n}
\cdot c_n[p] \cdot \widetilde{c}_m \left[ \partial/ \partial p \right] \cdot
\nu \nu^\dagger\vert \varnothing \rangle_B \CR
&= \sum_{n=0}^{\ell-1} [[\ell-n]]_{(q^{-1},t)} \cdot c_n[p] \cdot \widetilde{c}_{k-\ell+n} \left[ \partial/ \partial p \right],
\end{align}
where we have used \eqref{VEV} and the fact that $\widetilde{c}_{m}=0$ for $m <0$.  
Thus we see that \eqref{bi-linear-coeff} agrees with \eqref{Fbi-linear}.

Similarly we have
\begin{equation}
\hat{\mathcal{H}}_{1}^{-,F} = \sum_{k,\ell =1}^\infty  \widetilde{C}_{k, \ell}[p, \partial/ \partial p] 
\left( \pi_{k+1} \frac{\partial}{\partial \pi_{\ell+1}} \right).
\end{equation}
On the other hand the difference of $\hat{\mathcal{H}}_{1}^{-}$ and $\hat{\mathcal{H}}_{2}^{-}$ is that
the factor $s^{k-1}$ in $V^{(-)}(w)$ is replaced by $s^{k-2}$. 
Hence, the fermion bi-linear terms of $\hat{\mathcal{H}}_{1}^{-}$ are
\begin{equation}
\hat{\mathcal{H}}_{1}^{-,F} = (1-t)(q^{-1}-1) \sum_{k,\ell =1}^\infty 
C^\prime_{k, \ell}[p, \partial/ \partial p] \left( \pi_k \frac{\partial}{\partial \pi_\ell} \right),
\end{equation}
with
\begin{align}
C^\prime_{k,\ell}[p, \partial/ \partial p]
&= \sum_{n,m=0}^\infty \oint \frac{dw}{w} \langle \varnothing \vert w^{2(k-\ell+n-m)}s^{\ell-2-n}
\cdot c_n[p] \cdot \widetilde{c}_m \left[ \partial/ \partial p \right] \cdot
\nu \nu^\dagger\vert \varnothing \rangle_B \CR
&= \sum_{n=0}^{\ell-2} [[\ell-n-1]]_{(q^{-1},t)} \cdot c_n[p] \cdot \widetilde{c}_{k-\ell+n} \left[ \partial/ \partial p \right].
\end{align}
By the shift of indices $k \to k+1$ and $\ell \to \ell \to \ell+1$, we see the agreement.

%%%%%%%%%%%%%%%%%%%%%%%%%%%%%%%%%%%%%%%%%%%%%%%%%%%%%%%%%%%%%%%%%%%%%%

\section{Pieri formula and the Hamiltonians $H_{1, +1}$ and $H_{2,+1}$}
\label{section:positive-modes}

\subsection{Hamiltonian $H_{2,+1}$} 

We have seen that the eigenvalues of $1- u(t/q)^{\frac{1}{2}} H_{2,-1}$ and $1 + u^{-1}(q/t)^{\frac{1}{2}} H_{2,+1}$
are related by the involution $(q,t) \to (q^{-1}, t^{-1})$. 
Since the bosonic Macdonald polynomials are invariant under the involution, 
the bosonic part of these Hamiltonians may be related by the involution $(q,t) \to (q^{-1}, t^{-1})$.
But the fermionic part is not.
Let us compute $1 + u^{-1}(q/t)^{\frac{1}{2}} H_{2,+1}$ by using the commutation relation;
\begin{align}
\left[ E_{2, +1}, F_{2,-1} \right] &= z^{-1}(z K_2^{+}(z) - K_2^{-}(z) \vert_{z^{-1}} = K_{2,0}^{+}H_{2,1} - K_{2,0}^{-} \CR
&= H_{2,+1} + (t/q)^{\frac{1}{2}}u = (t/q)^{\frac{1}{2}}u (1+ u^{-1}(q/t)^{\frac{1}{2}}H_{2,+1}).
\end{align}
Compare it with
\begin{equation}
\left[ E_{2, 0}, F_{2,-1} \right] = 1 - (t/q)^{\frac{1}{2}} u H_{2,-1},
\end{equation}
which we used to compute $H_{2,-1}$.

From the computation of the Pieri rule of the higher modes (see \eqref{E21-expansion}
in Appendix \ref{App:Higher-modes}), we find
\begin{align}
E_{2,+1} &= u(t/q)^{\frac{1}{2}} \cdot q (t^{-1}-1) \left( p_1 \frac{\partial}{\partial \pi_1} 
+ \frac{1}{2}q \left( (1- t^{-1}) p_1^2 + (1+ t^{-1}) p_2 \right) \frac{\partial}{\partial \pi_2} \right. \CR
& \qquad \left. 
+ \frac{1}{2}(q-1) \left( (1- t^{-1}) p_1^2 + (1+ t^{-1}) p_2 \right) 
\frac{\partial}{\partial p_1}\frac{\partial}{\partial \pi_1} 
%+ \epsilon_1 p_1 \frac{\partial}{\partial p_1}\frac{\partial}{\partial \pi_2}
+ \cdots \right). 
\end{align}
In contrast to \eqref{E-rep} a new feature here is the last term in $E_{2,+1}$ which involves the derivative $\frac{\partial}{\partial p_1}$. 
In fact the last term is necessary to reproduce the Pieri rule \eqref{E2nPieri-1} -- \eqref{E2nPieri-3}.
%
%Hence, 
By using \eqref{F-rep} the bosonic part of the commutator is
\begin{align}
& \left[ E_{2, +1}, F_{2,-1} \right]\vert_{\pi_k=0} \CR
&= u(t/q)^{\frac{1}{2}} \cdot (t^{-1}-1) \left( (q-1) p_1 \frac{\partial}{\partial p_1} 
 + \frac{1}{2}\left( (1- t^{-1}) p_1^2 + (1+ t^{-1}) p_2 \right)  (q^2-1)\frac{\partial}{\partial p_2} \right. \CR
& \qquad \left. - \frac{1}{4}\left( (1- t^{-1}) p_1^2 + (1+ t^{-1}) p_2 \right)(q-1)^2 \frac{\partial^2}{\partial p_1^2} 
 + \frac{1}{2} \left( (1- t^{-1}) p_1^2 + (1+ t^{-1}) p_2 \right)(q-1)^2 \frac{\partial^2}{\partial p_1^2} \right)
 + \cdots, 
\end{align}
which is the $(q,t)$-inverted version of $\left[ E_{2, 0}, F_{2,-1} \right]\vert_{\pi_k=0}$. 
The last term in $E_{2,+1}$ is responsible for the sign flip of the $\frac{\partial^2}{\partial p_1^2}$ term.

We can also compute the fermion bi-linear terms of $H_{2,-1}$ and up to level $\frac{3}{2}$ we obtain
\begin{align}
\overline{\mathcal{D}}_2 
&\simeq (q-1) p_1 \frac{\partial}{\partial p_1} 
+ (q-1) \left( 1 + (q-1)(1-t^{-1})p_1 \frac{\partial}{\partial p_1} \right)\pi_ 1 \frac{\partial}{\partial \pi_1} \CR
& \qquad + q(q-1)(1-t^{-1})p_1 \pi_1 \frac{\partial}{\partial \pi_2} 
+ t^{-1}(q-1)^2 \frac{\partial}{\partial p_1} \pi_2\frac{\partial}{\partial \pi_1} \CR
& \qquad \qquad + \frac{1}{2} \left( (q^2-1)(1+t^{-1}) - (1-q)^2 (1-t^{-1}) \right) \pi_2\frac{\partial}{\partial \pi_2}, 
\end{align}
which should be compared with
\begin{align}
\mathcal{D}_2 & \simeq (q^{-1}-1) p_1\frac{\partial}{\partial p_1} +
(q^{-1}-1) \pi_1\frac{\partial}{\partial \pi_1} + (q^{-1}-1)^2  \frac{\partial}{\partial p_1} \cdot 
\pi_2 \frac{\partial}{\partial \pi_1} \CR
& \qquad + (1-t)(q^{-1}-1) p_1 \cdot \pi_1 \frac{\partial}{\partial \pi_2}
+ \frac{1}{2} \left( (1+t)(q^{-2}-1) + (1-t)(q^{-1}-1)^2 \right)\pi_2 \frac{\partial}{\partial \pi_2}.
\end{align}
Up to level two the only super Macdonald polynomial which is not invariant under $(q,t) \to (q^{-1}, t^{-1})$ is
\begin{equation}
\mathcal{M}_{(\frac{3}{2})} =  \frac{q(1-t)}{1-qt} p_1 \pi_1 + \frac{1-q}{1-qt} \pi_2 
\sim q(1-t)p_1 \pi_1 + (1-q) \pi_2.
\end{equation}
It is instructive to check that both $\mathcal{D}_2$ and $\overline{\mathcal{D}}_2$ give the expected eigenvalues $q^{\mp 2}-1$. 
\begin{align}
\mathcal{D}_2 \cdot \mathcal{M}_{(\frac{3}{2})}
&\sim (q^{-1}-1)(1-t)\left(2q + (1-q) \right) p_1 \pi_1 + (q^{-1}-1)^2q(1-t)\pi_2 \CR
& \qquad + \frac{1}{2}((1+t)(q^{-2}-1) + (1-t)(q^{-1}-1)^2) (1-q) \pi_2 \CR
&= (q^{-2}-1)q(1-t)p_1 \pi_1 + (q^{-1}-1)\left( (1-t) + q^{-1}+ t \right) (1-q) \pi_2 \CR
&= (q^{-2}-1)(q(1-t)p_1 \pi_1 + (1-q) \pi_2). 
\end{align}
On the other hand 
\begin{align}
\overline{\mathcal{D}}_2 \cdot \mathcal{M}_{(\frac{3}{2})}
&\sim \left( 2(q-1) + (q-1)^2 (1-t^{-1}) + t^{-1}(q-1)^2 \right) q(1-t)p_1 \pi_1 \CR
& \qquad + \left( - (q-1)q(t^{-1}-1) + \frac{1}{2} \left( (q^2-1)(1+t^{-1}) - (1-q)^2 (1-t^{-1}) \right)\right) (1-q) \pi_2\CR
&= (q-1)(2+(q-1)) q(1-t)p_1 \pi_1 \CR
& \qquad + \left( q(q-1)(1-t^{-1}) + (q-1)  + t^{-1}q(q-1) \right) (1-q) \pi_2 \CR
&= (q^2-1)(q(1-t)p_1 \pi_1 + (1-q) \pi_2). 
\end{align}
%
%
%%%%%%%%%%%%%%%%%%%%%%%%%%%%%%%%%%%%%%%%%%%%%%%%%%%%%%%%%%%%%%%%%%%%%%%%%%%%%

\subsection{Hamiltonian $H_{1,+1}$}

$H_{1,+1}$ is given by the following  (anti)-commutators;
\begin{align}\label{H11-com}
\left[ E_{1,0}, F_{1,2} \right] &= z^2(z^{-1} K_1^{+}(z) - K_1^{-}(z) \vert_{z^{0}} \CR
&= K_{1,0}^{+} H_{1,+1} = H_{1,1},
\\
\left[ E_{1,1}, F_{1,1} \right] &= z (z^{-1} K_1^{+}(z) - K_1^{-}(z) \vert_{z^{-1}} \CR
&= K_{1,0}^{+} H_{1,+1} = H_{1,1} .  
\end{align}
Recall that $E_{1,0}=\pi$ and $F_{1,1}= \frac{\partial}{\partial \pi_1}$.
The computation of the eigenvalues motivates us to define
\begin{equation}
\overline{\mathcal{D}}_1 = \frac{1}{t^{-1}-1}(1- u^{-1}H_{1,1}). 
\end{equation}
First of all
\begin{equation}
\left[ E_{1,1}, F_{1,1} \right] = \left[ u(q + t^{-1}(1-q)) \pi_1 + \cdots , \frac{\partial}{\partial \pi_1}\right]
= u(q + t^{-1}(1-q)) + \cdots. 
\end{equation}
By the result \eqref{E11-level2}, which is valid up to level two;
\begin{equation}
\overline{\mathcal{D}}_1 = (q-1)p_1 \frac{\partial}{\partial p_1} + \cdots ,
\end{equation}
which gives the leading term of the expected result. Note that we cannot have such terms as 
$\pi_ 1 \frac{\partial}{\partial \pi_1}$, since $H_{1,1}$ is the anti-commutator with $E_{1,0}=\pi_1$.

Secondary,
\begin{align}\label{F21-level2}
F_{1,2} &= u \left( \frac{\partial}{\partial \pi_1} + q(1-t^{-1}) p_1 \frac{\partial}{\partial \pi_2} \right. \CR
& \qquad \left. + (1-t^{-1})(q-1) p_1 \frac{\partial}{\partial p_1}\frac{\partial}{\partial \pi_1}
- (1-t^{-1})(q-1) \pi_2 \frac{\partial}{\partial \pi_1}\frac{\partial}{\partial \pi_2}  + \cdots \right),
\end{align}
implies
\begin{equation}
u^{-1}H_{1,+1} = 1 + (1-t^{-1})(q-1) \left( p_1 \frac{\partial}{\partial p_1} +  \pi_2 \frac{\partial}{\partial \pi_2} \right).
\end{equation}
Hence,
\begin{equation}
\overline{\mathcal{D}}_1 = (q-1) \left( p_1 \frac{\partial}{\partial p_1} +  \pi_2 \frac{\partial}{\partial \pi_2} \right),
\end{equation}
which is the expected result up level $\frac{3}{2}$. It does not involve $\pi_1$ as is the case of $\mathcal{D}_1$.

Since the bosonic part of $H_{1,+1}$ is the $(q,t)$-inversion of the bosonic part of $H_{1,-1}$,
by taking account of \eqref{H11-com} we should have
\begin{equation}
F_{1,2} = u \sum_{k=0}^\infty  c_k^\vee[p] \widetilde{c}_k^\vee \left[ \partial/ \partial p \right] 
\frac{\partial}{\partial \pi_1} + \cdots,
\end{equation}
where
\begin{align}
\label{c-check-gen}
\sum_{k=0}^\infty c_k^\vee[p] z^k &:= \exp \left(\sum_{r=1}^\infty \frac{1-t^{-r}}{r} p_r z^r \right),
\\
\label{tc-check-gen}
\sum_{k=0}^\infty \widetilde{c}_k^\vee \left[ \partial/ \partial p \right] z^{-k}
&:= \exp \left( \sum_{n=1}^\infty (q^{n}-1) \frac{\partial}{\partial p_n} z^{-n} \right).
\end{align}
We can check up to level $\frac{5}{2}$ this is consistent with 
the following Pieri rules;\footnote{One may try to compute the higher level terms of $E_{1,+1}$.
But the Pieri rule is more involved.}
\begin{align*}
F_{1,n}  \cdot (1-q^2t)\mathcal{M}_{(\frac{5}{2})} &= u^{n-1} q^{2n}(1-t)\mathcal{M}_{(2)}, \\
F_{1,n} \cdot \mathcal{M}_{(2, \frac{1}{2})} &= u^{n-1}t^{1-n} \mathcal{M}_{(2)}, \\
F_{1,n} \cdot (1-qt^2)\mathcal{M}_{(\frac{3}{2},1)} &= u^{n-1} q^n (1-t^2) \mathcal{M}_{(1,1)}, \\
F_{1,n} \cdot \mathcal{M}_{(1, 1, \frac{1}{2})} &= u^{n-1}t^{2(1-n)} \mathcal{M}_{(1,1)}. 
\end{align*}
In fact these Pieri rules with $n=2$ imply that
\begin{align}\label{F12-level5/2}
u^{-1} F_{1,2} 
&= \left(1+ c_1^\vee[p] \widetilde{c}_1^\vee \left[ \partial/ \partial p \right]
+ c_2^\vee[p] \widetilde{c}_2^\vee \left[ \partial/ \partial p \right] 
+ (1-t^{-1})(q-1) \pi_2  \frac{\partial}{\partial \pi_2} \right) \frac{\partial}{\partial \pi_1} \CR
& \qquad \qquad + \left(c_1^\vee[p] + c_2^\vee[p]\widetilde{c}_1^\vee \left[ \partial/ \partial p \right]
\right) q \frac{\partial}{\partial \pi_2}
+ c_2^\vee[p] \cdot q^2 \frac{\partial}{\partial \pi_3} +\cdots . 
\end{align}

By taking the anti-commutator with $E_{1,0} = \pi_1$, we obtain the Hamiltonian $H_{1,+1}$ up to level 2;
\begin{align}
u^{-1}H_{1,+1} &= 1 + (1-t^{-1})(q-1) \left( p_1 \frac{\partial}{\partial p_1} 
+  \pi_2 \frac{\partial}{\partial \pi_2} \right) \CR
& \qquad \qquad + \frac{1}{2}\left( (1-t^{-2}) p_2 + (1-t^{-1})^2 p_1^2 \right)
\left( (q^2-1) \frac{\partial}{\partial p_2} + \frac{1}{2} (q-1)^2 \frac{\partial^2}{\partial p_1^2} \right). 
\end{align}

As a consistency check one can show that \eqref{F12-level5/2} anti-commutes with 
$E_{2,0}= \displaystyle{\sum_{k=1}^\infty} c_k[p] \frac{\partial}{\partial \pi_k}$;
\begin{equation}\label{E20-F12-com}
\left[ E_{2,0}, F_{1,2} \right]_{+} =0. 
\end{equation}
In fact the fermion cubic term 
$(1-t^{-1})(q-1) \pi_2  \frac{\partial}{\partial \pi_2} \frac{\partial}{\partial \pi_1}$
is necessary for the validity of \eqref{E20-F12-com}. 

%%%%%%%%%%%%%%%%%%%%%%%%%%%%%%%%%%%%%%%%%%%%%%%%%%%%%%%%%%%%%%%%%%%%%%

\subsection{Useful lemma}

Let us recall that $c_k^\vee[p]$ and $\widetilde{c}_k^\vee\left[ \partial/ \partial p \right] $ are defined 
by the generating functions \eqref{c-check-gen} and \eqref{tc-check-gen}. 
%\begin{align}
%\sum_{k=0}^\infty c_k^\vee[p] z^k &= \exp \left(\sum_{r=1}^\infty \frac{1-t^{-r}}{r} p_r z^r \right),
%\\
%\sum_{k=0}^\infty \widetilde{c}_k^\vee \left[ \partial/ \partial p \right] z^{-k}
%&= \exp \left( \sum_{n=1}^\infty (q^{n}-1) \frac{\partial}{\partial p_n} z^{-n} \right).
%\end{align}
For $n >0$ we have the bi-linear relations
\begin{equation}\label{bi-linear}
\sum_{\substack{k+\ell=n,\\ k, \ell \geq 0}} t^{-k} c_k [p] c_\ell^\vee [p] =0, \qquad
\sum_{\substack{k+\ell=n,\\ k, \ell \geq 0}} q^k \widetilde{c}_k \left[ \partial/ \partial p \right]
\widetilde{c}_\ell^\vee \left[ \partial/ \partial p \right] =0.
\end{equation}
In particular the second relation implies 
\begin{equation}\label{bi-linear2}
\widetilde{c}_n^\vee \left[ \partial/ \partial p \right]
= - \sum_{\ell=1}^{n-1} q^{n-\ell} \widetilde{c}_{n - \ell} \left[ \partial/ \partial p \right]
\widetilde{c}_\ell^\vee \left[ \partial/ \partial p \right].
\end{equation}

%%%%%%%%%%%%%%%%%%%%%%%%%%%%%%%%%%%%%%%%%%%%%%%%%%%%%%%%%%%%%%%%%%%%%%%%%%%%%%%%%%%%%%%%%%%%%%%%%%

The following summation formula is derived from the bi-linear relations \eqref{bi-linear}
and plays an important role in the computations in the next section. 
\begin{lem}\label{C-lemma}
For any $k \geq 0$,
\begin{align}
\label{C-sum1}
&\sum_{n=0}^\infty  t^{-n} c_n [p]  C_{k-n}[p, \partial/\partial p] =0, \\
\label{C-sum2}
&\sum_{n=0}^\infty C_{n-k}[p, \partial/\partial p]  q^n \widetilde{c}_n [\partial/\partial p] =0,
\end{align}
where
\begin{equation}\label{C-def}
C_\ell [p, \partial/\partial p] := \sum_{n=0}^\infty c_{n +\ell}^\vee[p]~\widetilde{c}_{n}^\vee \left[ \partial/ \partial p \right].
\end{equation}
\end{lem}
\begin{proof}
By \eqref{C-def} the coefficient of $\widetilde{c}_m^\vee [\partial/\partial p]$ in \eqref{C-sum1} is
$$
\sum_{n=0}^{m+k}  t^{-n} c_n[p] c^\vee_{k-n+m}[p].
$$
Similarly the coefficient of $c_m^\vee [p]$ in \eqref{C-sum2} is
$$
\sum_{n=0}^{m+k} \widetilde{c}_{m+k-n}[\partial/\partial p]q^n\widetilde{c}_n [\partial/\partial p].
%\widetilde{c}_{m+k-n}[\partial/\partial p]q^n\widetilde{c}_n [\partial/\partial p]
$$
Both of them  vanish due to \eqref{bi-linear}. 
\end{proof}
%%%%%%%%%%%%%%%%%%%%%%%%%%%%%%%%%%%%%%%%%%%%%%%%%%%%%%%%%%%%%%%%%%%%%%%%%%%%%%%%%%%%%%%%%%%%%%%%%%%%%%%%%%%%%%%%%

\section{Integral formula for $E_{2, 1}$ and $F_{1, 2}$}
\label{section;integral-formula}

Motivated by the integral formula of the Hamiltonians $H_{2, +1}$ and $H_{1,+1}$ presented in \cite{Galakhov:2025phf},
we conjecture the following integral representation of $E_{2,1}$ and $F_{1,2}$;
\begin{con}
\label{E21-F12-con}
\begin{align}
u^{-1} (q/t)^{\frac{1}{2}} E_{2,1} &= 
- \sum_{k=1}^\infty \oint \frac{dw}{w} w^{-2k} V_B^{(-)}(w) V_B^{(+)}(w)  
\langle \varnothing \vert \widetilde{V}_F^{(-)}(w) \widetilde{V}_F^{(+)}(w) \vert \varnothing\rangle_{F} q^{k} \frac{\partial}{\partial \pi_k},
\\
u^{-1} F_{1,2} &=
\sum_{k=1}^\infty\oint \frac{dw}{w}  w^{-2k+2} V_B^{(-)}(w) V_B^{(+)}(w) 
\langle \varnothing \vert \widetilde{V}_F^{(-)}(w) \widetilde{V}_F^{(+)}(w) \vert \varnothing\rangle_{F} q^{k-1} \frac{\partial}{\partial \pi_k}.
\end{align}
\end{con}
In the above integral representation of $E_{2,1}$ and $F_{1,2}$, the bosonic vertex operator
\begin{equation}
V_B^{(-)}(w) := \exp \left( \sum_{k=1}^\infty \frac{1-t^{-k}}{k} p_k w^{2k} \right), 
\quad 
V_B^{(+)}(w) := \exp \left( \sum_{k=1}^\infty  (q^k -1) \frac{\partial}{\partial p_k} w^{-2k} \right),
\end{equation}
is the same as what was employed in the free field realization of the Ruijsenaars-Macdonald Hamiltonian \cite{Awata:1994xd},\cite{Awata:1995zk}. 
The fermionic vertex operator $\widetilde{V}_F^{(\pm)}(w)$ is
defined by\footnote{We assume $\psi_m, \psi_m^\dagger$ and $\pi_k$ are anti-commuting and make the products \lq\lq normal ordered\rq\rq.}
\begin{align}
% V_F^{(-)}(w) &= \exp \left(
% (1-t^{-1}) \sum_{k=1}^\infty t^{1-k} w^{2k-1} \pi_k \cdot \sum_{m=0}^{k-1}t^m \psi_m \right),
%\\
% V_F^{(+)}(w)&= \exp \left(
% (q-1) \sum_{k=1}^\infty q^{k-1} w^{-2k+1}  \left(\sum_{m=0}^{k-1} q^{-m} \psi_m^\dagger \right) \frac{\partial}{\partial \pi_k} \right), 
%\\
 \widetilde{V}_F^{(-)}(w) &= \exp \left(
 (1-t^{-1}) \sum_{k=2}^\infty t^{2-k} w^{2k-1} \pi_k \cdot \sum_{m=0}^{k-2}t^m \psi_m \right),
\\
 \widetilde{V}_F^{(+)}(w)&= \exp \left(
 (q-1) \sum_{k=2}^\infty q^{k-2} w^{-2k+1} \left( \sum_{m=0}^{k-2} q^{-m} \psi_m^\dagger \right) \frac{\partial}{\partial \pi_k}  \right).
\end{align} 
The vacuum expectation values are computed with respect to the charged fermions $\psi_m, \psi_m^\dagger$ with
the anti-commutation relation $\{ \psi_m, \psi_n^\dagger \} = \delta_{mn}$.  
The fermion vacuum is defined by $\psi_m \vert \varnothing \rangle_F =0$ and we have
$\langle \varnothing \vert \psi_m \psi_n^\dagger \vert \varnothing \rangle_F = \delta_{mn}$.

\begin{rmk}
Conjecture \ref{E21-F12-con} tells that by the shift $k \to k+1$ of the powers of $q$ and $w$, 
$F_{1,2}$ agrees with $E_{2,+1}$.\footnote{We 
have seen a similar shift of indices in the relation of $H_{1, -1}$ and $H_{2,-1}$.}  
\end{rmk}

\begin{prp}
We have the anti-commutation relations
\begin{equation}
\left[ E_{2,1}, F_{1,0} \right]_{+} = \left[ F_{1,2}, E_{2,0} \right]_{+} = 0.
\end{equation}
\end{prp}

We can confirm the anti-commutation relations by noting the following relations;
\begin{lem}\label{lemma-EF}
\begin{align}
\label{com-forE21}
\left[ H_{1,+1}, E_{2,0} \right] &= (q^{\frac{1}{2}} - q^{-\frac{1}{2}})(t^{\frac{1}{2}} - t^{-\frac{1}{2}}) E_{2,1}, \\
\label{com-forF12}
\left[ H_{2,+1}, F_{1,1} \right] &= (q^{\frac{1}{2}} - q^{-\frac{1}{2}})(t^{\frac{1}{2}} - t^{-\frac{1}{2}}) F_{1,2},
\end{align}
where $H_{1,+1}$ and $H_{2,+1}$ are the Hamiltonians obtained in \cite{Galakhov:2025phf} (see Appendix \ref{App:Integral-formula}).
\end{lem}

In fact by \eqref{com-forE21} and \eqref{com-forF12} we have
\begin{align}
\left[ E_{2,1}, F_{1,0} \right] &\sim \Big[ \left[H_{1,+1}, E_{2,0} \right], F_{1,0} \Big] \CR
&= - \Big[ \left[E_{2,0}, F_{1,0} \right],  H_{1,+1}\Big] -\Big[ \left[ F_{1,0}, H_{1,+1}\right],  E_{2,0}\Big].
\end{align}
and 
\begin{align}
\left[ F_{1,2}, E_{2,0} \right] &\sim \Big[ \left[H_{2,+1}, F_{1,1} \right], E_{2,0} \Big] \CR
&= - \Big[ \left[ F_{1,1}, E_{2,0} \right], H_{2,+1},\Big] - \Big[ \left[E_{2,0}, H_{2,+1} \right], F_{1,1} \Big].
\end{align}
By looking at explicit formulas for $E_{2,0}, F_{1,0}$ and $F_{1,1}$ we see $\left[E_{2,0}, F_{1,0} \right] = \left[E_{2,0}, F_{1,1} \right] =0$.
In addition Prop \ref{Prop:commutativity+} tells $\left[ F_{1,0}, H_{1,+1}\right] = \left[E_{2,0}, H_{2,+1} \right] =0$. 

%%%%%%%%%%%%%%%%%%%%%%%%%%%%%%%%%%%%%%%%%%%%%%%%%%%%%%%%%%%%%%%%%%%%%%%%%%%%%%%%%%%%%%%%%%%%%%%%%%%%%%%%%%%%%%%%%%%%%%%%%%%%%%%

\begin{proof}

We prove Lemma \ref{lemma-EF}.

Since $F_{1,1} = \frac{\partial}{\partial \pi_1}$, the commutation relation \eqref{com-forF12} is easier to compute.
In fact from \eqref{H_2+-formula}, we have
\begin{align}
& u^{-1} (q-1)(1-t^{-1}) F_{1,2} \CR
&= \oint \frac{dw}{w} V_B^{(-)}(w) V_B^{(+)}(w)  
\langle \varnothing \vert (1-t^{-1}) w \psi_0 \widetilde{V}_F^{(-)}(w) V_F^{(+)}(w) \vert \varnothing \rangle_{F} \CR
&= (q-1)(1-t^{-1}) \sum_{k=1}^\infty \oint \frac{dw}{w} w^{-2k+2} V_B^{(-)}(w) V_B^{(+)}(w)  
\langle \varnothing \vert \widetilde{V}_F^{(-)}(w) \widetilde{V}_F^{(+)}(w) 
\vert \varnothing \rangle_{F} q^{k-1} \frac{\partial}{\partial \pi_k}.
\end{align}

On the other hand, the check of \eqref{com-forE21} is more involved.
The commutation relation \eqref{com-forE21} means
\begin{equation}
(1-q^{-1})(t-1) u^{-1} (q/t)^{\frac{1}{2}} E_{2,1} 
= \left[ u^{-1} H_{1,+1}, \sum_{n=1}^\infty c_n[p] \frac{\partial}{\partial \pi_n} \right],
\end{equation}
which can be computed in a similar way to the proof of Prop. \ref{Prop:commutativity+}.
Here $H_{2,+1}$ is simply replaced by $H_{1,+1}$. Hence, replacing $V_F^{(\pm)}$ in \eqref{1st-term}
by $\widetilde{V}_F^{(\pm)}$ we have
\begin{align}\label{bosonic-com}
& (q-1)(1-t) \oint \frac{dw}{w} \sum_{n=1}^\infty 
 \left( \sum_{m=0}^{n-1} t^{n-m-1} b_{n-m}(q,t) w^{-2(n-m)} c_{m} [p] \right) \CR
& \qquad  \times V_B^{(-)}(w) V_B^{(+)}(w) \langle \varnothing \vert
\widetilde{V}_F^{(-)}(w) \widetilde{V}_F^{(+)}(w) \vert \varnothing \rangle_{F}\frac{\partial}{\partial \pi_n} \CR
&=  (q-1)(1-t) \oint \frac{dw}{w} \sum_{n=1}^\infty \sum_{k=1}^n
\sum_{m=0}^{k-1} q^{n-k} t^{k-m-1} w^{-2(n-m)} c_{m} [p] \CR
& \qquad  \times V_B^{(-)}(w) V_B^{(+)}(w) \langle \varnothing \vert
\widetilde{V}_F^{(-)}(w) \widetilde{V}_F^{(+)}(w) \vert \varnothing \rangle_{F}\frac{\partial}{\partial \pi_n}.
\end{align}
The crucial difference from the case of Prop. \ref{Prop:commutativity+}
is the range of the summation in the anti-commutation relation of fermions,
which is 
\begin{align}\label{fermionic-com}
& (t^{-1}-1) \oint \frac{dw}{w} V_B^{(-)}(w) V_B^{(+)}(w) \sum_{n=2}^\infty c_n[p]  t^{2-n} w^{2n-1} 
\langle \varnothing \vert\left( \sum_{m=0}^{n-2} t^m \psi_m \right)
\widetilde{V}_F^{(-)}(w) \widetilde{V}_F^{(+)}(w) \vert \varnothing \rangle_{F} \CR
&= (q-1)(1-t) \oint \frac{dw}{w} V_B^{(-)}(w) V_B^{(+)}(w) \sum_{n=2}^\infty c_{n} [p] t^{1-n} w^{2(n-k)} \sum_{m=0}^{n-2} (t/q)^m \CR
& \qquad \qquad \times \langle \varnothing \vert
\widetilde{V}_F^{(-)}(w) \widetilde{V}_F^{(+)}(w) \vert \varnothing \rangle_{F} 
\sum_{k=m+2}^\infty q^{k-2}\frac{\partial}{\partial \pi_k}.
\end{align}
Now the summation in the integrand is recasted as follows;
\begin{align}\label{change-sum-order}
& \sum_{m=2}^\infty c_m[p] t^{1-m} w^{2(m-n)} \sum_{k=2}^m (t/q)^{k-2} \sum_{n=k}^\infty 
q^{n-2}\frac{\partial}{\partial \pi_n} \CR
&= \sum_{m=2}^\infty \sum_{k=2}^m \sum_{n=k}^\infty c_m[p] t^{k-m-1} q^{n-k} w^{2(m-n)}\frac{\partial}{\partial \pi_n} \CR 
&= \sum_{n=2}^\infty \sum_{k=2}^n \sum_{m=k}^\infty c_m[p] t^{k-m-1} q^{n-k} w^{2(m-n)}\frac{\partial}{\partial \pi_n}.
\end{align}
The region of the triple summation is $\{ (n,m,k) \in \mathbb{N} \mid 2 \leq n,m, 1\leq k \leq n, 1\leq k \leq m \}$,
which is symmetric in $n$ and $m$. In the last equality the summations over $n$ and $m$ are exchanged. 
Note that the range of $(n,m)$ is $1\leq n,m$ in the case of Prop.\ref{Prop:commutativity+}, but here it is $2 \leq n,m$. 
Consequently the cancellation that takes place in the proof of Prop.\ref{Prop:commutativity+} becomes incomplete and 
the term corresponds to  $k=1$ and $m=0$ in the summation of \eqref{bosonic-com} survives. 
Hence, we have
\begin{align}
\Big[ u^{-1} H_{1,+1}, \sum_{n=1}^\infty c_n[p] \frac{\partial}{\partial \pi_n} \Big] 
&=  (q-1)(1-t) \oint \frac{dw}{w} \sum_{n=1}^\infty q^{n-1} w^{-2n} \CR
& \qquad \times V_B^{(-)}(w) V_B^{(+)}(w) \langle \varnothing \vert
\widetilde{V}_F^{(-)}(w) \widetilde{V}_F^{(+)}(w) \vert \varnothing \rangle_{F}\frac{\partial}{\partial \pi_n}.
\end{align}
By equating the right hand side with $(1-q^{-1})(t-1) u^{-1} (q/t)^{\frac{1}{2}} E_{2,1}$ we obtain  \eqref{com-forE21}.
\end{proof}

From Conjecture \ref{E21-F12-con} we can compute the fermion linear terms of $E_{2,1}$ and $F_{1,2}$
by setting $\langle \varnothing \vert \widetilde{V}_F^{(-)}(w) \widetilde{V}_F^{(+)}(w) \vert \varnothing\rangle_{F} =1$;
\begin{prp}
The leading terms of $E_{2,1}$ and $F_{1,2}$ are linear in the fermionic derivatives and given by 
\begin{align}
\label{E21-lin}
u^{-1} (q/t)^{\frac{1}{2}} E_{2,1}^{\mathrm{(lin)}}
&= - \sum_{n=1}^\infty C_\ell [p, \partial/\partial p]  q^n \frac{\partial}{\partial \pi_n}, \\
\label{F12-lin}
u^{-1} F_{1,2}^{\mathrm{(lin)}} &= \sum_{n=1}^\infty  C_{n-1} [p, \partial/\partial p] q^{n-1}\frac{\partial}{\partial \pi_n},
\end{align} 
where
\begin{equation}
C_\ell [p, \partial/\partial p] :=\sum_{n=1}^\infty c_n^\vee[p] \left( \sum_{\ell=1}^{n}
\widetilde{c}_{n-\ell}^\vee \left[ \partial/ \partial p \right] \right)
= \sum_{n=0}^\infty c_{n +\ell}^\vee[p]~\widetilde{c}_{n}^\vee \left[ \partial/ \partial p \right],
\end{equation}
and the generating functions of $c_{n}^\vee[p]$ and $\widetilde{c}_{n}^\vee \left[ \partial/ \partial p \right]$ are 
given by \eqref{c-gen} and \eqref{tc-gen} with the involution $(q,t) \to (q^{-1}, t^{-1})$. 
\end{prp}
In the last section by examining the Pieri rule of higher modes 
we worked out the first two fermion linear terms of $E_{2,1}$ and $F_{2,1}$.
\eqref{E21-lin} and \eqref{F12-lin} are consistent with the explicit computations there. 
We have explicitly checked that the leading terms of $E_{2,1}$ and $F_{1,2}$ are consistent with the Pieri rule 
of the super Macdonald polynomials of lower levels. 
Compared with $E_{2,0}$ and $F_{1,1}$ given by Conjecture \ref{Pieri-con}, $E_{2,1}$ and $F_{1,2}$ are
qualitatively different in the sense that the coefficients $C_\ell [p, \partial/\partial p]$ involve both $p$ and $\partial/\partial p$,
or both the creation and the annihilation operators. Moreover we find that if we expand $E_{2,1}$ and $F_{1,2}$ in fermionic variables
$\pi_k$, the expansion has any higher order terms. Consequently the Hamiltonians $H_{1,1}$ and $H_{2,1}$ have the same property. 
We would like to emphasize that this is qualitatively different from $H_{1,1}$ and $H_{2,1}$ whose fermionic parts are bi-linear 
in fermionic variables (see \eqref{H2-bilinear} and \eqref{H1-bilinear}).

\begin{rmk}
\begin{enumerate}
\item
By the anti-commutation relation 
$\left[ E_{1,0}, F_{1,2} \right]_{+} = \left[ \pi_1, F_{1,2} \right]_{+}= H_{1,+1}$,
the $n=1$ term of $u^{-1} F_{1,2}^{\mathrm{(lin)}}$ should give 
the bosonic part of $(q,t)$ inverted Hamiltonian $H_{1,1}$, which is in fact $C_0 [p, \partial/\partial p]$. 
\item
\eqref{F12-lin} agrees with the computation from $\mathcal{Q}_3$ in Appendix \ref{App:Tq}. 
\end{enumerate}
\end{rmk}
%

%%%%%%%%%%%%%%%%%%%%%%%%%%%%%%%%%%%%%%%%%%%%%%%%%%%%%%%%%%%%%%%%%%%%%%%%%%%%%%%%

The generating function of the coefficients $C_\ell [p, \partial/\partial p]$ is
\begin{equation}
\sum_{\ell=1}^\infty C_\ell [p, \partial/\partial p] z^\ell
=  \left[ \exp \left(\sum_{r=1}^\infty \frac{1-t^{-r}}{r} p_r z^r \right),  \exp \left( \sum_{n=1}^\infty (q^{n}-1) \frac{\partial}{\partial p_n} z^{-n} \right) \right]_{+},
\end{equation}
where the subscript $[\bullet]_{+}$ means taking the positive power part of $\bullet$. 
In other words, we have
\begin{equation}\label{bosonic-part}
C_\ell [p, \partial/\partial p] = \oint \frac{dz}{z} z^{-\ell}
\exp \left(\sum_{r=1}^\infty \frac{1-t^{-r}}{r} p_r z^r \right) \exp \left( \sum_{n=1}^\infty (q^{n}-1) \frac{\partial}{\partial p_n} z^{-n} \right).
\end{equation}
%Note that $1+ C_0 [p, \partial/\partial p]$ is nothing but the bosonic part of the Hamiltonian. 
In the following we extrapolate $C_\ell [p, \partial/\partial p]$ for $\ell \leq 0$ by \eqref{bosonic-part}. 

\begin{prp}
Up to fermion quartic terms we have
\begin{align}
\left[ E_{2,+1}, F_{2,-1} \right] &=
 \sum_{n=1}^\infty c_n^\vee[p] \widetilde{c}_n^\vee \left[ \partial/ \partial p \right] \CR
& \qquad + (1-t^{-1})(1-q^{-1}) 
\sum_{k=1}^\infty \sum_{n=1}^\infty  \sum_{r=1}^{k} [[r]]_{(q^{-1},t^{-1})} 
 C_{n-r} [p, \partial/\partial p] \cdot \widetilde{c}_{k-r} \left[ \partial/ \partial p \right] 
q^{n} \pi_k \frac{\partial}{\partial \pi_n}. 
\end{align}
\end{prp}

\begin{rmk}
The bosonic term of the commutation relation with
$F_{2,-1} = - \displaystyle{\sum_{k=1}^\infty} \pi_k \cdot \widetilde{c}_k \left[ \partial/ \partial p \right]$
is
\begin{align}
\left[ E_{2,+1}, F_{2,-1} \right]_{\pi_k=0} &= - \sum_{n=1}^\infty c_n^\vee[p] \sum_{\substack{k+ \ell=n, \\  \ell \neq 0}} 
\widetilde{c}_k^\vee \left[ \partial/ \partial p \right] q^\ell \widetilde{c}_\ell \left[ \partial/ \partial p \right] \CR
&= \sum_{n=1}^\infty c_n^\vee[p] \widetilde{c}_n^\vee \left[ \partial/ \partial p \right],
\end{align}
where we have used \eqref{bi-linear2}.
This is exactly the $(q,t) \to (q^{-1}, t^{-1})$ form of the commutation relation 
\begin{equation}
\left[ E_{2,0}, F_{2,-1} \right] = \sum_{k=1}^\infty c_k [p] \widetilde{c}_k\left[ \partial/ \partial p \right].
\end{equation}
\end{rmk}
This is a consistency check of our conjecture. Note that in both cases we employ the common super charge $F_{2,-1}$. 
Assuming \eqref{E21-lin} we can obtain the fermion bi-linear terms of the anti-commutator;
\begin{equation}\label{H-bilinear}
\left[ E_{2,+1}, F_{2,-1} \right]_{\pi_k \neq 0} =
- \sum_{k=1}^\infty \sum_{n=1}^\infty  \Big[ \widetilde{c}_k \left[ \partial/ \partial p \right],
C_n \left[ p , \partial/\partial p \right] \Big]  q^n \pi_k \frac{\partial}{\partial \pi_n}. 
\end{equation}
The generating function of the commutator is
\begin{align}\label{generatingC}
& \sum_{k, n=0}^\infty \Big[ \widetilde{c}_k \left[ \partial/ \partial p \right], C_n \left[ p , \partial/\partial p \right]  \Big]  z^{-k} w^n  \CR
&= \Big[ \exp \left( \sum_{\ell=1}^\infty (q^{-\ell}-1) \frac{\partial}{\partial p_\ell} z^{-\ell} \right),~ 
\exp \left( \sum_{m=1}^\infty \frac{1-t^{-m}}{m} p_m w^m \right) \exp \left( \sum_{n=1}^\infty (q^{n}-1) \frac{\partial}{\partial p_n} w^{-n} \right) \Big] \CR
&= \frac{(1-t^{-1})(q^{-1}-1) (w/z)}{(1- t^{-1} (w/z))(1-q^{-1}(w/z))} 
\exp \left( \sum_{m=1}^\infty \frac{1-t^{-m}}{m} p_m w^m \right) \CR
& \qquad \times 
\exp \left( \sum_{\ell=1}^\infty (q^{-\ell}-1) \frac{\partial}{\partial p_\ell} z^{-\ell} \right)
\exp \left( \sum_{n=1}^\infty (q^{n}-1) \frac{\partial}{\partial p_n} w^{-n} \right) \CR
&= (1-t^{-1})(q^{-1}-1) \sum_{n=1}^\infty \sum_{m,\ell=0}^\infty
[[n]]_{(q^{-1},t^{-1})} \cdot C_m [p, \partial/\partial p] \cdot \widetilde{c}_\ell \left[ \partial/ \partial p \right] z^{-(n+\ell)} w^{n+m},
\end{align}
where $[[n]]_{(t_1,t_2)}$ is defined by the generating function \eqref{qt-integer1}.
Note that we have (see \eqref{qt-integer2} and \eqref{qt-integer3})
\begin{equation}
[[n]]_{(t_1,t_2)} = \frac{t_1^n - t_2^n}{t_1 - t_2} = t_2^{n-1} \sum_{k=0}^{n-1} (t_1/t_2)^k.
\end{equation}
%Hence,
%\begin{align}
%& (1-t^{-1})(q^{-1}-1) [[n]]_{(q^{-1},t^{-1})} = \frac{(t-1)(1-q)}{tq} \frac{q^{-n} - t^{-n}}{q^{-1} - t^{-1}} \CR
%&= (t-1)(1-q)\frac{q^n -t^n}{(qt)^n (q-t)} =  - q \cdot b[n,0].
%\end{align}

Extracting the relevant coefficient of \eqref{generatingC}, we have
\begin{align}
\Big[ \widetilde{c}_k \left[ \partial/ \partial p \right], C_n \left[ p , \partial/\partial p \right]  \Big] 
&= (1-t^{-1})(q^{-1}-1) \sum_{r=1}^{k} [[r]]_{(q^{-1}, t^{-1})} 
 C_{n-r} [p, \partial/\partial p] \cdot \widetilde{c}_{k-r} \left[ \partial/ \partial p \right].
\end{align}
Substituting this to \eqref{H-bilinear}, we obtain a general formula of the fermion bi-linear terms 
\begin{align}\label{Fer-bi-linear}
& \left[ E_{2,+1}, F_{2,-1} \right]_{\pi_k \neq 0} \CR
& = (1-t^{-1})(1-q^{-1}) 
\sum_{k=1}^\infty \sum_{n=1}^\infty  \sum_{r=1}^{k} [[r]]_{(q^{-1},t^{-1})} 
 C_{n-r} [p, \partial/\partial p] \cdot \widetilde{c}_{k-r} \left[ \partial/ \partial p \right] 
q^{n} \pi_k \frac{\partial}{\partial \pi_n}. 
\end{align}
In particular the terms involving $\pi_1$ are
\begin{equation}\label{pi1-term}
(q-1)(1-t^{-1})\sum_{n=1}^\infty  C_{n-1} [p, \partial/\partial p]
q^{n-1} \pi_1 \frac{\partial}{\partial \pi_n}. 
\end{equation}
%

%%%%%%%%%%%%%%%%%%%%%%%%%%%%%%%%%%%%%%%%%%%%%%%%%%%%%%%%%%%%%%%%%%%%%%%%%%%%%%%%%%%%%%%%%%%%%%%%%%%%%%%%%%%%%%%%%%%%%%%%%%

\subsection{$H_{2, +1}$ and $H_{1, +1}$ as the anti-commutator of the super charges}

Finally the Hamiltonians $H_{2, +1}$ and $H_{1, +1}$ are expressed as anti-commutator of the super charges. 
\begin{prp}
\label{prop:positive-Hamiltonian}
\begin{align}
\left[ E_{2, +1}, F_{2,-1} \right]_{+} &= (t/q)^{\frac{1}{2}}u (1+ u^{-1}(q/t)^{\frac{1}{2}}H_{2,+1}), \\
\left[ E_{1,0}, F_{1,2} \right]_{+} &= H_{1,+1}.
\end{align}
\end{prp}
This means the Hamiltonians defined by \eqref{H_1+-formula} and \eqref{H_2+-formula} are consistent with the commutation relation 
of the shifted quantum toroidal algebra. Since $E_{1,0} = \pi_1$, it is easy to see the second commutation relation. 
To prove  the first relation we need 
\begin{lem}\label{Lemma6}
Setting $\widetilde{E}_{2,1} = (q/t)^{\frac{1}{2}} u^{-1} E_{2,1}$, we have
\begin{align}
\left[ \widetilde{E}_{2,1}, F_{2,-1} \right]_{+}
&= 1 - \oint \frac{dw}{w} V_B^{(-)}(w) V_B^{(+)}(w) \langle \varnothing \vert \widetilde{V}_F^{(-)}(w) \widetilde{V}_F^{(+)}(w) \vert \varnothing\rangle_{F} \CR
& - (q-1)(1-t^{-1})  \sum_{k=1}^\infty  \sum_{\ell=1}^\infty \oint \frac{dw}{w} w^{2(\ell-k)} V_B^{(-)}(w) V_B^{(+)}(w) \CR
& \qquad \times q^{k-1} t^{1-\ell} \pi_\ell  \langle \varnothing \vert \widetilde{V}_F^{(-)}(w) \widetilde{V}_F^{(+)}(w) 
\vert \varnothing\rangle_{F}\frac{\partial}{\partial \pi_k}.
\end{align}
\end{lem}
In fact applying the relation \eqref{H_1andH_2} to Lemma \ref{Lemma6}, we can replace 
$\langle \varnothing \vert \widetilde{V}_F^{(-)}(w) \widetilde{V}_F^{(+)}(w) \vert \varnothing\rangle_{F}$
with $\langle \varnothing \vert V_F^{(-)}(w) V_F^{(+)}(w) \vert \varnothing\rangle_{F}$. Namely
\begin{align}
\left[ \widetilde{E}_{2,1}, F_{2,-1} \right]
&= 1- \oint \frac{dw}{w} V_B^{(-)}(w) V_B^{(+)}(w) \langle \varnothing \vert V_F^{(-)}(w) V_F^{(+)}(w) \vert \varnothing\rangle_{F} \CR
&=  1 + u^{-1} (q/t)^{\frac{1}{2}} H_{2,+1},
\end{align}
which is the desired result. 

\begin{proof}
We prove  Lemma \ref{Lemma6},
Recall
$$
F_{2,-1} = - \sum_{k=1}^\infty \pi_k \widetilde{c}_k[\partial/\partial p].
$$
From Prop 6.2 we have
$$
- (q/t)^{\frac{1}{2}} u^{-1} E_{2,1} =
\sum_{k=1}^\infty \oint \frac{dw}{w} w^{-2k} V_B^{(-)}(w) V_B^{(+)}(w)  
\langle \varnothing \vert \widetilde{V}_F^{(-)}(w) \widetilde{V}_F^{(+)}(w) \vert \varnothing\rangle_{F} q^{k} \frac{\partial}{\partial \pi_k}.
$$
For simplicity we first look at the leading bosonic terms,
$$
-(q/t)^{\frac{1}{2}} u^{-1} E_{2,1} =
\sum_{k=1}^\infty \oint \frac{dw}{w} w^{-2k} V_B^{(-)}(w) V_B^{(+)}(w) q^{k} \frac{\partial}{\partial \pi_k}.
$$
We have
\begin{align}
(q/t)^{\frac{1}{2}} u^{-1} \left[ E_{2, +1}, F_{2,-1} \right]_{+} 
&= \sum_{k=1}^\infty \sum_{\ell=1}^\infty \left[ C_k[p, \partial/\partial p] q^k \frac{\partial}{\partial \pi_k}, 
 \pi_\ell \widetilde{c}_\ell[\partial/\partial p] \right]_{+} \CR
&= \sum_{m=1}^\infty \sum_{k=1}^m c_m^\vee[p] \widetilde{c}_{m-k}^\vee [\partial/\partial p] q^k \widetilde{c}_k[\partial/\partial p] \CR
&= -\sum_{m=1}^\infty c_m^\vee[p] \widetilde{c}_m^\vee[\partial/\partial p] \CR
&= 1+ u^{-1}(q/t)^{\frac{1}{2}}H_{2,+1}.
\end{align}
Hence
\begin{equation}
-u^{-1}(q/t)^{\frac{1}{2}} H_{2,+1} = 1 + \sum_{m=1}^\infty c_m^\vee[p] \widetilde{c}_m^\vee[\partial/\partial p] 
= \oint \frac{dw}{w} V_B^{(-)}(w) V_B^{(+)}(w).
\end{equation}

In general, we have
\begin{align}
\left[ \widetilde{E}_{2,1}, F_{2,-1} \right]_{+} = \left[ F_{2,-1}, \widetilde{E}_{2,1}\right]_{+} 
&= \sum_{n=1}^\infty \left[\pi_n, \widetilde{E}_{2,1}\right]_{+} \widetilde{c}_n[\partial/\partial p]
+ \sum_{n=1}^\infty \pi_n \left[\widetilde{c}_n[\partial/\partial p], \widetilde{E}_{2,1}\right].
\end{align}
The commutation relations
\begin{align}
\left[\widetilde{c}_n[\partial/\partial p], \widetilde{E}_{2,1}\right]
&= (1-q^{-1})(t^{-1}-1) \sum_{k=1}^\infty \sum_{m=0}^{n-1}  \oint \frac{dw}{w} w^{2(n-m-k)}
t^{1-n+m} b_{n-m}(q^{-1},t^{-1}) \CR
& \qquad \times V_B^{(-)}(w) V_B^{(+)}(w)~\widetilde{c}_m[\partial/\partial p]~
\langle \varnothing \vert \widetilde{V}_F^{(-)}(w) \widetilde{V}_F^{(+)}(w) 
\vert \varnothing\rangle_{F} q^{k} \frac{\partial}{\partial \pi_k}, \\
\left[ \pi_n , \widetilde{E}_{2,1}\right]_{+} 
&= \oint \frac{dw}{w} w^{-2n} q^n V_B^{(-)}(w) V_B^{(+)}(w)  
\langle \varnothing \vert \widetilde{V}_F^{(-)}(w) \widetilde{V}_F^{(+)}(w) \vert \varnothing\rangle_{F} \CR
& \qquad + \theta(n \geq2) (q-1)(1-t^{-1}) \sum_{k=1}^\infty \sum_{m=0}^{n-2} \sum_{\ell=m+2}^\infty 
\oint \frac{dw}{w} w^{2(\ell-k-n)} \CR
& \qquad \times q^{k+n-2} (t/q)^m t^{2-\ell}  V_B^{(-)}(w) V_B^{(+)}(w)  \pi_\ell 
\langle \varnothing \vert \widetilde{V}_F^{(-)}(w) \widetilde{V}_F^{(+)}(w) 
\vert \varnothing\rangle_{F} \frac{\partial}{\partial \pi_k},
\end{align}
imply 
\begin{align}
\left[ \widetilde{E}_{2,1}, F_{2,-1} \right]
&= \sum_{n=1}^\infty  \oint \frac{dw}{w} w^{-2n} V_B^{(-)}(w) V_B^{(+)}(w)~q^n\widetilde{c}_n[\partial/\partial p]
\langle \varnothing \vert \widetilde{V}_F^{(-)}(w) \widetilde{V}_F^{(+)}(w) \vert \varnothing\rangle_{F} \CR
& +  (q-1)(t^{-1}-1) \sum_{n=1}^\infty\sum_{k=1}^\infty \sum_{m=0}^{n-1} \oint \frac{dw}{w} w^{2(n-m-k)}
t^{1-n+m} b_{n-m}(q^{-1},t^{-1})\CR
& \qquad \times V_B^{(-)}(w) V_B^{(+)}(w)~\widetilde{c}_m[\partial/\partial p] 
  \pi_n  \langle \varnothing \vert \widetilde{V}_F^{(-)}(w) \widetilde{V}_F^{(+)}(w) 
\vert \varnothing\rangle_{F} q^{k-1} \frac{\partial}{\partial \pi_k} \CR
& + (q-1)(t^{-1}-1)  \sum_{k=1}^\infty  \sum_{n=2}^\infty  \sum_{m=2}^{n} \sum_{\ell=m}^\infty
\oint \frac{dw}{w} w^{2(\ell-k-n)} V_B^{(-)}(w) V_B^{(+)}(w) \widetilde{c}_n[\partial/\partial p] \CR
& \qquad \times q^{k+n-m} t^{m-\ell} \pi_\ell 
\langle \varnothing \vert \widetilde{V}_F^{(-)}(w) \widetilde{V}_F^{(+)}(w) 
\vert \varnothing\rangle_{F}\frac{\partial}{\partial \pi_k}.
\end{align}
Then by the same computation as above the first term is simplified to
\begin{equation}
 1 - \oint \frac{dw}{w} V_B^{(-)}(w) V_B^{(+)}(w) \langle \varnothing \vert \widetilde{V}_F^{(-)}(w) \widetilde{V}_F^{(+)}(w) \vert \varnothing\rangle_{F}.
 \end{equation}
To confirm this, let us assume the expansion
\begin{equation}
\langle \varnothing \vert \widetilde{V}_F^{(-)}(w) \widetilde{V}_F^{(+)}(w) \vert \varnothing\rangle_{F}
= 1+ \sum_{m \in \mathbb{Z}} d_m[\pi, \partial/\partial \pi] w^{-2m}, %\qquad d_0=1 ??
\end{equation} 
and use the identity 
\begin{equation}
\sum_{n+m = \ell, n\geq 1} q^n \widetilde{c}_n[\partial/\partial p]  \widetilde{c}^\vee_m [\partial/\partial p] = -  \widetilde{c}^\vee_\ell [\partial/\partial p],
\end{equation}
to obtain
\begin{align}
 &\sum_{n=1}^\infty  \oint \frac{dw}{w} w^{-2n} V_B^{(-)}(w) V_B^{(+)}(w)~q^n\widetilde{c}_n[\partial/\partial p]
\langle \varnothing \vert \widetilde{V}_F^{(-)}(w) \widetilde{V}_F^{(+)}(w) \vert \varnothing\rangle_{F} \CR
& = \sum_{n=1}^\infty \sum_{k=0}^\infty \sum_{\ell=0}^\infty  \oint \frac{dw}{w}  w^{2(k-n-\ell)} c_k^\vee[p] \widetilde{c}_\ell^\vee[\partial/\partial p] 
q^n \widetilde{c}_n[\partial/\partial p]  \left( 1+ \sum_{m \in \mathbb{Z}} d_m[\pi, \partial/\partial \pi] w^{-2m} \right) \CR
&= -\sum_{m=1}^\infty c_m^\vee[p] \widetilde{c}_m^\vee[\partial/\partial p] 
-  \sum_{k=0}^\infty \sum_{n=0}^\infty   c_k^\vee[p] \widetilde{c}_n^\vee[\partial/\partial p] d_{k-n} [\pi, \partial/\partial \pi] \CR
& = 1 -  \oint \frac{dw}{w} V_B^{(-)}(w) V_B^{(+)}(w) \langle \varnothing \vert \widetilde{V}_F^{(-)}(w) \widetilde{V}_F^{(+)}(w) \vert \varnothing\rangle_{F}.
\end{align}
%Then we have
%\begin{equation}
% \oint \frac{dw}{w} V_B^{(-)}(w) V_B^{(+)}(w) \langle \varnothing \vert \widetilde{V}_F^{(-)}(w) \widetilde{V}_F^{(+)}(w) \vert \varnothing\rangle_{F}
% = \sum_{k=1}^\infty  \sum_{\ell=1}^\infty  c_k^\vee[p] \widetilde{c}_\ell^\vee [\partial/\partial p] d_{k-\ell}[\pi, \partial/\partial \pi]. 
%\end{equation}

For the remaining terms we exchange the summation over $n$ and $\ell$ for the third term to obtain
%&= 1 - \oint \frac{dw}{w} V_B^{(-)}(w) V_B^{(+)}(w) \langle \varnothing \vert \widetilde{V}_F^{(-)}(w) \widetilde{V}_F^{(+)}(w) \vert \varnothing\rangle_{F} \CR
\begin{align}
& (q-1)(t^{-1}-1) \sum_{k=1}^\infty \sum_{\ell=1}^\infty \sum_{n=0}^{\ell-1} \oint \frac{dw}{w} w^{2(\ell-n-k)}
t^{1-\ell+n} b_{\ell-n}(q^{-1},t^{-1})\CR
& \qquad \times V_B^{(-)}(w) V_B^{(+)}(w)~\widetilde{c}_n[\partial/\partial p] 
  \pi_\ell  \langle \varnothing \vert \widetilde{V}_F^{(-)}(w) \widetilde{V}_F^{(+)}(w) 
\vert \varnothing\rangle_{F} q^{k-1} \frac{\partial}{\partial \pi_k} \CR
& + (q-1)(t^{-1}-1)  \sum_{k=1}^\infty  \sum_{\ell=2}^\infty  \sum_{m=2}^{\ell} \sum_{n=m}^\infty
\oint \frac{dw}{w} w^{2(\ell-k-n)} V_B^{(-)}(w) V_B^{(+)}(w) \widetilde{c}_n[\partial/\partial p] \CR
& \qquad \times q^{k+n-m} t^{m-\ell} \pi_\ell 
\langle \varnothing \vert \widetilde{V}_F^{(-)}(w) \widetilde{V}_F^{(+)}(w) 
\vert \varnothing\rangle_{F}\frac{\partial}{\partial \pi_k}.
\end{align}
Let us look at the coefficient of $t^{1-\ell}  q^{k-1} \pi_\ell\frac{\partial}{\partial \pi_k}$;
\begin{align}
& \sum_{n=0}^{\ell-1} \sum_{r=1}^{\ell-n} \oint \frac{dw}{w} w^{2(\ell-k-n)} t^{n+r-1} q^{-r+1} \widetilde{c}_{n} [\partial/\partial p] \CR
& \qquad + \theta( \ell \geq 2) \sum_{m=2}^{\ell} \sum_{n=m}^\infty  \oint \frac{dw}{w} w^{2(\ell-k-n)}  t^{m-1} q^{n-m+1}\widetilde{c}_{n} [\partial/\partial p] \CR
&=  \sum_{m=1}^{\ell}\sum_{n=0}^{m-1} \oint \frac{dw}{w} w^{2(\ell-k-n)} t^{m-1} q^{n-m+1} \widetilde{c}_{n} [\partial/\partial p] \CR
& \qquad + \theta( \ell \geq 2) \sum_{m=2}^{\ell} \sum_{n=m}^\infty  \oint \frac{dw}{w} w^{2(\ell-k-n)}  t^{m-1} q^{n-m+1} \widetilde{c}_{n} [\partial/\partial p] \CR
&=  \oint \frac{dw}{w} w^{2(\ell-k)} \widetilde{c}_{0} [\partial/\partial p] 
+  \sum_{m=2}^{\ell} \sum_{n=0}^\infty  \oint \frac{dw}{w} w^{2(\ell-k-n)}  t^{m-1} q^{n-m+1} \widetilde{c}_{n} [\partial/\partial p].
\end{align}
Hence, by the similar argument as the proof of Lemma \ref{lemma-EF} the sum of the remaining terms is simplified to
\begin{align}
& - (q-1)(1-t^{-1})  \sum_{k=1}^\infty  \sum_{\ell=1}^\infty \oint \frac{dw}{w} w^{2(\ell-k)} V_B^{(-)}(w) V_B^{(+)}(w) \CR
& \qquad \times q^{k-1} t^{1-\ell} \pi_\ell  \langle \varnothing \vert \widetilde{V}_F^{(-)}(w) \widetilde{V}_F^{(+)}(w) 
\vert \varnothing\rangle_{F}\frac{\partial}{\partial \pi_k}.
\end{align}
\end{proof}

\subsection{Commutativity of $H_{i,-1}$ and $H_{i,+1}$}

In a level zero representation the Hamiltonians are mutually commuting. We can check it explicitly based on the formula 
of $H_{i,-1}$ and $H_{i,+1}$ we have obtained. 
Recall that $H_{i, -1}$ is expressed as the anti-commutator of the super charges;
\begin{align}
H_{2,-1} &= u^{-1}(q/t)^{\frac{1}{2}} (1- \big[ E_{2.0}, F_{2,-1} \big]),  \\
H_{1,-1} &= -u  \big[ E_{1.-1}, F_{1,0} \big].
\end{align}
We have
\begin{align}
\big[ H_{2,+1}, H_{2,-1} \big]
&=  - u^{-1}(q/t)^{\frac{1}{2}} \big[ H_{2,+1}, \big[ E_{2,0}, F_{2,-1} \big] \big] \CR
&= u^{-1}(q/t)^{\frac{1}{2}} \left(\big[ E_{2,0}, \big[ F_{2,-1}, H_{2,+1} \big] \big]  - \big[ F_{2,-1}, \big[ H_{2,+1}, E_{2,0} \big] \big]  \right) =0,
\end{align}
where we have used Prop \ref{Prop:commutativity+}. We also have
\begin{align}
\big[ H_{2,+1}, H_{1,-1} \big]
&= -u \big[ H_{2,+1}, \big[ E_{1-1}, F_{1,0} \big] \big] \CR
&= u(q^{\frac{1}{2}}- q^{-\frac{1}{2}})(t^{\frac{1}{2}}- t^{-\frac{1}{2}}) 
\left( \big[ F_{1,0}, E_{1,0} \big] - \big[E_{1,-1}, F_{1,1} \big] \right) =0,
\end{align}
by using explicit formula for the super charges.
\begin{align}
\big[ H_{1,+1}, H_{1,-1} \big]
&=  - u  \big[ H_{1,+1}, \big[ E_{1,-1}, F_{1,0} \big] \big] \CR
&= u \left(\big[ E_{1,-1}, \big[ F_{1,0}, H_{1,+1} \big] \big]  - \big[ F_{1,0}, \big[ H_{1,+1}, E_{1,-1} \big] \big]  \right) =0,
\end{align}
where we have used Prop \ref{Prop:commutativity+} again. 
Finally computing similarly, we have
\begin{align}
\big[ H_{2,+1}, H_{1,-1} \big]  = u^{-1}(q-1)(1-t^{-1})
\left( \big[ E_{2,0}, F_{2,0} \big] - \big[ F_{2,-1}, E_{2,1} \big] \right).
\end{align}
To show the vanishing of the right hand side, we need an explicit formula for $F_{2,0}$, which we have not given yet. 
By using 
\begin{equation}
\big[ H_{1,+1}, F_{2,-1} \big] = - (q^{\frac{1}{2}}- q^{-\frac{1}{2}})(t^{\frac{1}{2}}- t^{-\frac{1}{2}})  F_{2,0},
\end{equation}
we obtain\footnote{The computation is quite parallel to that of $E_{2,1}$ (see Conjecture \ref{E21-F12-con}).}
\begin{equation}
F_{2,0} =  u (t/q)^{\frac{1}{2}} \sum_{k=1}^\infty \oint \frac{dw}{w} w^{2k} V_B^{(-)}(w) V_B^{(+)}(w)  
t^{-k} \pi_k \langle \varnothing \vert \widetilde{V}_F^{(-)}(w) \widetilde{V}_F^{(+)}(w) \vert \varnothing\rangle_{F}.
\end{equation}
Proposition \ref{prop:positive-Hamiltonian} gives  $\big[ F_{2,-1}, E_{2,1} \big]$. 
A parallel computation to the proof of Proposition \ref{prop:positive-Hamiltonian} implies that $\big[ E_{2,0}, F_{2,0} \big] $ 
agrees with it.\footnote{We should have $\big[ E_{2,0}, F_{2,0} \big] = K_{2,0}^{+} H_{2,+1} - K_{2,0}^{-} = \big[E_{2,1}, F_{2,-1} \big]$.}

%%%%%%%%%%%%%%%%%%%%%%%%%%%%%%%%%%%%%%%%%%%%%%%%%%%%%%%%%%%%%%

\vspace{5mm}
\begin{ack}
We would like to thank O.~Blondeau-Fournier, J.-E.~Bourgine, 
D.~Galakhov, Al.~Morozov, N.~Tselousov and S.~Yanagida for useful discussions. 
Our work is supported in part by Grants-in-Aid for Scientific Research (Kakenhi);
23K03087 (H.K.), 21K03180 (R.O. and J.S.) and 24K06753 (J.S.).
The work of R.O. was partly supported by the Research Institute for Mathematical Sciences,
an International Joint Usage/Research Center located in Kyoto University.
\end{ack}

%%%%%%%%%%%%%%%%%%%%%%%%%%%%%%%%%%%%%%%%%%%%%%%%%%%%%%%%%

\newpage

\appendix

\section{Super Macdonald polynomials up to level 4}
\label{App:Lower-Mac}

To write down the super Macdonald polynomials $\mathcal{M}_{\Lambda}(x, \theta;q,t)$ explicitly,
it is convenient to introduce the bosonic and the fermionic power sum 
polynomials defined by
\begin{equation}
p_k := \sum_{i=1}^N x_i^k, \qquad \pi_k := \sum_{i=1}^N \theta_i x_i^{k-1}.
\end{equation}
In the following we list the super Macdonald polynomials up to level $4$.\footnote{We change 
the convention in {\tt 2506.01415}; $(q^2, t^2) \to (q,t)$.}
In contrast to the Macdonald polynomials the super Macdonald polynomials do depend on the parameter even if we set $q=t$.
Another different feature from the Macdonald polynomials is that
the super Macdonald polynomials are not invariant under $(q,t) \to (q^{-1}, t^{-1})$.
\subsubsection*{Up to level $2$}
\begin{align*}
&\mathcal{M}_{(\frac{1}{2})} = \pi_1, \qquad \mathcal{M}_{(1)} =p_1, \\
& \mathcal{M}_{(\frac{3}{2})} = \frac{q(1-t)}{1-qt}p_1\pi_1 + \frac{1-q}{1-qt}\pi_2,
\qquad \mathcal{M}_{(1, \frac{1}{2})} = p_1 \pi_1 - \pi_2, \\
& \mathcal{M}_{(2)} = \frac{1}{2}\left( \frac{(1+q)(1-t)}{1-qt} p_1^2 + \frac{(1+t)(1-q)}{1-qt}p_2 \right), \\
& \mathcal{M}_{(1,1)} = \frac{1}{2}(p_1^2 - p_2), 
\qquad \mathcal{M}_{(\frac{3}{2},\frac{1}{2})} = \pi_2\pi_1.
\end{align*}
\subsubsection*{Level $\frac{5}{2}$}
There are four super Macdonald polynomials. They are all fermionic;
\begin{align*}
& \mathcal{M}_{(\frac{5}{2})} = \frac{q^2(1+q)(1-t)^2}{2(1-qt)(1-q^2t)}p_1^2 \pi_1
+ \frac{q^2(1-q)(1-t^2)}{2(1-qt)(1-q^2t)} p_2 \pi_1 \\
& \qquad \qquad + \frac{q(1-q^2)(1-t)}{(1-qt)(1-q^2t)} p_1 \pi_2 
+ \frac{(1-q)^2(1+q)}{(1-qt)(1-q^2t)} \pi_3, \\
& \mathcal{M}_{(2, \frac{1}{2})} = \frac{(1+q)(1-t)}{2(1-qt)} p_1^2 \pi_1
+ \frac{(1-q)(1+t)}{2(1-qt)} p_2 \pi_1 - \frac{q(1-t)}{1-qt} p_1 \pi_2 
- \frac{1-q}{1-qt} \pi_3 , \\
& \mathcal{M}_{(\frac{3}{2},1)} = \frac{q(1-t^2)}{2(1-qt^2)}(p_1^2 \pi_1 - p_2 \pi_1) 
+ \frac{1-q}{1- q t^2}( p_1 \pi_2 - \pi_3), \\
& \mathcal{M}_{(1,1, \frac{1}{2})} = \frac{p_1^2 \pi_1}{2} - \frac{p_2 \pi_1}{2} - p_1\pi_2 + \pi_3.
\end{align*}
\subsubsection*{Level 3}
There are five super Macdonald polynomials, three are bosonic and two are fermionic;
\begin{equation}\label{level3}
\mathcal{M}_{(\frac{5}{2},\frac{1}{2})} = \frac{1-q}{1-qt} \pi_3 \pi_1 + \frac{q(1-t)}{1-qt} p_1 \pi_2 \pi_1, \qquad
\mathcal{M}_{(\frac{3}{2} ,1, \frac{1}{2})} = p_1 \pi_2 \pi_1 - \pi_3 \pi_ 1.
\end{equation}
\begin{align*}
\mathcal{M}_{(3)} &= \frac{(1-q)(1-q^2)(1+t+t^2)}{3(1-qt)(1-q^2t)} p_3
+ \frac{(1-q^3)(1-t^2)}{2(1-qt)(1-q^2t)} p_2p_1 \\
& \qquad + \frac{(1+q)(1+q+q^2)(1-t)^2}{6(1-qt)(1-q^2t)} p_1^3,
\\
\mathcal{M}_{(2,1)} &= - \frac{(1-q)(1+t+t^2)}{3(1-qt^2)} p_3 
+ \frac{(t-q)(1+t)}{2(1-qt^2)} p_2p_1 + \frac{(1-t)(2+q+t+2qt)}{6(1-qt^2)}p_1^3,
\\
\mathcal{M}_{(1,1,1)} &= \frac{1}{3} p_3 -\frac{1}{2}p_2 p_1 + \frac{1}{6} p_1^3.
\end{align*}

\subsubsection*{Level $\frac{7}{2}$}
There are seven super Macdonald polynomials. They are all fermionic;
\begin{align*}
\mathcal{M}_{(\frac{7}{2})} &= \frac{(1-q)^2(1-q^3)(1+q)}{(1-qt)(1-q^2t)(1-q^3t)} \pi_4
+ \frac{q^3(1-q)(1-q^2)(1-t^3)}{3(1-qt)(1-q^2t)(1-q^3t)} \pi_1 p_3 \\
& \quad + \frac{q(1-q)(1-q^3)(1+q)(1-t)}{(1-qt)(1-q^2t)(1-q^3t)} \pi_3 p_1
+ \frac{q^2(1-q)(1-q^3)(1-t^2)}{2(1-qt)(1-q^2t)(1-q^3t)}\pi_2 p_2 \\
& \qquad + \frac{q^3(1-q^3)(1-t)(1-t^2)}{2(1-qt)(1-q^2t)(1-q^3t)} \pi_1 p_2 p_1
+ \frac{q^2(1-q^3)(1+q)(1-t)^2}{2(1-qt)(1-q^2t)(1-q^3t)} \pi_2 p_1^2 \\
& \qquad +\frac{q^3(1+q)(1+q+q^2)(1-t)^3}{6(1-qt)(1-q^2t)(1-q^3t)} \pi_1 p_1^3, \\
\mathcal{M}_{(3,\frac{1}{2})} &= - \frac{(1-q)^2(1+q)}{(1-qt)(1-q^2t)} \pi_4
+ \frac{(1-q)^2(1+q)(1+t+t^2)}{3(1-qt)(1-q^2t)} \pi_1 p_3 \\
& \qquad - \frac{q(1-q)(1+q)(1-t)}{(1-qt)(1-q^2t)} \pi_3 p_1 
- \frac{q^2(1-q)(1-t^2)}{2(1-qt)(1-q^2t)} \pi_2 p_2 \\
& \qquad + \frac{(1-q^3)(1-t^2)}{2(1-qt)(1-q^2t)} \pi_1 p_2 p_1 
 - \frac{q^2(1+q)(1-t)^2}{2(1-qt)(1-q^2t)} \pi_2 p_1^2 \\
& \qquad \qquad + \frac{(1+q)(1+q+q^2)(1-t)^2}{6(1-qt)(1-q^2t)} \pi_1 p_1^3,
\end{align*}
\begin{align*}
\mathcal{M}_{(\frac{5}{2},1)} &= - \frac{(1-q)^2(1+q)}{(1-qt)^2(1+qt)} \pi_4 
- \frac{q^2(1-q)(1-t^3)}{3(1-qt)^2(1+qt)} \pi_1 p_3 
+ \frac{(1-q)(1-q-q^2+qt)}{(1-qt)^2(1+qt)}\pi_3 p_1 \\
& \qquad - \frac{q^2(1-q)(1-t)(1+t)}{2(1-qt)^2(1+qt)} \pi_2 p_2 
+ \frac{q^2(t-q)(1-t)(1+t)}{2(1-qt)^2(1+qt)}\pi_1 p_2 p_1  \\
& \qquad + \frac{q(1-q)(1-t)(2+q+qt)}{2(1-qt)^2(1+qt)} \pi_2 p_1^2 
+\frac{q^2(1-t)^2(2+q+t+2qt)}{6(1-qt)^2(1+qt)} \pi_1 p_1^3, \\
\mathcal{M}_{(2,\frac{3}{2})} &= \frac{(1-q)(q-t)}{(1-qt)(1-qt^2)} \pi_4
- \frac{q(1-q)(1-t^3)}{3(1-qt)(1-qt^2)} \pi_1 p_3 \\
& \qquad - \frac{(1-q)(1-t)(1+q+qt)}{(1-qt)(1-qt^2)} \pi_3 p_1 
+ \frac{(1+t)(1-q+q^2-qt-q^2t+q^2t^2)}{2(1-qt)(1-qt^2)} \pi_2 p_2 \\
& \qquad - \frac{q(q-t)(1-t)(1+t)}{2(1-qt)(1-qt^2)} \pi_1 p_2 p_1
 + \frac{(1-t)(1-q-q^2+qt-q^2t + q^2t^2)}{2(1-qt)(1-qt^2)} \pi_2 p_1^2 \\
& \qquad + \frac{q(1-t)^2(2+q+t+2qt)}{6(1-qt)(1-qt^2)} \pi_1 p_1^3, \\
\mathcal{M}_{(2,1,\frac{1}{2})} &= \frac{(1-q)(1+t)}{1-qt^2} \pi_4 
- \frac{(1-q)(1+t+t^2)}{3(1-qt^2)} \pi_1 p_3
+ \frac{q-t+qt-qt^2}{1-qt^2} \pi_3 p_1 \\
& \qquad - \frac{(1-q)(1+t)}{2(1-qt^2)} \pi_2 p_2 - \frac{(q-t)(1+t)}{2(1-qt^2)} \pi_1 p_2p_1
- \frac{(1-t)(1+q+2qt)}{2(1-qt^2)} \pi_2 p_1^2 \\
& \qquad + \frac{(1-t)(2+q+t+2qt)}{6(1-qt^2)} \pi_1 p_1^3, \\
\mathcal{M}_{(\frac{3}{2},1,1)} &= \frac{1-q}{1-qt^3}\pi_4 + \frac{q(1-t^3)}{3(1-qt^3)} \pi_ 1 p_3
- \frac{1-q}{1-qt^3} \pi_3 p_1 - \frac{1-q}{2(1-qt^3)} \pi_2 p_2 \\
& \qquad - \frac{q(1-t^3)}{2(1-qt^3)} \pi_1 p_2 p_1 + \frac{1-q}{2(1-qt^3)} \pi_2 p_1^2 
+ \frac{q(1-t^3)}{6(1-qt^3)} \pi_ 1 p_1^3,
\\
\mathcal{M}_{(1,1,1,\frac{1}{2})} &= -\pi_4 + \frac{1}{3}p_3 \pi_1 + \pi_3 p_1 + \frac{1}{2}\pi_2 p_2
- \frac{1}{2}\pi_1 p_2p_1 -\frac{1}{2}\pi_2 p_1^2 + \frac{1}{6} \pi_1 p_1^3.
\end{align*}

\subsubsection*{Level 4}
There are ten super Macdonald polynomials. Five are bosonic and five are fermionic.
\begin{align*}
\mathcal{M}_{(4)} &= \frac{(1-q)^3(1+q)(1+q+q^2)(1+t)(1+t^2)}{4(1-qt)(1-q^2t)(1-q^3t)} p_4 \\
& \qquad + \frac{(1-q)^2(1+q)^2(1+q^2)(1-t)(1+t+t^2)}{3(1-qt)(1-q^2t)(1-q^3t)} p_3 p_1 \\
& \qquad + \frac{(1-q)^2(1+q)(1+q^2)(1+q+q^2)(1-t)(1+t)^2}{8(1-qt)(1-q^2t)(1-q^3t)} p_2^2 \\
& \qquad + \frac{(1-q)(1+q)(1+q^2)(1+q+q^2)(1-t)^2(1+t)}{4(1-qt)(1-q^2t)(1-q^3t)} p_2 p_1^2 \\
& \qquad +\frac{(1+q)^2(1+q^2)(1+q+q^2)(1-t)^3}{24 (1-qt)(1-q^2t)(1-q^3t)} p_1^4, 
\\
\mathcal{M}_{(3,1)} &= - \frac{(1-q)^2(1+q)(1+t)(1+t^2)}{4(1-qt)^2 (1+qt)} p_4
- \frac{(1-q)(1+q)(q-t)(1+t+t^2)}{3(1-qt)^2 (1+qt)} p_3 p_1 \\
& \qquad - \frac{(1-q)(1+q^2)(1-t)(1+t)^2}{8(1-qt)^2 (1+qt)} p_2^2 \\
& \qquad + \frac{(1-t)(1+t)(1-q-q^2-q^3 +t +tq + q^2t - q^3t)}{4(1-qt)^2 (1+qt)} p_2 p_1^2\\
& \qquad + \frac{(1+q)(1-t)^2(3+2q+q^2+t+2qt+3q^2t)}{24(1-qt)^2 (1+qt)} p_1^4,
\end{align*}
\begin{align*}
\mathcal{M}_{(2,2)} &= -\frac{(1-q)(q-t)(1+t^2)}{4(1-qt)(1-qt^2)} p_4 
- \frac{(1-q)(1+q)(1-t)(1+t+t^2)}{3(1-qt)(1-qt^2)} p_3 p_1 \\
& \qquad + \frac{(1+t)(2-q+q^2-t-2qt-q^2t+t^2-qt^2+2q^2t^2)}{8(1-qt)(1-qt^2)} p_2^2 \\
& \qquad - \frac{(1+q)(q-t)(1-t)(1+t)}{4(1-qt)(1-qt^2)} p_2 p_1^2 
+ \frac{(1+q)(1-t)^2(2+q+t+2qt)}{24(1-qt)(1-qt^2)} p_1^4, 
\\
\mathcal{M}_{(2,1,1)} &= \frac{(1-q)(1+t)(1+t^2)}{4(1-qt^3)} p_4 
+ \frac{(q-t)(1+t+t^2)}{3(1-qt^3)} p_3 p_1 \\
& \qquad - \frac{(1-q)(1+t)(1+t^2)}{8(1-qt^3)} p_2^2 
- \frac{1+q-t+qt-t^2+qt^2-t^3-qt^3}{4(1-qt^3)} p_2 p_1^2 \\
& \qquad + \frac{(1-t)(3+q+2t+2qt+t^2+3qt^2)}{24(1-qt^3)} p_1^4,
\\
\mathcal{M}_{(1,1,1,1)} &= -\frac{1}{4} p_4 +\frac{1}{3}p_3 p_1 + \frac{1}{8}p_2^2 
- \frac{1}{4} p_2 p_1^2 + \frac{1}{24}p_1^4.
\end{align*}
\begin{align*}
\mathcal{M}_{(\frac{7}{2},\frac{1}{2})} &= \frac{(1-q)^2(1+q)}{(1-qt)(1-q^2t)}\pi_4 \pi_1
+\frac{q(1-q)(1+q)(1-t)}{(1-qt)(1-q^2t)} p_1 \pi_3 \pi_1 \\
& \qquad +\frac{q^2(1-q)(1-t)(1+t)}{2(1-qt)(1-q^2t)} p_2 \pi_2 \pi_1
+\frac{q^2(1+q)(1-t)^2}{2(1-qt)(1-q^2t)} p_1^2 \pi_2 \pi_1,
\\
\mathcal{M}_{(\frac{5}{2},\frac{3}{2})} &= - \frac{q(1-q)(1-t)}{(1-qt)^2(1+qt)} \pi_4 \pi_1
+ \frac{(1-q)(1-q^2t)}{(1-qt)^2(1+qt)} \pi_3 \pi_2 \\
& \qquad + \frac{q(1-q)(1-t)}{(1-qt)^2} p_1 \pi_3 \pi_1 - \frac{q^2(1-t)(1+t)}{2(1-qt)(1+qt)} p_2 \pi_2 \pi_1
+ \frac{q^2(1-t)^2}{2(1-qt)^2} p_1^2 \pi_2 \pi_1,
\\
\mathcal{M}_{(\frac{5}{2}, 1,\frac{1}{2})} &= - \frac{1-q}{1-qt}\pi_4 \pi_1 - \frac{1-q}{1-qt} \pi_3 \pi_2 
+ \frac{1-2q+qt}{1-qt} p_1 \pi_3 \pi_1 + \frac{q(1-t)}{1-qt} p_1^2 \pi_2 \pi_1,
\\
\mathcal{M}_{(2,\frac{3}{2},\frac{1}{2})} &= - \frac{t(1-q)}{1-qt^2} \pi_4 \pi_1 + \pi_3 \pi_2
- \frac{(1-t)(1+qt)}{1-qt^2} p_1 \pi_3 \pi_1 \\
& \qquad + \frac{(1+t)(1-qt)}{2(1-qt^2)} p_2 \pi_2 \pi_1 
+ \frac{(1-t)(1+qt)}{2(1-qt^2)} p_1^2 \pi_2 \pi_1,
\\
\mathcal{M}_{(\frac{3}{2},1,1,\frac{1}{2})} &=  \pi_4 \pi_1 
- p_1 \pi_3 \pi_1 - \frac{1}{2} p_2 \pi_2 \pi_1 + \frac{1}{2} p_1^2 \pi_2 \pi_1 .
\end{align*}

%%%%%%%%%%%%%%%%%%%

\section{Pieri rule of the higher modes}
\label{App:Higher-modes}

\subsection{$E_{1,n}$}

Let us look at the Pieri rule of $E_{1,n}$. On $\mathcal{M}_{(1)}=p_1$ the Pieri rule tells\footnote{For simplicity 
we put $u=1$.}
\begin{align}
E_{1,n}\cdot \mathcal{M}_{(1)} &= q^n \mathcal{M}_{(\frac{3}{2})} + t^{-n}\frac{1-q}{1-qt}\mathcal{M}_{(1, \frac{1}{2})} \CR
&= \frac{q^{n+1}(1-t)}{1-qt}p_1 \pi_1 + \frac{q^n(1-q)}{1-qt}\pi_2 + \frac{t^{-n}(1-q)}{1-qt}(p_1 \pi_1 - \pi_2).
\end{align}
Assuming 
\begin{equation}\label{E1n-ansatz}
E_{1,n} = \left( 1 + \alpha_n p_1 \frac{\partial}{\partial p_1} \right) \pi_1 + \beta_n \pi_2 \frac{\partial}{\partial p_1} + \cdots.
\end{equation}
we obtain
\begin{align}
1+ \alpha_n & =\frac{q^n-t^{-n}}{q^{-1} - t} + \frac{q^{n+1}- t^{-n-1}}{q-t^{-1}},
\\
\beta_n &= (q^{-1}-1)\frac{q^n-t^{-n}}{q^{-1}-t}.
\end{align}
When $n=0,-1$, we recover
\begin{equation}
E_{1,0} = \pi_ 1, \qquad E_{1,-1} = \pi_1 + (q^{-1}-1) \pi_2 \frac{\partial}{\partial p_1} + \cdots .
\end{equation}

In particular up to level 2, we have
\begin{equation}\label{E11-level2}
E_{1,+1} = u \left( \pi_1 + (q-1)(1-t^{-1}) p_1 \frac{\partial}{\partial p_1} \pi_1 - t^{-1}(1-q)\pi_2 \frac{\partial}{\partial p_1} + \cdots \right). 
\end{equation}
We need the additional term such as 
\begin{equation}
\gamma_n \pi_2 \pi_1 \frac{\partial}{\partial \pi_1} \frac{\partial}{\partial p_1}.
\end{equation}

\subsection{$F_{1,n}$}

The Pieri coefficients for $F_{1,n}$ is 
\begin{equation}
\widetilde{\psi}_k^{(1,n)}(q,t) 
=  (-1)^{F(k)} u^{n-1} \frac{t^{(n-1)(1-k)}q^{n(\lambda_k-1)}(1-t)}{1-q^{\lambda_k-1}t^{\ell(\lambda)-k+1}}
\prod_{i=k+1}^{\ell(\lambda)} \frac{1 - q^{\lambda_k -\lambda_i -1 + \barsigma_i}t^{i-k+1}}
{1 - q^{\lambda_k -\lambda_i -1 + \barsigma_i}t^{i-k}},
\end{equation}
which implies
\begin{align}
\label{F1n-1}
F_{1,n} \cdot \mathcal{M}_{(\frac{1}{2})} 
&= u^{n-1}\mathcal{M}_{(\varnothing)}, \\
\label{F1n-2}
F_{1,n} \cdot \mathcal{M}_{(\frac{3}{2})} 
&= u^{n-1}q^n\frac{1-t}{1-qt}\mathcal{M}_{(1)}, \\
\label{F1n-3}
F_{1,n} \cdot \mathcal{M}_{(1,\frac{1}{2})} 
&= u^{n-1} t^{1-n} \mathcal{M}_{(1)}.
\end{align}
Taking \eqref{F1n-1} into account, we may assume
\begin{equation}\label{F1n-ansatz}
F_{1, n} = u^{n-1} \left( \frac{\partial}{\partial \pi_1} + \alpha_n p_1 \frac{\partial}{\partial \pi_2} 
+ \beta_n p_1 \frac{\partial}{\partial p_1} \frac{\partial}{\partial \pi_1} + \cdots \right).
\end{equation}
Then from \eqref{F1n-2} and \eqref{F1n-3} we obtain, for $n>1$
\begin{align}
\beta_n &= (q^{n-1}-1) + (1-q)\sum_{k=0}^{n-2}q^kt^{k-n+1}, \CR
\alpha_n &= q (t-1)\sum_{k=0}^{n-2}q^kt^{k-n+1} = (1-t)\frac{q^{n-1}-t^{1-n}}{q^{-1}-t}.
\end{align}
Note that when $n=0,1$, $\alpha_n=(t-1)$ and $\alpha_n=0$, respectively. 

The Pieri rule at the next level is
\begin{align}
F_{1,n} \cdot \mathcal{M}_{(\frac{3}{2}, \frac{1}{2})}
&= u^{n-1} \left( \frac{q^n(1-t)}{1-qt} \mathcal{M}_{(1,\frac{1}{2})} - t^{1-n}\mathcal{M}_{(\frac{3}{2})} \right) \CR
&= u^{n-1} \left( \frac{q(1-t)(q^{n-1}-t^{1-n})}{1-qt} p_1\pi_1 - \frac{q^n(1-t)+t^{1-n}(1-q)}{1-qt} \pi_2 \right).
\end{align}
On the other hand \eqref{F1n-ansatz} implies
\begin{equation}
F_{1,n} \cdot \mathcal{M}_{(\frac{3}{2}, \frac{1}{2})}
= -\pi_2 + \alpha_n p_1 \pi_1 + \cdots.
\end{equation}
Hence, for the matching of the $\pi_2$-term, we need a cubic term 
$- \beta_n \pi_2 \frac{\partial}{\partial \pi_1}\frac{\partial}{\partial \pi_2}$ in fermions. 
Note that when $n=0, 1$, we have $\beta_n = \gamma_n=0$. 

\subsection{$E_{2,n}$}

The Pieri formula for $E_{2,n}$ tells
\begin{align}
E_{2,n} \cdot \mathcal{M}_{(\frac{1}{2})} &= (u(t/q)^{\frac{1}{2}})^n (q/t)^n (1-t) \mathcal{M}_{(1)}, 
\label{E2nPieri-1}\\
E_{2,n} \cdot \mathcal{M}_{(\frac{3}{2})} &= (u(t/q)^{\frac{1}{2}})^n (q^2/t)^n (1-t) \mathcal{M}_{(2)}, 
\label{E2nPieri-2}\\
E_{2,n} \cdot \mathcal{M}_{(1, \frac{1}{2})} &= (u(t/q)^{\frac{1}{2}})^n (q/t^2)^n (1-t^2) \mathcal{M}_{(1,1)}.
\label{E2nPieri-3}
\end{align}
From \eqref{E2nPieri-1} we may assume
\begin{align}\label{E2n-ansatz}
& (1-t)^{-1}(t/q)^n (u(t/q)^{\frac{1}{2}})^{-n} E_{2,n} \CR
&= p_1 \frac{\partial}{\partial \pi_1} 
+ (\alpha_n p_1^2 + \beta_n p_2) \frac{\partial}{\partial \pi_2} 
+ (\gamma_n p_1^2 + \delta_n p_2) \frac{\partial}{\partial p_1}\frac{\partial}{\partial \pi_1} + \cdots. 
\end{align}
The degree counting allows the term proportional to $\pi_2 \pi_1$. But it is clear that $E_{2,n}$ never creates such term. 
Then the equations \eqref{E2nPieri-2} and \eqref{E2nPieri-3} imply that
\begin{align}
1+\gamma_n &= \frac{1}{2}(q^n + t^{-n}) + \frac{1}{2}\frac{q^n - t^{-n}}{q^{-1}-t} + \frac{1}{2}\frac{q^n - t^{-n}}{q-t^{-1}},
\\
\delta_n &=\frac{1}{2}(q^n - t^{-n}) - \frac{1}{2}\frac{q^n - t^{-n}}{q^{-1}-t} - \frac{1}{2}\frac{q^n - t^{-n}}{q-t^{-1}},
\end{align}
and 
\begin{align}
\beta_n &= \frac{1}{2}(q^n + t^{1-n}) - \frac{1}{2}\frac{q^n - t^{-n}}{q^{-1}-t} - \frac{1}{2}\frac{q^n - t^{-n}}{q-t^{-1}},
\\
\alpha_n &= \frac{1}{2}(q^n - t^{1-n}) + \frac{1}{2}\frac{q^n - t^{-n}}{q^{-1}-t} + \frac{1}{2}\frac{q^n - t^{-n}}{q-t^{-1}}.
\end{align}
In particular we have
\begin{align}
\alpha_{1} &= \frac{q}{2}(1-t^{-1}),
\qquad
\beta_{1} = \frac{q}{2}(1+t^{-1}),
\\
\gamma_{1} &= \frac{1}{2}(q-1)(1-t^{-1}),
\qquad
\delta_{1} = \frac{1}{2}(q-1)(1+t^{-1}).
\end{align}
Hence,
\begin{align}
(1-t)^{-1} u^{-1} (t/q)^{\frac{1}{2}} E_{2,1}
&= p_1 \frac{\partial}{\partial \pi_1} 
+ \frac{q}{2}((1-t^{-1}) p_1^2 + (1+t^{-1}) p_2) \frac{\partial}{\partial \pi_2} \CR
& \qquad \qquad 
+ \frac{q-1}{2}((1-t^{-1}) p_1^2 + (1+t^{-1}) p_2) \frac{\partial}{\partial p_1}\frac{\partial}{\partial \pi_1} + \cdots. 
\label{E21-expansion}
\end{align}

Similarly we obtain
\begin{align}
\alpha_{-1} &= \frac{1}{2}(1-t)(q^{-1}+t+1) ,
\qquad
\beta_{-1} = \frac{1}{2}(1+t)(q^{-1}+t-1),
\\
\gamma_{-1} &= \frac{1}{2}(1-t)(q^{-1}-1),
\qquad
\delta_{-1} = \frac{1}{2}(1+t)(q^{-1}-1).
\end{align}
Hence,
\begin{align}
&(1-t)^{-1}(u(q/t)^{\frac{1}{2}}E_{2,-1}) \CR
&= p_1\frac{\partial}{\partial \pi_1} + (1-t) p_1^2 \frac{\partial}{\partial \pi_2} 
+ \frac{1}{2} \left( (1-t) p_1^2 + (1+t) p_2  \right) 
\left( (q^{-1} +t -1) \frac{\partial}{\partial \pi_2} + (q^{-1}-1)\frac{\partial}{\partial p_1}\frac{\partial}{\partial \pi_1}\right),
%&= p_1\frac{\partial}{\partial \pi_1} + \frac{1}{2} \left((1-t)(q^{-1}+t+1) p_1^2 + (1+t)(q^{-1}+t-1)p_2 \right)\frac{\partial}{\partial \pi_2}  \CR
%& \qquad + \frac{1}{2} \left((q^{-1}-1)(1-t)p_1^2 + (q^{-1}-1)(1+t)p_2 \right)\frac{\partial}{\partial p_1}\frac{\partial}{\partial \pi_1}.
\end{align}
which exactly agrees with
\begin{align}
u \left[ H_{1,-1}, E_{2,0} \right] &= (-u) (q^{\frac{1}{2}} - q^{-\frac{1}{2}})
 (t^{\frac{1}{2}} - t^{-\frac{1}{2}}) E_{2,-1} \CR
 &=(1-q^{-1})(1-t) (u(q/t)^{\frac{1}{2}}E_{2,-1}).
\end{align}

\subsection{$F_{2,n}$}

The Pieri coefficients for $F_{2,n}$ is 
\begin{align}
\widetilde{\psi}_k^{(2,n)}(q,t) 
&=  (-1)^{F(k)+1} u^{n+1} t^{(n+1)(\frac{1}{2}-k)} q^{n\lambda_k-\frac{1}{2}(n+1)} \CR
& \qquad \times (1-  q^{\lambda_k} t^{-k+\ell(\lambda)}) 
\prod_{i=k+1}^{\ell(\lambda)} \frac{1 - q^{\lambda_k - \lambda_i}t^{i-k-1}}{1 - q^{\lambda_k -\lambda_i}t^{i-k}},
\end{align}
which implies
\begin{align}
F_{2,n} \cdot \mathcal{M}_{(1)} &= u^{n+1}t^{-\frac{1}{2}(n+1)} q^{\frac{1}{2}(n-1)} 
(q-1) \mathcal{M}_{(\frac{1}{2})},
\\
F_{2,n} \cdot \mathcal{M}_{(2)} &= u^{n+1} t^{-\frac{1}{2}(n+1)} q^{\frac{1}{2}(3n-1)} 
(q^2-1) \mathcal{M}_{(\frac{3}{2})},
\\
F_{2,n} \cdot \mathcal{M}_{(1,1)} &= u^{n+1}t^{-\frac{3}{2}(n+1)} q^{\frac{1}{2}(n-1)} 
(q-1) \mathcal{M}_{(1,\frac{1}{2})}.
\end{align}
We can assume that
\begin{align}
F_{2,n} &= u^{n+1}t^{-\frac{1}{2}(n+1)} q^{\frac{1}{2}(n-1)}(q-1)
\left(\pi_1\frac{\partial}{\partial p_1} + \pi_2 
\left( 2\alpha_n\frac{\partial}{\partial p_2} + \beta_n \frac{\partial^2}{\partial p_1^2} \right) \right. \CR
& \qquad \qquad \left. +\pi_1 p_1 \left( 2\gamma_n\frac{\partial}{\partial p_2} + \delta_n \frac{\partial^2}{\partial p_1^2} \right) 
+ \cdots \right).
\end{align}
The second and the third conditions imply
\begin{align}
\alpha_n &= \frac{1}{2}t^{-n-1} + \frac{1}{2} \left(\frac{q^{-1}(q^n - t^{-n})}{q^{-1}-t} - \frac{q^{n+1} - t^{-n-1}}{q^{-1}-t} \right), 
\\
\beta_n &= \frac{1}{2}(q^{-1}-1) \left( \frac{q^n - t^{-n}}{q^{-1} - t} +  \frac{q^{n+1} - t^{-n-1}}{q^{-1} - t} \right),
\\
\gamma_n &= \frac{1}{2}(1+q^{-1})\left( \frac{q^{n+1} - t^{-n-1}}{q^{-1} -t } \right), 
\\
1 + \delta_n &= \frac{1}{2} t^{-n-1} + \frac{1}{2}\left( \frac{q^n - t^{-n}}{q^{-1} -t} +  \frac{q^{n+1} - t^{-n-1}}{q^{-1}-t} \right). 
\end{align}
We have
\begin{equation}
\alpha_{-1} = \frac{1}{2}(1+q^{-1}), \qquad \beta_{-1} = \frac{1}{2}(1-q^{-1}), \qquad
\gamma_{-1} = \delta_{-1} =0,
\end{equation}
which is consistent with the formula of $F_{2,-1}$. 
On the other hand, we have
\begin{align}
\alpha_{0} &= \frac{1}{2}t^{-1}(1+q), \qquad \beta_{0} = \frac{1}{2}t^{-1}(q-1), 
\\
\gamma_{0} &= - \frac{1}{2}t^{-1}(1+q), \qquad \delta_{0} = \frac{1}{2}t^{-1}(1-q) -1.
\end{align}
Hence we obtain
\begin{align}\label{F20}
F_{2,0} &= u t^{-\frac{1}{2}} q^{\frac{1}{2}}(1-q^{-1})
\left(\pi_1\frac{\partial}{\partial p_1} + 
t^{-1} \pi_2 \left( (q+1) \frac{\partial}{\partial p_2} + \frac{1}{2}(q-1) \frac{\partial^2}{\partial p_1^2} \right) \right. \CR
& \qquad \qquad \left. - t^{-1} \pi_1 p_1 \left( (q+1) \frac{\partial}{\partial p_2} 
+ (\frac{1}{2}(q-1) +t) \frac{\partial^2}{\partial p_1^2} \right) 
+ \cdots \right).
\end{align}
As a consistency check we compute
\begin{align}
F_{2,0} \cdot \mathcal{M}_{(1,1)} & = u t^{-\frac{1}{2}} q^{\frac{1}{2}} (1-q^{-1}) \CR
& \times \left( \pi_1 p_1 + t^{-1}\frac{\pi_2}{2} (-(q+1)+(q-1)) \right. \CR
& \qquad \qquad \left. -t^{-1} \frac{\pi_1 p_1}{2}  (-(q+1)+(q-1) +2t)  \right) \CR
& = u t^{-\frac{1}{2}} q^{\frac{1}{2}} (1-q^{-1})t^{-1}(\pi_1 p_1 - \pi_2),
\end{align}
which shows the role of the last term in \eqref{F20} which is missing in $F_{2,-1}$. 

%%%%%%%%%%%%%%%%%%%%%%%%%%%%%%%%%%%%%%%%%%%%%%%%%%%%%%%%%%%%%%%%%%%%%%%%%%%

\section{Integral formula for $H_{1, +1}$ and $H_{2, +1}$}
\label{App:Integral-formula}

In \cite{Galakhov:2025phf},the following integral formula;\footnote{See the dictionary \eqref{R-dictionary+}. 
We have changed the notation $\theta_k$ in \cite{Galakhov:2025phf} to $\pi_k$.}
for the Hamiltonians $H_{1, +1}$ and $H_{2, +1}$ was obtained;
\begin{align}\label{H_2+-formula}
-u^{-1}(q/t)^{\frac{1}{2}} H_{2,+1} &= \oint \frac{dw}{w} V_B^{(-)}(w) V_B^{(+)}(w) 
\langle \varnothing \vert V_F^{(-)}(w) V_F^{(+)}(w) \vert \varnothing \rangle_{F}, \\
\label{H_1+-formula}
u^{-1}H_{1,+1} &= \oint \frac{dw}{w} V_B^{(-)}(w) V_B^{(+)}(w) 
\langle \varnothing \vert \widetilde{V}_F^{(-)}(w) \widetilde{V}_F^{(+)}(w) \vert \varnothing \rangle_{F},
\end{align}
where the bosonic vertex operator
\begin{equation}
V_B^{(-)}(w) := \exp \left( \sum_{k=1}^\infty \frac{1-t^{-k}}{k} p_k w^{2k} \right), 
\quad 
V_B^{(+)}(w) := \exp \left( \sum_{k=1}^\infty  (q^k -1) \frac{\partial}{\partial p_k} w^{-2k} \right),
\end{equation}
is the same as what was employed in the free field realization of the Ruijsenaars-Macdonald Hamiltonian \cite{Awata:1994xd},\cite{Awata:1995zk}. 
The fermionic vertex operators $V_F^{(\pm)}(w)$ and $\widetilde{V}_F^{(\pm)}(w)$ are 
defined by\footnote{We assume $\psi_m, \psi_m^\dagger$ and $\pi_k$ are anti-commuting and make the products \lq\lq normal ordered\rq\rq.}
\begin{align}
 V_F^{(-)}(w) &= \exp \left(
 (1-t^{-1}) \sum_{k=1}^\infty t^{1-k} w^{2k-1} \pi_k \cdot \sum_{m=0}^{k-1}t^m \psi_m \right),
\\
 V_F^{(+)}(w)&= \exp \left(
 (q-1) \sum_{k=1}^\infty q^{k-1} w^{-2k+1}  \left(\sum_{m=0}^{k-1} q^{-m} \psi_m^\dagger \right) \frac{\partial}{\partial \pi_k} \right), 
\\
 \widetilde{V}_F^{(-)}(w) &= \exp \left(
 (1-t^{-1}) \sum_{k=2}^\infty t^{2-k} w^{2k-1} \pi_k \cdot \sum_{m=0}^{k-2}t^m \psi_m \right),
\\
 \widetilde{V}_F^{(+)}(w)&= \exp \left(
 (q-1) \sum_{k=2}^\infty q^{k-2} w^{-2k+1} \left( \sum_{m=0}^{k-2} q^{-m} \psi_m^\dagger \right) \frac{\partial}{\partial \pi_k}  \right).
\end{align} 
The vacuum expectation values are computed with respect to the charged fermions $\psi_m, \psi_m^\dagger$ with
the anti-commutation relation $\{ \psi_m, \psi_n^\dagger \} = \delta_{mn}$.  
The fermion vacuum is defined by $\psi_m \vert \varnothing \rangle_F =0$ and we have
$\langle \varnothing \vert \psi_m \psi_n^\dagger \vert \varnothing \rangle_F = \delta_{mn}$.

Note that the bosonic parts of $H_{1,+1}$ and $H_{2,+1}$ are the same. 
Expanding $V_F^{(\pm)}(w)$ in $\psi_0$ and $\psi_0^\dagger$, we have 
\begin{align}
 V_F^{(-)}(w) &= \left(1+ (1-t^{-1})\sum_{k=1}^\infty t^{1-k} w^{2k-1} \pi_k \psi_0 \right) 
 \exp \left( (1-t^{-1}) \sum_{k=2}^\infty t^{1-k} w^{2k-1} \pi_k \cdot \sum_{m=1}^{k-1}t^m \psi_m \right),
\\
 V_F^{(+)}(w)&= \left(1+  (q-1) \sum_{k=1}^\infty q^{k-1} w^{-2k+1}\psi_0^\dagger  \frac{\partial}{\partial \pi_k} \right) 
 \exp \left( (q-1) \sum_{k=2}^\infty q^{k-1} w^{-2k+1}  \left(\sum_{m=1}^{k-1} q^{-m} \psi_m^\dagger \right) 
 \frac{\partial}{\partial \pi_k} \right), 
\end{align}
Hence the vacuum expectation values of $V_F^{(\pm)}(w)$ and $\widetilde{V}_F^{(\pm)}(w)$ are related by
\begin{align}\label{H_1andH_2}
& \langle \varnothing \vert V_F^{(-)}(w) V_F^{(+)}(w) \vert \varnothing \rangle_{F} \CR
&= \left( 1+ (1-t^{-1})(q-1) \sum_{k, \ell =1}^\infty t^{1-k} q^{\ell-1} w^{2(k-\ell)} \pi_k \frac{\partial}{\partial \pi_\ell}\right) 
\langle \varnothing \vert \widetilde{V}_F^{(-)}(w) \widetilde{V}_F^{(+)}(w) \vert \varnothing \rangle_{F}.
\end{align}

Since $\widetilde{V}_F^{(\pm)}(w)$ does not involve $\pi_1$, it is easy to see
\begin{equation}
\left[ H_{1,+1}, E_{1,0} \right] = \left[ H_{1,+1}, F_{1,1} \right] =0. 
\end{equation}
Furthermore, we can also prove\footnote{Compare this Proposition with Proposition \ref{Prop;commutativity}
for $H_{i, -1}$.}
\begin{prp}\label{Prop:commutativity+}
\begin{align}
\label{H21-anticom}
&\left[ H_{2,+1}, E_{2,0} \right] = \left[ H_{2,+1}, F_{2,-1} \right] =0, \\
\label{H11-anticom}
&\left[ H_{1,+1}, E_{1,-1} \right] = \left[ H_{1,+1}, F_{1,0} \right] =0. 
\end{align}
\end{prp}

\begin{proof}
We prove \eqref{H21-anticom} for $H_{2,+1}$ . The relation \eqref{H11-anticom} for $H_{1,+1}$ can be proved in the same manner 
with an appropriate shift (or redefinition) of the summation indices. 
Recall that $ E_{2,0} = \displaystyle{\sum_{n=1}^\infty} c_n[p] \frac{\partial}{\partial \pi_n}$. 
Hence, 
\begin{equation}
\left[ H_{2,+1}, E_{2,0} \right] = \sum_{n=1}^\infty \left[ H_{2,+1}, c_n[p] \right]  \frac{\partial}{\partial \pi_n}
+ \sum_{n=1}^\infty  c_n[p] \left[ H_{2,+1},  \frac{\partial}{\partial \pi_n}\right].
\end{equation}
Namely the commutator consists of two contributions; one comes from the commutation relation among
the power sum variables $p_k$ and the other from the anti-commutation relation among $\pi_k$. 
Using the formula
\begin{equation}\label{c_n-commutation}
\left[ V_B^{(+)}(w), c_n[p] \right] = (q-1)(1-t) \sum_{r=0}^{n-1} [[n-r]]_{(q,t)} c_{r} [p]  w^{-2(n-r)} V_B^{(+)}(w),
\end{equation}
and the relation $[[r]]_{(q,t)} = t^{r-1} b_r(q,t)$,
the first term is evaluated as 
\begin{align}\label{1st-term}
%&(q-1)(1-t) \oint \frac{dw}{w} \sum_{n=1}^\infty 
% \left( \sum_{r=1}^n [[r]]_{(q,t)} w^{-2r} c_{n-r} [p] \right) \langle \varnothing \vert
%V_1^{(-)}(w) V_1^{(+)}(w) \vert \varnothing \rangle_{F}\frac{\partial}{\partial \pi_n} \CR
& (q-1)(1-t) \oint \frac{dw}{w} \sum_{n=1}^\infty 
 \left( \sum_{m=0}^{n-1} t^{n-m-1} b_{n-m}(q,t) w^{-2(n-m)} c_{m} [p] \right) \CR
& \qquad  \times V_B^{(-)}(w) V_B^{(+)}(w) \langle \varnothing \vert
V_F^{(-)}(w) V_F^{(+)}(w) \vert \varnothing \rangle_{F}\frac{\partial}{\partial \pi_n} \CR
&= (q-1)(1-t) \oint \frac{dw}{w} \sum_{n=1}^\infty \sum_{m=0}^{n-1} 
\sum_{\ell=1}^{n-m} q^{\ell-1} t^{n-m-\ell} w^{-2(n-m)} c_{m} [p] \CR
& \qquad  \times V_B^{(-)}(w) V_B^{(+)}(w) \langle \varnothing \vert
V_F^{(-)}(w) V_F^{(+)}(w) \vert \varnothing \rangle_{F}\frac{\partial}{\partial \pi_n} \CR
&= (q-1)(1-t) \oint \frac{dw}{w} \sum_{n=1}^\infty \sum_{k=1}^n
\sum_{m=0}^{k-1} q^{n-k} t^{k-m-1} w^{-2(n-m)} c_{m} [p] \CR
& \qquad  \times V_B^{(-)}(w) V_B^{(+)}(w) \langle \varnothing \vert
V_F^{(-)}(w) V_F^{(+)}(w) \vert \varnothing \rangle_{F}\frac{\partial}{\partial \pi_n}.
\end{align}

On the other hand the second term is\footnote{The sign here is fixed by the normal ordering prescription
between $\psi_m, \psi^\dagger_m$ and $\pi_k$.}
\begin{align}
& (t^{-1}-1) \oint \frac{dw}{w} \sum_{n=1}^\infty c_n[p]  t^{1-n} w^{2n-1} 
 V_B^{(-)}(w) V_B^{(+)}(w)  \langle \varnothing \vert\left( \sum_{m=0}^{n-1} t^m \psi_m \right)
V_F^{(-)}(w) V_F^{(+)}(w) \vert \varnothing \rangle_{F} \CR
&= (q-1)(1- t) \oint \frac{dw}{w} \sum_{m=1}^\infty c_{m} [p]  t^{-m} w^{2m-1} 
\sum_{k=0}^{m-1} (t/q)^k \sum_{n=k+1}^\infty q^{n-1} w^{-2n+1} \CR
& \qquad \times  V_B^{(-)}(w) V_B^{(+)}(w)  \langle \varnothing \vert
V_F^{(-)}(w) V_F^{(+)}(w) \vert \varnothing \rangle_{F} \frac{\partial}{\partial \pi_n},
\end{align}
where we have used $\langle \varnothing \vert \psi_m \psi_n^\dagger \vert \varnothing \rangle_F = \delta_{mn}$.
By making the shift $k \to k+1$ we see that the coefficient of $\frac{\partial}{\partial \pi_n}$ is\footnote{If we fix $n$, the range of $k$ is bounded;
$1 \leq k+1 \leq n$ and $m$ satisfies $k+1 \leq m$.}
\begin{equation}
\sum_{k=1}^n \sum_{m=k}^\infty c_m[p] t^{k-m-1} q^{n-k} w^{-2(n-m)}.
\end{equation}
%For example lower coefficients are
%\begin{enumerate}
%\item $n=1$
%$$ \sum_{m=1}^\infty c_m[p] t^{-m} w^{2m-2} $$
%\item $n=2$
%$$ \sum_{m=1}^\infty c_m[p] t^{-m} q w^{2m-4} +  \sum_{m=2}^\infty c_m[p] t^{1-m} w^{2m-4}$$
%\item $n=3$
%$$ \sum_{m=1}^\infty c_m[p] t^{-m} q^2 w^{2m-6} +  \sum_{m=2}^\infty c_m[p] t^{1-m}q w^{2m-6}
%+ \sum_{m=3}^\infty c_m[p] t^{2-m} w^{2m-6}$$
%\end{enumerate}
%The corresponding coefficients in \eqref{1st-term} are
%\begin{enumerate}
%\item $n=1$
%$$c_0[p] w^{-2} $$
%\item $n=2$
%$$ (t+q) c_0[p] w^{-4} + c_1[p] w^{-2}$$
%\item $n=3$
%$$ (t^2 + qt+q^2) c_0[p] w^{-6}  + (t+q) c_1[p] w^{-4} + c_2[p] w^{-2}$$
%\end{enumerate}
Hence, taking the sum of two terms we have
\begin{align}
\left[ H_{2,+1}, E_{2,0} \right] &=  (q-1)(1- t) \sum_{n=1}^\infty  
 \oint \frac{dw}{w} \sum_{k=1}^n \sum_{m=0}^\infty c_m[p] t^{k-m-1} w^{-2(k-m)} \CR
& \qquad \times V_B^{(-)}(w) V_B^{(+)}(w)  \langle \varnothing \vert
V_F^{(-)}(w) V_F^{(+)}(w) \vert \varnothing \rangle_{F}  q^{n-k} \frac{\partial}{\partial \pi_n}. 
\end{align}
By Lemma \ref{C-lemma} we see the right hand side vanishes. For example in the bosonic sector where we have
$\langle \varnothing \vert V_2^{(-)}(w) V_2^{(+)}(w) \vert \varnothing \rangle_{F} =1$,  the right hand side is
\begin{equation}
 (q-1)(t^{-1}- 1) \sum_{n=1}^\infty  
 \sum_{k=1}^n \left( \sum_{m=0}^\infty c_m[p] t^{-m} C_{k-m}[p, \partial/\partial p] \right)  (t/q)^k q^{n} \frac{\partial}{\partial \pi_n} =0. 
\end{equation}
Higher order terms expanded in the fermion bi-linears $\pi_k \frac{\partial}{\partial \pi_\ell}$ are similarly vanishing. 

Since $F_{2,-1} = - \displaystyle{\sum_{n=1}^\infty}~\widetilde{c}_n [\partial/\partial p] \pi_n$, we have
\begin{equation}
\left[ H_{2,+1}, F_{2,-1} \right] = - \sum_{n=1}^\infty \left[ H_{2,1}, \widetilde{c}_n [\partial/\partial p] \right]  \pi_n 
- \sum_{n=1}^\infty  \widetilde{c}_n [\partial/\partial p]  \left[ H_{2,1},  \pi_n \right].
\end{equation}
This time we use the commutation relation 
\begin{equation}
\left[ V_B^{(-)}(w), \widetilde{c}_n [\partial/\partial p]  \right] 
= - (q^{-1}-1)(1-t^{-1}) \sum_{m=0}^{n-1} [[n-m]]_{(q^{-1},t^{-1})}  \widetilde{c}_m [\partial/\partial p] w^{2(n-m)} V_B^{(-)}(w),
\end{equation}
to obtain the first term;
\begin{align}
& (q^{-1}-1)(1-t^{-1})  \oint \frac{dw}{w} \sum_{n=1}^\infty  \sum_{m=0}^{n-1} t^{1-n+m} b_{n-m} (q^{-1},t^{-1})  w^{2(n-m)}  \widetilde{c}_m [\partial/\partial p]  \CR
& \qquad \times V_B^{(-)}(w) V_B^{(+)}(w)  \langle \varnothing \vert V_F^{(-)}(w) V_F^{(+)}(w) \vert \varnothing \rangle_{F} \pi_n.
\end{align}
The second term is
\begin{align}
& -(q-1) \sum_{n=1}^\infty \widetilde{c}_n [\partial/\partial p ] q^{n-1}   \oint \frac{dw}{w}   w^{-2n+1}  V_B^{(-)}(w) V_B^{(+)}(w)  \CR
& \qquad  \langle \varnothing \vert V_F^{(-)}(w) V_F^{(+)}(w) \sum_{m=0}^{n-1} q^{-m} \psi_{m}^\dagger  \vert \varnothing  \rangle_{F} \CR
&= (q-1)(t^{-1}-1)  \sum_{m=1}^\infty \widetilde{c}_m [\partial/\partial p ] q^{m-1}   \oint \frac{dw}{w} V_B^{(-)}(w) V_B^{(+)}(w)  \CR
& \qquad  \langle \varnothing \vert V_F^{(-)}(w) V_F^{(+)}(w) \vert \varnothing \rangle_{F} \sum_{k=0}^{m-1} (t/q)^k
\sum_{n=k+1}^\infty t^{1-n} w^{2(n-m)} \pi_n.
\end{align}
The sum of two terms is
\begin{align}
&(q-1)(1-t)  \sum_{n=1}^\infty \sum_{k=1}^n \sum_{m=0}^\infty \widetilde{c}_m [\partial/\partial p ] t^{k-m-1} \CR
& \qquad  \oint \frac{dw}{w} w^{-2(k-m)}  V_B^{(-)}(w) V_B^{(+)}(w) 
 \langle \varnothing \vert V_F^{(-)}(w) V_F^{(+)}(w) \vert \varnothing \rangle_{F} q^{n-k} \pi_n,
\end{align}
which vanishes by the same argument as before.\footnote{To be checked again.}
\end{proof}

If we expand the Hamiltonians $H_{1, +1}$ and $H_{2, +1}$ in even numbers of fermions with the Ruijsenaars-Macdonald Hamiltonian
as the initial condition, the commutativity in Prop \ref{Prop:commutativity+} gives the recursion relations among the expansion coefficients.
It is possible to write down a solution in terms of the multiple commutator of type \eqref{c_n-commutation}, which becomes
combinatorially involved in higher orders. The use of the vacuum expectation values for the auxiliary fermions 
$\psi_m, \psi_m^\dagger$ provides  a sophisticated way of representing such complicated coefficients. 

%%%%%%%%%%%%%%%%%%%%%%%%%%%%%%%%%%%%%%%%%%%%%%%%%%%%%%%%%%%%%%%%%%%%%%%%%%%%

Let us look at the fermionic terms in the formula of $(q,t)$-inverted Hamiltonian proposed by \cite{Galakhov:2025phf};
\begin{equation}
\widehat{\mathcal{H}}^{+}_i = \oint \frac{dw}{w} \langle \varnothing \vert V_i^{(-)}(w) V_i^{(+)}(w) \vert \varnothing \rangle_{F},
\qquad i=1,2.
\end{equation}
 
The total sum for $\widehat{\mathcal{H}}^{+}_2$ is 
\begin{align}
& (1-t^{-1})(q-1) \left(\pi_1 \frac{\partial}{\partial \pi_1} + (1+ t^{-1}q) \pi_2 \frac{\partial}{\partial \pi_2}
+ (1+ t^{-1}q + t^{-2}q^2) \pi_3 \frac{\partial}{\partial \pi_3} \right. \CR
& \qquad + q w^{-2} \pi_1\frac{\partial}{\partial \pi_2} + t^{-1} w^{2} \pi_2\frac{\partial}{\partial \pi_1}
+ q^2 w^{-4} \pi_1 \frac{\partial}{\partial \pi_3} + t^{-2} w^{4} \pi_3 \frac{\partial}{\partial \pi_1} \CR
& \qquad \left. + q(1+t^{-1}q) w^{-2} \pi_2\frac{\partial}{\partial \pi_3} + t^{-1}(1+ t^{-1}q) w^{2} \pi_3\frac{\partial}{\partial \pi_2}\right).
\end{align}
For $\widehat{\mathcal{H}}^{+}_1$ it is 
\begin{align}
(1-t^{-1})(q-1) t^{-1} \left( \pi_2 \frac{\partial}{\partial \pi_2} + (1+ t^{-1}q) \pi_3 \frac{\partial}{\partial \pi_3} 
+ q w^{-2} \pi_2\frac{\partial}{\partial \pi_3} + t^{-1} w^{2} \pi_3\frac{\partial}{\partial \pi_2} \right).
\end{align}

Hence, the structure of the fermionic terms in the vertex operator $V_i^{(\pm)}(w)$ leads that
\begin{equation}
b_k(q,t) := 1+ (q/t) + \cdots + (q/t)^{k-1} = t^{1-k} \frac{q^k - t^k}{q-t} = t^{1-k} [[k]]_{(q,t)}.
\end{equation}
appears in the coefficients in the expansion in $\pi_k$. 
The general form of the fermion bi-linear terms of $\widehat{\mathcal{H}}^{+}_2$ is 
\begin{align}\label{fermion-bilinear}
(1-t^{-1})(q-1) & \left( \sum_{k=1}^\infty b_k(q,t) \cdot \pi_k \frac{\partial}{\partial \pi_k}
+ \sum_{1 \leq k < \ell < \infty} b_k(q,t) q^{\ell-k} w^{2(k-\ell)} \cdot \pi_k \frac{\partial}{\partial \pi_\ell} \right. \CR
& \qquad \left. + \sum_{1 \leq \ell < k < \infty} b_\ell(q,t) t^{\ell-k} w^{2(k-\ell)}\cdot \pi_k \frac{\partial}{\partial \pi_\ell} \right),
\end{align}
For $\widehat{\mathcal{H}}^{+}_1$  we have
\begin{align}
(1-t^{-1})(q-1)t^{-1} & \left( \sum_{k=2}^\infty b_k(q,t) \cdot \pi_k \frac{\partial}{\partial \pi_k}
+ \sum_{2 \leq k < \ell < \infty} b_k(q,t) q^{\ell-k} w^{2(k-\ell)} \cdot \pi_k \frac{\partial}{\partial \pi_\ell} \right. \CR
& \qquad \left. + \sum_{2 \leq \ell < k < \infty} b_\ell(q,t) t^{\ell-k} w^{2(k-\ell)}\cdot \pi_k \frac{\partial}{\partial \pi_\ell} \right).
\end{align}
Thus according the relative magnitudes of $k$ and $\ell$ we have an additional monomial factor $q^{\ell-k}$ or  $t^{\ell-k}$. 
Combined with the bosonic part \eqref{bosonic-part}, the fermion bi-linear terms of $\widehat{\mathcal{H}}^{+}_2$  becomes 
\begin{align}\label{H2+}
(1-t^{-1})(q-1) & \left( \sum_{1 \leq k \leq \ell < \infty} b_k(q,t) q^{\ell-k} 
C_{\ell -k} [p, \partial/\partial p]\cdot \pi_k \frac{\partial}{\partial \pi_\ell} \right. \CR
& \qquad \left. + \sum_{1 \leq \ell < k < \infty} b_\ell(q,t) t^{\ell-k} C_{\ell -k} [p, \partial/\partial p]
\cdot \pi_k \frac{\partial}{\partial \pi_\ell} \right).
\end{align}
The fermion bi-linear terms of $\widehat{\mathcal{H}}^{+}_1$ are given by the shift of indices $k, \ell \to k+1, \ell+1$.

When $k=1$ the second term in \eqref{H2+} disappears;
\begin{equation}
(1-t^{-1})(q-1) \left( \sum_{1 \leq \ell < \infty} q^{\ell-1} 
C_{\ell -1} [p, \partial/\partial p] \cdot \pi_1 \frac{\partial}{\partial \pi_\ell} \right), 
\end{equation}
which agrees with \eqref{pi1-term}.
On the other hand for $k=2$ there is a contribution from the second term;
\begin{align}
(1-t^{-1})(q-1) \left( \sum_{2 \leq \ell < \infty} (1+ q/t) q^{\ell-2} 
C_{\ell -2} [p, \partial/\partial p]\cdot \pi_2 \frac{\partial}{\partial \pi_\ell} 
+ t^{-1} C_{-1} [p, \partial/\partial p]
\cdot \pi_2 \frac{\partial}{\partial \pi_1} \right).
\end{align}

%%%%%%%%%%%%%%%%%%%%%%%%%%%%

\section{Involution operator $\mathsf{T}_q$}
\label{App:Tq}

In \cite{Alarie-Vezina:2019ohz} they introduced four super charges; 
\begin{align}
\mathcal{Q}_1 &:= \sum_{i=1}^N \theta_i \tau_i^{-1}, \qquad
\mathcal{Q}_2 := \sum_{i=1}^N A_i(t^{-1}) \frac{\partial}{\partial \theta_i}, \\
\mathcal{Q}_3 &:= \sum_{i=1}^N A_i(t) \xi_i \tau_i \frac{\partial}{\partial \theta_i}, \qquad
\mathcal{Q}_4 := \sum_i \theta_i = \pi_1,
\end{align}
where $\tau_i$ is the $q$-shift operator of $x_i \to qx_i$, 
\begin{equation}
A_i(t) := \prod_{j \neq i} \frac{tx_i - x_j}{x_i - x_j},
\end{equation}
and
\begin{equation}\label{xi-op}
\xi_i := \sum_{I \subset \{1,2,\cdots N \}} \prod_{j \in I, j \neq i}
\frac{(qtx_i -x_j)(x_i-x_j)}{(qx_i -x_j)(tx_i-x_j)} \theta_I \rho_I.
\end{equation}
Since 
\begin{equation}
\left[ \mathcal{Q}_2, \mathcal{Q}_1 \right]_{+} = \sum_{i=1}^N A_i(t^{-1}) \cdot \tau_i^{-1} 
+ \sum_{i,j=1}^N \left[A_i(t^{-1}), \tau_j^{-1} \right] \theta_j \frac{\partial}{\partial \theta_i}
,
\end{equation}
the relation of the first two super charges $\mathcal{Q}_1$ and $\mathcal{Q}_2$ to the Hamiltonian of the Macdonald polynomials is clear.
On the other hand for the remaining super charges $\mathcal{Q}_3$ and $\mathcal{Q}_4$ the relation is not obvious. 
As was shown in \cite{Alarie-Vezina:2019ohz},  $\mathcal{Q}_3$ and $\mathcal{Q}_4$ are related to the $(q,t)$ inverted form 
of $\mathcal{Q}_1$ and $\mathcal{Q}_2$.\footnote{Since the Macdonald polynomials are invariant under the inversion $(q,t) \to (q^{-1}, t^{-1})$, 
they are also eigenfunctions of the $(q,t)$ inverted Hamiltonian.} 
In fact we have
\begin{lem}\label{Q-inversion}
\begin{align*}
\mathsf{T}_q \mathcal{Q}_4 &= \mathcal{Q}_1^{(q^{-1}, t^{-1})} \mathsf{T}_q
= \left( \sum_{i=1}^N \theta_i \tau_i \right) \mathsf{T}_q, \\
\mathsf{T}_q \mathcal{Q}_3 &= \mathcal{Q}_2^{(q^{-1}, t^{-1})} \mathsf{T}_q
= \left(\sum_{i=1}^N A_i(t) \frac{\partial}{\partial \theta_i}\right) \mathsf{T}_q , 
\end{align*}
where
\begin{equation}
\mathsf{T}_q  := \sum_{I \subset \{1,2, \ldots, N \}} \tau_I \theta_I \rho_I. 
\end{equation}
\end{lem}
Note that $\theta_I \rho_I$ is the projection operator to the fermion sector labelled by $I$ and 
on this sector $\mathsf{T}_q: x_i \to qx_i~(i \in I),~x_j \to x_j ~(j \notin I)$. Hence, $\mathsf{T}_q$
preserves the fermion number. On the subspace of a fixed fermion number $m$, the action of $\mathsf{T}_q$ is
\begin{equation}
\mathsf{T}_q^{(m)}  := \sum_{|I| = m} \tau_I \theta_I \rho_I, \qquad \mathsf{T}_q = \sum_{m=0}^N \mathsf{T}_q^{(m)}.
\end{equation}
In particular $\mathsf{T}_q^{(0)} = \tau_\varnothing\theta_\varnothing \rho_\varnothing = 1 \cdot \rho_\varnothing$ 
is the projection to the subspace with no fermions.

\begin{proof}
It is enough to prove the Lemma on the fermionic sector $\theta_J$ for any $J \subset \{1, 2, \ldots N \}$. 
Note that $\mathsf{T}_q \theta_J = \theta_J  \tau_J$. Hence, 
\begin{equation}
\mathsf{T}_q  \mathcal{Q}_4  \theta_{J} = \mathsf{T}_q  \sum_{ i \notin J} \theta_i \theta_J
=  \sum_{ i \notin J} \tau_{J} \tau_i  \theta_i \theta_J
=  \left( \sum_{i=1}^N \theta_i \tau_i \right) \mathsf{T}_q \theta_J. 
\end{equation}
Namely the commutation with $\mathsf{T}_q$ produces the shift operator $\tau_i$. 

To prove the second relation we first look at 
\begin{align}
\mathcal{Q}_3 \theta_J = \sum_{i=1}^N \prod_{k \neq i} \frac{tx_i - x_k}{x_i -x_k}
\sum_{I \subset \{1, 2, \ldots N \}} \prod_{j \in I, j \neq i} \frac{(qt x_i -x_j)(x_i - x_j)}{(q x_i - x_j)(tx_i -x_j)}
\theta_I \rho_I \tau_i \partial_{\theta_i} \theta_J. 
\end{align}
The single term with $I = J \setminus \{i\}$ survives in the summation over the subsets $I$. 
We divide the indices $k \neq i$ into two sets $ j \in J \setminus \{i\}$ and $k \notin J$. 
Thus we have 
\begin{align}
\mathcal{Q}_3 \theta_J = \sum_{i=1}^N \prod_{k \notin J} \frac{tx_i - x_k}{x_i -x_k}
\prod_{j \in J \setminus \{i\} } \frac{qt x_i -x_j}{q x_i - x_j}\tau_i \partial_{\theta_i} \theta_J. 
\end{align}
Then we have 
\begin{align}
\mathsf{T}_q  \mathcal{Q}_3 \theta_J &=  \sum_{i=1}^N  \tau_{ J \setminus \{i\}} \prod_{k \notin J} \frac{tx_i - x_k}{x_i -x_k}
\prod_{j \in \in J \setminus \{i\} } \frac{qt x_i -x_j}{q x_i - x_j}\tau_i \partial_{\theta_i} \theta_J \CR
&=  \sum_{i=1}^N  \prod_{k \notin J} \frac{tx_i - x_k}{x_i -x_k}
\prod_{j \in \in J \setminus \{i\} } \frac{t x_i -x_j}{ x_i - x_j} \tau_J \partial_{\theta_i} \theta_J  \CR
&=  \sum_{i=1}^N  \prod_{k \neq i} \frac{tx_i - x_k}{x_i -x_k} \partial_{\theta_i} \mathsf{T}_q \theta_J.
\end{align}
Hence we obtain the second relation. 
\end{proof}

\begin{rmk}
Note that $\theta_I \rho_I$ are mutually orthogonal projections to the fermion sector labelled by $I$;
\begin{equation}
(\theta_I \rho_I)(\theta_J \rho_J) = \delta_{I,J} (\theta_I \rho_I).
\end{equation}
Hence we have $\mathsf{T}_q^{-1} = \mathsf{T}_{q^{-1}} = \sum_{I} \tau_I^{-1} \theta_I \rho_I$. 
In fact
\begin{equation}
\mathsf{T}_q \mathsf{T}_{q^{-1}} = (\sum_{I} \tau_I\theta_I \rho_I)
(\sum_{J} \tau_J^{-1} \theta_J \rho_J) = \sum_{I} \theta_I \rho_I = 1.
\end{equation}
\end{rmk}

Now we have
\begin{equation}
\mathcal{Q}_3 = \mathsf{T}_{q^{-1}} \left(\sum_{i=1}^N A_i(t) \frac{\partial}{\partial \theta_i}\right) \mathsf{T}_q,
\end{equation}
and $\mathcal{Q}_3$ decreases the fermion number by one. Hence the leading term is the action on the fermion number one subspace.
The fermionic power sum $\pi_n, (1 \leq n \leq N) $ is a basis of the fermion number one subspace 
and by using \eqref{q-shift-powersum} we have
\begin{align}
\mathcal{Q}_3 \pi_n &= \left(\sum_{i=1}^N A_i(t) \frac{\partial}{\partial \theta_i}\right) \sum_{j=1}^N
\tau_j \theta_j x_j^{n-1} \CR
&= \sum_{i=1}^N  A_i(t) \tau_i x_i^{n-1} \CR
&= \sum_{i=1}^N  \sum_{m=1}^\infty A_i(t) x_i^{n+m-1} \widetilde{c}_m^\vee[\partial/\partial p]\CR
&= \frac{t^N}{t-1}\sum_{m=1}^\infty c_{n+m-1}^\vee[p] \widetilde{c}_m^\vee[\partial/\partial p].
\end{align}
We see that on the fermion number one subspace,
\begin{equation}
\mathcal{Q}_3 = \frac{t^N}{t-1} \sum_{n=1}^N \sum_{m=1}^\infty c^\vee_{n+m-1}[p] \widetilde{c}_m^\vee[\partial/\partial p]
\frac{\partial}{\partial \pi_n},
\end{equation}
which agrees with the leading term of $F_{1,2}$ computed in section \ref{section;integral-formula}.

%%%%%%%%%%%%%%%%%%%%%%%%%%%%%%%%%%%%%%%%%%%%%%%%%%%%%%%%%%%%%%%%%%%%%%%
In \cite{Alarie-Vezina:2019ohz} (see Eq.(4.9)), it is proved that the operator $\mathsf{T}_q$ acts 
on the super Macdonald polynomials as the inversion of parameters $(q,t) \to (q^{-1}, t^{-1})$ up to some power of $q$.
\begin{equation}
\label{Tq-onMac}
\mathsf{T}_q \mathcal{M}_{\Lambda}(x, \theta;q,t) = q^{|\Lambda^{\mathsf{a}}|} \mathcal{M}_{\Lambda}(x, \theta;q^{-1},t^{-1}).
\end{equation}
To check \eqref{Tq-onMac} explicitly at lower levels we use the following lemme;
\begin{lem}\label{power-sum}
On the one fermion sector, the action of $\mathsf{T}_q$ is described as follows;
\begin{equation}
\mathsf{T}_q^{(1)} := \sum_{n=1}^\infty \sum_{k=0}^\infty \widetilde{c}_k^\vee \left[ \partial/ \partial p \right] 
\pi_{n+k}~q^{n-1}\frac{\partial}{\partial \pi_n},
\end{equation}
where
\begin{equation}
\sum_{k=0}^\infty \widetilde{c}_k^\vee \left[ \partial/ \partial p \right] z^{-k}
= \exp \left( \sum_{n=1}^\infty (q^{n}-1) \frac{\partial}{\partial p_n} z^{-n} \right).
\end{equation}
\end{lem}

\begin{proof}
We evaluate the left hand side on the sector $\theta_J$.
\begin{align*}
\mathsf{T}_q~\pi_n \theta_J %&= \mathsf{T}_q \left( \sum_{i \notin J} x_i^{n-1} \theta_i \right) \theta_J \\
&= q^{n-1} \left( \sum_{i \notin J} x_i^{n-1} \theta_i \right) \tau_i \mathsf{T}_q \theta_J \\
&= q^{n-1} \sum_{i \notin J} x_i^{n-1} \theta_i \left( \sum_{k=0}^\infty x_i^k 
\widetilde{c}_k^\vee [\partial/\partial p] \right)\mathsf{T}_q \theta_J \\
&= q^{n-1} \sum_{k=0}^\infty \pi_{n+k} \widetilde{c}_k^\vee [\partial/\partial p] \mathsf{T}_q \theta_J.
\end{align*}
\end{proof}

Explicit formulas following from Lemma \ref{power-sum} are
\begin{align}
\label{Tq-action-1}
\mathsf{T}_q(\pi_k p_\ell) &= q^{k-1} ( \pi_k p_\ell + (q^\ell -1) \pi_{k + \ell}), \\
\label{Tq-action-2}
\mathsf{T}_q( \pi_m p_k p_\ell) &= q^{m-1} \big(  \pi_m p_k p_\ell + (q^k-1) \pi_{k+m} p_\ell  + (q^\ell-1) \pi_{\ell+m} p_k  
+ (q^k-1) (q^\ell-1)\pi_{k+\ell+m} \big).
\end{align}

By induction we obtain 
\begin{equation}
\mathsf{T}_q =\sum_{m=0}^\infty \mathsf{T}_q^{(m)}
= \mathsf{P}^{(0)} + \mathsf{T}_q^{(1)}\mathsf{P}^{(1)} + \frac{1}{2}:\!(\mathsf{T}_q^{(1)})^2: \mathsf{P}^{(2)} + \cdots 
%\neq ~:\!\exp \left(\mathsf{T}_q^{(1)}\right)\!:,
\end{equation}
where $: \bullet :$ denotes the normal ordering.
Note that each term involves the projection $\mathsf{P}^{(m)}$ to the subspace with fermion number $m$,
which plays a  crucial role for the correct action of $\mathsf{T}_q$. 
For example
\begin{align}
\mathsf{T}_q( \pi_m \pi_k p_\ell) &= \frac{1}{2} :\left(\sum_{n=1}^\infty
\left(\pi_{n} + \sum_{k=1}^\infty \pi_{n+k} (q^k-1) \frac{\partial}{\partial p_k}
+ \cdots \right)~q^{n-1}\frac{\partial}{\partial \pi_n}\right)^2: ( \pi_m \pi_k p_\ell) \CR
&= q^{m+k-2}\pi_m \pi_k p_\ell 
+\sum_{r=1}^\infty   \sum_{s=1}^\infty \pi_r \pi_{s+j} \sum_{j=1}^\infty (q^j-1) \frac{\partial}{\partial p_j}
\left(q^{r+s-2} \frac{\partial}{\partial \pi_r} \frac{\partial}{\partial \pi_s}\right) ( \pi_m \pi_k p_\ell)\CR
&= q^{m+k-2} \big(  \pi_m \pi_k p_\ell + (q^\ell-1) (\pi_{m+\ell} \pi_k + \pi_m \pi_{k+\ell} ) \big).
\end{align}

\begin{rmk}
If we introduce the fermionic creation and annihilation operators by
\begin{align}
\psi^{(+)}(z) &= \sum_{k=1}^\infty \psi_{-k+\frac{1}{2}} z^{k-1} =  \sum_{k=1}^\infty \pi_k  z^{k-1}, \\
\psi^{(-)}(z) &= \sum_{k=1}^\infty \psi_{k-\frac{1}{2}} z^{-k} =  \sum_{k=1}^\infty \frac{\partial}{\partial \pi_k}  q^k z^{-k}.
\end{align}
we can write
\begin{equation}
\mathsf{T}_q \psi^{(+)}(z) = \psi^{(+)}(qz) 
\exp \left( \sum_{n=1}^\infty (1-q^{-n}) \frac{\partial}{\partial p_n} z^{-n} \right) \mathsf{T}_q.
\end{equation}
\end{rmk}

\begin{rmk}
In \cite{Alarie-Vezina:2019ohz} it is claimed that $\mathsf{T}_q$ is a ring endomorphism.
However, Lemma \ref{power-sum} means it is only a linear endomorphism.
\end{rmk}

We can confirm the inversion formula \eqref{Tq-onMac} by using \eqref{Tq-action-1} and \eqref{Tq-action-2}.
Up to level two there are three super Macdonald polynomials that involve the fermionic power sums;
\begin{align*}
\mathsf{T}_q(\mathcal{M}_{(1, \frac{1}{2})}) 
&= \mathsf{T}_q(p_1 \pi_1 - \pi_2) \\
&= p_1 \pi_1 +(q-1) \pi_2 - q \pi_2 = \mathcal{M}_{(1, \frac{1}{2})}, \\
\mathsf{T}_q(\mathcal{M}_{(\frac{3}{2})}) 
&= \mathsf{T}_q \left(\frac{q(1-t)}{1-qt}p_1\pi_1 + \frac{1-q}{1-qt}\pi_2 \right) \\
&= \frac{q(1-t)}{1-qt}p_1\pi_1 + \frac{q(q-1)(1-t)}{1-qt} \pi_2 + \frac{q(1-q)}{1-qt}\pi_2 \\
&= \frac{q(1-t)}{1-qt}p_1\pi_1 + \frac{tq(1-q)}{1-qt}\pi_2 \\
&= q \left(\frac{1-t}{1-qt}p_1\pi_1 + \frac{t(1-q)}{1-qt}\pi_2  \right)
= q \cdot \mathcal{M}_{(\frac{3}{2})}(q \to q^{-1}, t \to t^{-1}), \\
\mathsf{T}_q(\mathcal{M}_{(\frac{3}{2},\frac{1}{2})}) 
&= \mathsf{T}_q (\pi_2\pi_1) = q \mathcal{M}_{(\frac{3}{2},\frac{1}{2})}.
\end{align*}

At level $\frac{5}{2}$ we have four super Macdonald polynomials;
\begin{align*}
\mathsf{T}_q(\mathcal{M}_{(1,1, \frac{1}{2})}) 
&= \mathsf{T}_q \left( \frac{p_1^2 \pi_1}{2} - \frac{p_2 \pi_1}{2} - p_1\pi_2 + \pi_3 \right) \\
&= \frac{1}{2} (p_1^2 \pi_1 + 2(q-1) p_1 \pi_2+ (q-1)^2 \pi_3 )  - \frac{1}{2} (p_2 \pi_1 + (q^2-1) \pi_3) \\
& \qquad - q(p_1\pi_2 + (q-1) \pi_3) + q^2 \pi_3 \\
&= \frac{p_1^2 \pi_1}{2} - \frac{p_2 \pi_1}{2} - p_1\pi_2 + \frac{1}{2}(q-1)^2 \pi_3 -\frac{1}{2}(q^2-1) \pi_3 + q\pi_3 \\
&= \mathcal{M}_{(1,1, \frac{1}{2})}.
\end{align*}
\begin{align*}
\mathsf{T}_q(\mathcal{M}_{(\frac{3}{2},1)} ) 
&= \mathsf{T}_q \left( \frac{q(1-t^2)}{2(1-qt^2)}(p_1^2 \pi_1 - p_2 \pi_1) 
+ \frac{1-q}{1- q t^2}( p_1 \pi_2 - \pi_3) \right) \\
&= \frac{q(1-t^2)}{2(1-qt^2)}(p_1^2 \pi_1 - p_2 \pi_1) 
+ \frac{q(q-1)(1-t^2)}{1-qt^2}(p_1 \pi_2 - \pi_3)
+ \frac{q(1-q)}{1- q t^2}( p_1 \pi_2 - \pi_3) \\
&=  \frac{q(1-t^2)}{2(1-qt^2)}(p_1^2 \pi_1 - p_2 \pi_1) 
+ \frac{qt^2(1-q)}{1-qt^2}(p_1 \pi_2 - \pi_3)\\
&= q \cdot \mathcal{M}_{(\frac{3}{2},1)}  (q \to q^{-1}, t \to t^{-1}).
\end{align*}
For the remaining super Macdonald polynomials, it is convenient to set 
\begin{align*}
A:&= (1-qt) \mathcal{M}_{(2, \frac{1}{2})} \\
&= \frac{(1+q)(1-t)}{2} p_1^2 \pi_1
+ \frac{(1-q)(1+t)}{2} p_2 \pi_1 - q(1-t)p_1 \pi_2 
- (1-q)\pi_3 , \\
B:&= (1-qt)(1-q^2t)\mathcal{M}_{(\frac{5}{2})} \\
&= \frac{q^2(1+q)(1-t)^2}{2}p_1^2 \pi_1
+ \frac{q^2(1-q)(1-t^2)}{2} p_2 \pi_1 \\
& \qquad \qquad + q(1-q^2)(1-t)p_1 \pi_2 + (1-q)^2(1+q)\pi_3.
\end{align*}
we have
\begin{align*}
\mathsf{T}_q A &= \frac{1}{2}(1+q)(1-t)(p_1^2 \pi_1 + 2(q-1)p_1 \pi_2 + (q-1)^2 \pi_3)
+ \frac{1}{2}(1-q)(1+t) (p_2 \pi_1 + (q^2-1)\pi_3) \\
& \qquad -q^2(1-t)(p_1 \pi_2 + (q-1) \pi_3) - q^2(1-q) \pi_3 \\
&= \frac{1}{2}(1+q)(1-t)p_1^2 \pi_1 + \frac{1}{2}(1-q)(1+t)p_2 \pi_1 
- (1-t) p_1 \pi_2 - t(1-q^2)(1-q) \pi_3 - tq^2(1-q) \pi_3 \\
&= \frac{1}{2}(1+q)(1-t)p_1^2 \pi_1 + \frac{1}{2}(1-q)(1+t)p_2 \pi_1 
- (1-t) p_1 \pi_2 - t(1-q) \pi_3,
\end{align*}
which means
\begin{equation}
\mathsf{T}_q(\mathcal{M}_{(2, \frac{1}{2})}) = \mathcal{M}_{(2, \frac{1}{2})}(q \to q^{-1}, t \to t^{-1}).
\end{equation}
Similarly we have
\begin{align*}
\mathsf{T}_q B &= \frac{q^2}{2}(1+q)(1-t)^2(p_1^2 \pi_1 + 2(q-1)p_1 \pi_2 + (q-1)^2 \pi_3)
+ \frac{q^2}{2}(1-q)(1-t^2) (p_2 \pi_1 + (q^2-1)\pi_3) \\
& \qquad + q^2(1-q^2)(1-t)(p_1 \pi_2 + (q-1) \pi_3) + q^2(1-q)^2(1+q) \pi_3 \\
&= \frac{q^2}{2}(1+q)(1-t)^2p_1^2 \pi_1 + \frac{q^2}{2}(1-q)(1-t^2) p_2 \pi_1 
+ tq^2 (1-t)(1-q^2) p_1 \pi_ 2 \\
& \qquad + tq^2 (1-q^2)(1-q)(t-1) \pi_3 + tq^2 (1-q^2)(1-q) \pi_3 \\
&= \frac{q^2}{2}(1+q)(1-t)^2p_1^2 \pi_1 + \frac{q^2}{2}(1-q)(1-t^2) p_2 \pi_1 \\
& \qquad + tq^2 (1-t)(1-q^2) p_1 \pi_ 2 + t^2q^2 (1-q^2)(1-q)\pi_3. 
\end{align*}
which means
\begin{equation}
\mathsf{T}_q(\mathcal{M}_{(\frac{5}{2})}) = q^2 \cdot \mathcal{M}_{(\frac{5}{2})}(q \to q^{-1}, t \to t^{-1}).
\end{equation}

Let us check the consistency with Lemma \ref{Q-inversion}, which tells
\begin{align}
& \mathcal{Q}_1^{(q^{-1}, t^{-1})} \sim E_{1,-1}^{(q^{-1}, t^{-1})}
= \sum_{k=1}^\infty \pi_k \widetilde{c}_{k-1}^\vee [\partial /\partial p], \qquad
\mathcal{Q}_2^{(q^{-1}, t^{-1})} \sim F_{1,0}^{(q^{-1}, t^{-1})}
= \sum_{k=1}^\infty  c_{k-1}^\vee [p] \frac{\partial}{\partial \pi_k}, \\
& \mathcal{Q}_4 \sim E_{1,0} = \pi_1, \qquad
\mathcal{Q}_3 \sim F_{1,2} = \sum_{n=1}^\infty \sum_{k=0}^\infty
c^\vee_{k+\ell}[p] \widetilde{c}^\vee_{k}[\partial/\partial p]
q^{n-1} \frac{\partial}{\partial \pi_n} + \hbox{higher order}.
\label{Q3-leading}
\end{align}
%where
%$$
%C_\ell[p, \partial/\partial p] = \sum_{k=0}^\infty c^\vee_{k+\ell}[p] \widetilde{c}^\vee_{k}[\partial/\partial p].
%$$
It is easy to see 
\begin{equation}
\mathsf{T}_q \mathcal{Q}_4 \mathsf{T}_q^{-1} = \mathsf{T}_q \pi_1 \mathsf{T}_q^{-1} 
= \sum_{k=0}^\infty \pi_{k+1} \widetilde{c}^\vee_{k}[\partial/\partial p] = \mathcal{Q}_1^{(q^{-1}, t^{-1})}.
\end{equation}
{\it On the fermoin number one subspace}, we can compute
\begin{align}
\mathsf{T}_q^{-1}\mathcal{Q}_2^{(q^{-1}, t^{-1})} \mathsf{T}_q
&= \mathsf{T}_{q^{-1}}^{(0)} \left(\sum_{\ell=1}^\infty  
c_{\ell-1}^\vee [p] \frac{\partial}{\partial \pi_\ell}\right) \mathsf{T}_{q}^{(1)} \CR
&= \sum_{\ell=1}^\infty c_{\ell-1}^\vee [p] \frac{\partial}{\partial \pi_\ell}
\left(\sum_{n=1}^\infty \sum_{k=0}^\infty \widetilde{c}_k^\vee [ \partial/\partial p] \pi_{n+k} q^{n-1} 
\frac{\partial}{\partial \pi_n} \right) \CR
&= \sum_{n=1}^\infty \sum_{k=0}^\infty c_{k+n-1}^\vee [p]\widetilde{c}_k^\vee [ \partial/\partial p]
q^{n-1} \frac{\partial}{\partial \pi_n},
\end{align}
which agrees with \eqref{Q3-leading}.

%%%%%%%%%%%%%%%%%%%%%%%%%%%%%%%%%%%%%%%%%%%%%%%%%%%%%%%%%%%


\begin{thebibliography}{99}


%\cite{Alarie-Vezina:2019ohz}
\bibitem{Alarie-Vezina:2019ohz}
L.~Alarie-V\'ezina, O.~Blondeau-Fournier, P.~Desrosiers, L.~Lapointe and P.~Mathieu,
``Symmetric functions in superspace: a compendium of results and open problems (including a SageMath worksheet),''
[arXiv:1903.07777 [math-ph]].


%\cite{Awata:2011ce}
\bibitem{Awata:2011ce}
H.~Awata, B.~Feigin and J.~Shiraishi,
``Quantum Algebraic Approach to Refined Topological Vertex,''
JHEP \textbf{03} (2012), 041
%doi:10.1007/JHEP03(2012)041
[arXiv:1112.6074 [hep-th]].


%\cite{Awata:1994xd}
\bibitem{Awata:1994xd}
H.~Awata, Y.~Matsuo, S.~Odake and J.~Shiraishi,
``Collective field theory, Calogero-Sutherland model and generalized matrix models,''
Phys. Lett. B \textbf{347} (1995), 49-55
%doi:10.1016/0370-2693(95)00055-P
[arXiv:hep-th/9411053 [hep-th]].

%\cite{Awata:1995zk}
\bibitem{Awata:1995zk}
H.~Awata, H.~Kubo, S.~Odake and J.~Shiraishi,
``Quantum W(N) algebras and Macdonald polynomials,''
Commun. Math. Phys. \textbf{179} (1996), 401-416
%doi:10.1007/BF02102595
[arXiv:q-alg/9508011 [math.QA]].


%\cite{Blondeau-Fournier:2011sft}
\bibitem{Blondeau-Fournier:2011sft}
O.~Blondeau-Fournier, P.~Desrosiers, L.~Lapointe and P.~Mathieu,
``Macdonald polynomials in superspace: conjectural definition and positivity conjectures,''
Lett. Math. Phys. \textbf{101} (2012), 27-47
%doi:10.1007/s11005-011-0542-5
[arXiv:1112.5188 [math-ph]].


%\cite{Blondeau-Fournier:2012exj}
\bibitem{Blondeau-Fournier:2012exj}
O.~Blondeau-Fournier, P.~Desrosiers, L.~Lapointe and P.~Mathieu,
``Macdonald polynomials in superspace as eigenfunctions of commuting operators,''
J. Comb. \textbf{3} (2012), 495-561
%doi:10.4310/JOC.2012.v3.n3.a8
[arXiv:1202.3922 [math-ph]].


%\cite{Cheewaphutthisakun:2025zoc}
\bibitem{Cheewaphutthisakun:2025zoc}
P.~Cheewaphutthisakun, J.~Shiraishi and K.~Wiboonton,
``Quantum corner VOA and the super Macdonald polynomials,''
JHEP \textbf{10} (2025), 064
%doi:10.1007/JHEP10(2025)064
[arXiv:2504.17326 [hep-th]].


%\cite{Cheewaphutthisakun:2025ovm}
\bibitem{Cheewaphutthisakun:2025ovm}
P.~Cheewaphutthisakun, J.~Shiraishi and K.~Wiboonton,
``Quantum Corner Polynomials: A Generalization of Super Macdonald Polynomials and Their VOA Correspondence,''
[arXiv:2508.12267 [hep-th]].


\bibitem{FFJMM}
B.~Feigin, E.~Feigin, M.~Jimbo, T.~Miwa and E.~Mukhin,
``Quantum continuous $\mathfrak{gl}_\infty$: Semi-infinite construction of representations,''
Kyoto J. Math. {\bf 51} (2011), 337--364,
[arXiv:1002.3100 [math.QA]]


%\cite{FT}
\bibitem{FT}
B.~Feigin and A.~I.~Tsymbaliuk,
``Equivariant K-theory of Hilbert schemes via shuffle algebra,''
Kyoto J. Math. {\bf 51} (2011), 831--854,
[arXiv:0904.1679 [math.RT]].

%\cite{Filoche:2026xix}
\bibitem{Filoche:2026xix}
B.~Filoche, S.~Hohenegger and T.~Kimura,
``Wall-crossing of Instantons on the Blow-up,''
[arXiv:2604.20674 [hep-th]].


%\cite{Galakhov:2021xum}
\bibitem{Galakhov:2021xum}
D.~Galakhov, W.~Li and M.~Yamazaki,
``Shifted quiver Yangians and representations from BPS crystals,''
JHEP \textbf{08} (2021), 146
%doi:10.1007/JHEP08(2021)146
[arXiv:2106.01230 [hep-th]].


%\cite{Galakhov:2021vbo}
\bibitem{Galakhov:2021vbo}
D.~Galakhov, W.~Li and M.~Yamazaki,
``Toroidal and elliptic quiver BPS algebras and beyond,''
JHEP \textbf{02} (2022), 024
%doi:10.1007/JHEP02(2022)024
[arXiv:2108.10286 [hep-th]].


%\cite{Galakhov:2024cry}
\bibitem{Galakhov:2024cry}
D.~Galakhov, A.~Morozov and N.~Tselousov,
``Macdonald polynomials for super-partitions,''
Phys. Lett. B \textbf{856} (2024), 138911
%doi:10.1016/j.physletb.2024.138911
[arXiv:2407.03301 [hep-th]].

%\cite{Galakhov:2024zqn}
\bibitem{Galakhov:2024zqn}
D.~Galakhov, A.~Morozov and N.~Tselousov,
``Supersymmetric polynomials and algebro-combinatorial duality,''
SciPost Phys. \textbf{17} (2024), 119
%doi:10.21468/SciPostPhys.17.4.119
[arXiv:2407.04810 [hep-th]].

%\cite{Galakhov:2025phf}
\bibitem{Galakhov:2025phf}
D.~Galakhov, A.~Morozov and N.~Tselousov,
``Super-Hamiltonians for super-Macdonald polynomials,''
Phys. Lett. B \textbf{865} (2025), 139481
[arXiv:2501.14714 [hep-th]].


%\cite{Kanno:2025ifd}
\bibitem{Kanno:2025ifd}
H.~Kanno, R.~Ohkawa and J.~Shiraishi,
``Super Macdonald polynomials and BPS state counting on the blow-up,''
Lett. Math. Phys. \textbf{115} (2025) no.6, 138
%doi:10.1007/s11005-025-02014-y
[arXiv:2506.01415 [hep-th]].


\bibitem{Mac1988}
I. G. Macdonald, ``A new class of symmetric functions,''
S\'eminaire Lotharingien de Combinatoire 20, B20a-41 (1988).


\bibitem{MacD} 
I. G. Macdonald, 
{\sl Symmetric functions and Hall polynomials}, 2nd ed., Oxford Mathematical
Monographs, Oxford University Press (1995).


%\cite{Noshita:2021dgj}
\bibitem{Noshita:2021dgj}
G.~Noshita and A.~Watanabe,
``Shifted quiver quantum toroidal algebra and subcrystal representations,''
JHEP \textbf{05} (2022), 122
%doi:10.1007/JHEP05(2022)122
[arXiv:2109.02045 [hep-th]].


\bibitem{Ruijs}
S.N.M.~Ruijsenaars,
``Complete integrability of relativistic Calogero-Moser systems and 
elliptic function identities,'' Commun.\ Math.\ Phys.\ \textbf{110} (1987), 191--213.


%\cite{Sergeev-Veselov}
\bibitem{Sergeev-Veselov1}
A.N.~Sergeev and A.P.~Veselov, 
``Deformed quantum Calogero-Moser systems and Lie superalgebras,''
Commun. Math. Phys. \textbf{245} (2004), 249-278
[arXiv:math-ph/0303025].


%\cite{Sergeev-Veselov}
\bibitem{Sergeev-Veselov2}
A.N.~Sergeev and A.P.~Veselov, 
``Deformed Macdonald-Ruijsenaars operators and super Macdonald polynomials,''
Commun. Math. Phys. \textbf{288} (2009), 653-675
[arXiv:0707.3129 [math.QA]].
	
\end{thebibliography}
\end{document}